\documentclass{article}
\usepackage{amsmath, amssymb, amsthm, tocloft, enumerate, tikz, algorithmicx,
            algpseudocode, mathrsfs, tkz-graph, color, comment, xspace, cite} 
\usepackage[top=50pt,bottom=60pt,left=78pt,right=76pt]{geometry}
\usepackage[ruled]{algorithm}
\usepackage[colorlinks]{hyperref}
\hypersetup{linkcolor=blue, urlcolor=blue, citecolor=red}

\newcommand{\vertex}[3]{\node [vertex] (#1) at (#2, #3 * 1.7) {};
\node at (#2, #3 * 1.7) {$\alpha_{#1}$};}
\newcommand{\edge}[5]{\draw[->] (#2) edge [edge, color=#4, #5] node[midway,
fill=white,inner sep=1]{$x_{#1}$} (#3);}
\newcommand{\edgeone}[3]{\edge{1}{#1}{#2}{black}{#3}}
\newcommand{\edgetwo}[3]{\edge{2}{#1}{#2}{red}{#3}}
\newcommand{\edgethr}[3]{\edge{3}{#1}{#2}{blue}{#3}}
\newcommand{\edgefou}[3]{\edge{4}{#1}{#2}{green}{#3}}

\algrenewcommand{\algorithmiccomment}[1]{\hfill[{\it #1}]}
\newcommand{\libsemigroups}{{\sc libsemigroups}\xspace}

\newcommand{\N}{\mathbb{N}}
\newcommand{\M}{\mathcal{M}}
\newcommand{\stab}{\operatorname{Stab}}
\newcommand{\sym}{\operatorname{Sym}}
\newcommand{\id}{\operatorname{id}}
\newcommand{\im}{\operatorname{im}}
\newcommand{\proj}{\operatorname{Proj}}

\renewcommand{\ker}{\operatorname{ker}}
\newcommand{\dom}{\operatorname{dom}}
\renewcommand{\P}{\mathcal{P}}
\renewcommand{\L}{\mathscr{L}}
\renewcommand{\H}{\mathscr{H}}
\newcommand{\D}{\mathscr{D}}
\newcommand{\R}{\mathscr{R}}
\newcommand{\J}{\mathscr{J}}
\newcommand{\K}{\mathcal{K}}
\newcommand{\ov}[1]{\overline{#1}}
\newcommand{\set}[2]{\{#1:#2\}}
\newcommand{\bn}{\mathbf{n}}
\newcommand{\genset}[1]{\langle#1\rangle}
\newcommand{\GAP}{{\sc GAP}\xspace}
\newcommand{\Orb}{{\sc Orb} }
\newcommand{\Semigroups}{{\sc Semigroups} }

\renewcommand{\to}{\longrightarrow}

\newcommand{\AND}{\quad\text{and}\quad}
\renewcommand{\iff}{\quad\text{if and only if}\quad}

\renewcommand{\l}[1]{\mu_{#1}}
\renewcommand{\r}[1]{\nu_{#1}}

\numberwithin{equation}{section} 
\newtheorem{thm}[equation]{Theorem}

\newtheorem{prop}[equation]{Proposition}
\newtheorem{lem}[equation]{Lemma}
\newtheorem{cor}[equation]{Corollary}
\newtheorem{defn}[equation]{Definition}
\newtheorem{exam}[equation]{Example}
\newtheorem*{assumption}{Assumptions}

\newenvironment{ass}{\begin{assumption}\rm}{\end{assumption}}

\begin{document}
\title{Computing finite semigroups}
\author{J. East, A. Egri-Nagy, J. D. Mitchell, and Y. P\'eresse}
\maketitle

\begin{abstract}
  Using a variant of Schreier's Theorem, and the theory of Green's relations,
  we show how to reduce the computation of an arbitrary subsemigroup of a
  finite regular semigroup to that of certain associated subgroups. Examples of
  semigroups to which these results apply include many important classes:
  transformation semigroups, partial permutation semigroups and inverse
  semigroups, partition monoids, matrix semigroups, and subsemigroups of finite
  regular Rees matrix and $0$-matrix semigroups over groups.  For any
  subsemigroup of such a semigroup, it is possible to, among other things,
  efficiently compute its size and Green's relations, test membership,
  factorize elements over the generators, find the semigroup generated by the
  given subsemigroup and any collection of additional elements, calculate the
  partial order of the $\D$-classes, test regularity, and determine the
  idempotents. This is achieved by representing the given subsemigroup without
  exhaustively enumerating its elements. It is also possible to compute the
  Green's classes of an element of such a subsemigroup without determining the
  global structure of the semigroup.  
\end{abstract}

\tableofcontents
\listofalgorithms
\listoffigures


\section{Introduction}

A \textit{semigroup} is a set with an associative binary operation.  There are
many articles in the literature concerned with the idea of investigating
semigroups using a computer; early examples are \cite{Cannon1969aa,
Cannon1971aa, Forsythe1955aa, Jurgensen1977aa, Perrot1972aa}. There are also
several examples of software packages specifically for computing with
semigroups, such as {\sf AUTOMATE} \cite{Champarnaud1991aa}, {\sf Monoid}
\cite{Linton1997aa}, {\sf SgpWin} \cite{McAlister2006aa}, {\sf Semigroupe}
\cite{Pin2009aa}, and more general computational algebra systems, such as {\sf
Magma} \cite{Bosma1997aa}, \GAP \cite{GAP4}, and {\sf Sage} \cite{Stein2013aa},
with some functionality relating to semigroups; see also \cite{Cousineau1973aa}.

Semigroups are commonly represented either by presentations (abstract generators
and defining relations) or by a generating set consisting of a specific type of
element, such as transformations, matrices, or binary relations. In this paper,
we are solely concerned with semigroups represented by a generating set. 

Computing with semigroups defined by generators or with finitely presented
semigroups is \textbf{hard}; it is shown in \cite{Beaudry1988aa} that testing
membership in a finite commutative transformation semigroup is NP-complete, and
it is well-known that determining any `sensible' property of a finitely
presented semigroup is undecidable by the famous results of Post
\cite{Post1947aa} and Markov \cite{Markov1947aa}. However, in spite of the fact
that the general case is hard, it is still possible to compute with semigroups
efficiently in many particular instances.  Perhaps more importantly, it is
possible to perform calculations using a computer that it would be impossible
(several times over) to do by hand.  

Algorithms, and their implementations, for computing semigroups defined by a
generating set fall into two classes: those that exhaustively enumerate the
elements, and those that do not.  Examples of the first type are the algorithms
described in \cite{Froidure1997aa} and implemented in {\sf Semigroupe}
\cite{Pin2009aa}, and those in {\sf SgpWin} \cite{McAlister2006aa}.
Exhaustively enumerating and storing the elements of a semigroup quickly
becomes impractical.  To illustrate, a \textit{transformation} is a function from
the set $\{1,\dots, n\}$ to itself for some $n\in \N$. A semigroup whose
elements are transformations and whose operation is composition of functions is
called a \textit{transformation semigroup}.   For example, if each of the
$9^9=387420489$ transformations on a $9$-element set is stored as a tuple of
$9$ integers from $1$ to $9$ and $\alpha$ is the number of bits required to
store such an integer, then 
$$9^{9+1}\cdot \alpha=3486784401\cdot\alpha \text{ bits}$$ 
are required to store these transformations. In \GAP, for example, such an
integer requires 16-bits, and so approximately $6$ gigabytes of memory would be
required in this case. Therefore if we want to exhaustively enumerate a
semigroup, then we must be happy to do so with relatively small semigroups.
Exhaustive algorithms have the advantage that they are relatively
straightforward to apply; if the multiplication of a class of semigroups can be
defined to a computer, then these algorithms can be applied.  For
example, in  {\sf Semigroupe} \cite{Pin2009aa} it is possible to compute with
semigroups of transformations, partial transformations, and
several types of matrix semigroups including boolean matrices. 

Non-exhaustive algorithms are described in
\cite{Konieczny1994aa,Lallement1990aa,Linton1998aa,Linton2002aa}, and the
latter were implemented in the {\sf Monoid} package \cite{Linton1997aa} for
\GAP 3 and its later incarnations in \GAP 4. In examples where it is not
possible to store the elements, these methods can be used to determine
structural information about a semigroup, such as its size and Green's
relations (see Section~\ref{section-prerequisites} for the relevant
definitions). In many examples, the non-exhaustive algorithms have better
performance than their exhaustive analogues. However, on the down side, the
non-exhaustive algorithms described in
\cite{Lallement1990aa,Linton1998aa,Linton2002aa} only apply to transformation
semigroups. The methods in \cite{Konieczny1994aa} are analogues of the methods
in \cite{Lallement1990aa} in the context of semigroups of binary relations;
but an implementation does not appear to be readily available.

To one degree or another, the articles
\cite{Beaudry1988aa,Konieczny1994aa,Lallement1990aa,Linton1998aa,Linton2002aa}
use variants of Schreier's Theorem \cite[Theorem 2.57]{B.-Eick2004aa} and the
theory of Green's relations to reduce the computation of the semigroup to that
of its Sch\"utzenberger groups.  It is then possible to use the well-developed
and efficient algorithms from Computational Group Theory
\cite{B.-Eick2004aa,Seress2003ab,Sims1970aa}, stemming from the Schreier-Sims
Algorithm, to compute with these subgroups.  In this paper, we go one step
further by giving a computational paradigm for arbitrary subsemigroups of
finite regular semigroups.  Semigroups to which the paradigm can be efficiently
applied include many of the most important classes: transformation semigroups,
partial permutation semigroups
and inverse semigroups, partition monoids, matrix semigroups, and subsemigroups
of finite regular Rees matrix and $0$-matrix semigroups. We generalise and
improve the central notions in \cite{Lallement1990aa,Linton1998aa,Linton2002aa}
from transformation semigroups to arbitrary subsemigroups of an arbitrary finite
regular semigroup.  For such a subsemigroup, it is possible to efficiently
compute its size and Green's classes, test membership, factorize elements over
the generators, find the semigroup generated by the given subsemigroup and any
collection of additional elements, calculate the partial order of the
$\D$-classes, test regularity, and determine the idempotents. This is achieved
by representing the given subsemigroup without exhaustively enumerating its
elements. In particular, our methods can be used to determine properties of
semigroups, where it is not possible to store every element of that semigroup.
It is also possible to compute the Green's classes of an element of such a
subsemigroup without determining the global structure of the
semigroup. 

Although not described here, it is also possible to use the data structures
provided to find the group of units (if it exists), minimal ideal, find a small
generating set, and test if a semigroup satisfies various properties such as
being simple, completely regular, Clifford and so on.  The algorithms described
in this paper are implemented in their full generality in the \GAP \cite{GAP4}
package \Semigroups \cite{Mitchell2016aa}, which is open source software.

The analogue of Cayley's Theorem for semigroups states that every finite
semigroup is isomorphic to a transformation semigroup.  Consequently, it could
be argued that it is sufficient to have computational tools available for
transformation semigroups only. An analogous argument could be made for
arbitrary groups with respect to permutation groups, but developments in
computational group theory suggest otherwise.  For example, the Matrix Group
Recognition Project has produced efficient algorithms for computing with groups
of matrices over finite fields; \cite{MGRP} and \cite{Leedham-Green2001aa}.
Similarly, for some classes of semigroups, such as subsemigroups of the
partition monoid or a Rees matrix semigroup, the only known faithful
transformation representations are those that act on the elements of the
semigroup itself. Hence, it is necessary in such examples to exhaustively
compute the elements of the semigroup before a transformation representation is
available.  At this point any transformation representation, and the
non-exhaustive methods that could be applied to it, are redundant. Therefore,
to compute with such semigroups without exhaustively enumerating them, it is
necessary to have non-exhaustive algorithms that apply directly to the given
semigroups and, in particular, do not require a transformation representation.
It is such algorithms that we present in this paper. 

When considering matrix semigroups, it is straightforward to determine a
transformation representation. Even in these cases, it is sometimes preferable
to compute in the native matrix representation: in particular, where there are
methods for matrix groups that are more efficient than computing a
permutation representation.

The algorithms in this paper only apply to subsemigroups of a regular
semigroup.  However, it would be possible to modify several of these
algorithms, including the main one (Algorithm~\ref{algorithm-enumerate}), so
that they apply to subsemigroups $S$ of a non-regular semigroup $U$. In
particular, if it were possible to determine whether a given element was
regular or not in $U$, then we could use the data structure for $\R$-classes
described in Section~\ref{section-algorithms} for the regular elements, and
perform an exhaustive enumeration of $\R$-classes of non-regular elements
in $U$. Or alternatively, it might be possible to use a combination of the
approaches described in this paper and those in \cite{Konieczny1994aa,
Lallement1990aa}.  Such an approach would be possible with, say, semigroups of
binary relations.  It is possible to check that a binary relation is regular as
an element of the semigroup of all binary relations in polynomial time; see
\cite{Fitz-Gerald1977aa, Schein1976aa}. However, we did not yet follow this
approach either in the paper or in the \Semigroups package since it is
relatively easy to find a transformation representation of a semigroup of
binary relations. 

This paper is organised as follows.  In Section~\ref{section-prerequisites}, we
recall some well-known mathematical notions, and establish some notation that
is required in subsequent sections.  In Section~\ref{section-paradigm}, we
provide the mathematical basis that proves the validity of the algorithms
presented in Sections~\ref{section-algorithms}.  In
Section~\ref{section-special-cases}, we show that transformation semigroups,
partial permutation semigroups and inverse semigroups, partition monoids,
semigroups of matrices over a finite field, and subsemigroups of finite regular
Rees matrix or $0$-matrix semigroups satisfy the conditions from
Section~\ref{section-paradigm} and, hence, belong to the class of semigroups
with which we can compute efficiently.  Detailed algorithms are presented as
pseudocode in Section~\ref{section-algorithms}, including some remarks about
how the main algorithms presented can be simplified in the case of regular and
inverse semigroups.  In Section~\ref{section-examples} we give several detailed
examples. The final section, Section~\ref{section-benchmarks} is devoted to a
discussion of the performance of the algorithms presented herein, and a
comparison with the algorithm described in~\cite{Froidure1997aa}.


\section{Mathematical prerequisites}\label{section-prerequisites}

A semigroup $S$ is \textit{regular} if for every $x\in S$ there exists $x'\in S$
such that $xx'x=x$. A semigroup $S$ is a \textit{monoid} if it has an identity
element, i.e.\ an element $e\in S$ such that $es=se=s$ for all $s\in S$. 
If $S$ is a semigroup, we write $S^1$ for the monoid
obtained by adjoining an identity $1_S\not\in S$ to $S$. Note that with this
definition $S ^ 1$ always contains one more element ($1_S$) than $S$ 
regardless of whether or not $S$ has an existing identity element. 

For any set $\Omega$, the set $\Omega^\Omega$ of transformations of $\Omega$ is
a semigroup under composition of functions, known as the \textit{full
transformation monoid on $\Omega$}.  The identity element of $\Omega^\Omega$
is the identity function on $\Omega$, which will be denoted $\id_\Omega$.  We
denote the full transformation monoid on the finite set
$\{1,\ldots, n\}$ by $T_n$.  Throughout this article, we will write functions
to the right of their arguments and compose functions from left to right.

A \textit{subsemigroup} $T$ of a semigroup $S$ is just a subset which is a
semigroup under the same operation as $S$; we denote this by $T\leq S$. 
If $X$ is a subset of a semigroup $S$, then the least subsemigroup of $S$
containing $X$ is denoted by $\genset{X}$; this is also referred to as the
\textit{subsemigroup generated by $X$}.  We denote the cardinality of a set $X$ by
$|X|$.

Let $S$ be a semigroup and let $x,y\in S$ be arbitrary. We say that $x$ and $y$
are $\L$-related if the principal left ideals generated by $x$ and $y$ in $S$
are equal; in other words, $S^1x=S^1y$. We write $x\L y$ to denote that $x$ and
$y$ are $\L$-related. In Section~\ref{section-paradigm}, we often want to
distinguish between the cases when elements are $\L$-related in a semigroup $U$
or in a subsemigroup $S$ of $U$.  We write $x\L^Sy$ or $x\L^Uy$ to
differentiate these cases.

Green's $\R$-relation is defined dually to Green's $\L$-relation; Green's
$\H$-relation is the meet, in the lattice of equivalence relations on $S$, of
$\L$ and $\R$; and $\D$ is the join. We will refer to the equivalence classes
as $\mathscr{K}$-classes where $\mathscr{K}$ is any of $\R$, $\L$, $\H$, or
$\D$, and the $\mathscr{K}$-class of $x\in S$ will be denoted by $K_x$, or
$K_x^S$ if it is necessary to explicitly refer to the semigroup where the
relation is defined. We denote the set of $\mathscr{K}$-classes of a semigroup
$S$ by $S/\mathscr{K}$.

If $S$ is a semigroup and $x, y\in S$, then $x$ and $y$ are $\J$-related if
they generate the same principal two-sided ideal of $S$, i.e.\ $S ^ 1 x S ^ 1 =
S ^ 1 y S ^ 1$.  In a finite semigroup, $x\J y$ if and only if $x\D y$.
Containment of principal ideals induces a partial order on the $\J$-classes of
$S$ (and hence $\D$-classes if $S$ is finite), sometimes denoted $\leq_{\J}$;
that is, $J_x\leq_{\J}J_y$ if and only if $S^1xS^1\subseteq S^1yS^1$. 


\begin{prop}[cf. Proposition A.1.16 in \cite{Rhodes2009aa}]
  Let $U$ be a semigroup and let $S$ be a subsemigroup of $U$.  Suppose that
  $x$ and $y$ are regular elements of $S$. Then $x\mathscr{K}^Uy$ if and only
  if $x\mathscr{K}^Sy$, where $\mathscr{K}$ is any of $\R$, $\L$ or $\H$.
\end{prop}


Note that the previous result does not necessarily hold for $\mathscr{K}=\D$.

Let $S$ be a semigroup and let $\Omega$ be a set. A function $\Psi:\Omega\times
S^1\to \Omega$ is a \textit{right action} of $S$ on $\Omega$ if  
\begin{itemize}
  \item $((\alpha, s)\Psi, t)\Psi=(\alpha, st)\Psi$;
  \item $(\alpha, 1_S)\Psi=\alpha$.
\end{itemize}
For the sake of brevity, we will write $\alpha\cdot s$ instead of $(\alpha,
s)\Psi$, and we will say that $S$ acts on $\Omega$ on the right.  \textit{Left
actions} are defined analogously, and we write $s\cdot \alpha$ in this case,
and say $S$ acts on $\Omega$ on the left.  The \textit{kernel} of a function
$f: X\to Y$, where $X$ and $Y$ are any sets, is the equivalence relation
$\set{(x,y)\in X\times X}{(x)f = (y)f}$.  A right action of a semigroup $S$ on
a set $\Omega$ induces a homomorphism from $S$ to the full transformation
monoid on $\Omega$ defined by mapping $s\in S$ to the transformation defined by
$$\alpha\mapsto \alpha \cdot s\quad\text{for all}\quad \alpha\in \Omega.$$ An
action is called \textit{faithful} if the induced homomorphism is injective.
The \textit{kernel of a right action} of a semigroup $S$ on a set $\Omega$ is
just the kernel of the induced homomorphism, i.e.~the equivalence relation
$\set{(s,t)\in S\times S}{\alpha\cdot s = \alpha\cdot t\ (\forall \alpha\in
\Omega)}$. The kernel of a left action is defined analogously.

If $S$ acts on the sets $\Omega$ and $\Omega'$ on the right, then we say that
$\lambda: \Omega\to \Omega'$ is a \textit{homomorphism of right actions} if
$(\alpha \cdot s)\lambda=(\alpha)\lambda\cdot s$ for all $\alpha \in \Omega$
and $s\in S^1$. \textit{Homomorphisms of left actions} are defined analogously.
An \textit{isomorphism} of (left or right) actions is a bijective homomorphism
of (left or right) actions. 

If $\Omega$ is a set, then we denote the set of subsets of $\Omega$ by
$\P(\Omega)$.  If $S$ is a semigroup acting on the right on $\Omega$, then the
action of $S$ on $\Omega$ induces a natural action of $S$ on $\P(\Omega)$,
which we write as:
\begin{equation}\label{equation-induced-action}
  \Sigma\cdot s=\{\alpha\cdot s:\alpha\in \Sigma\}\qquad \text{for each}\
  \Sigma\subseteq \Omega.
\end{equation}
We will denote the function from $\Sigma$ to $\Sigma\cdot s$ defined by
$\alpha\mapsto \alpha\cdot s$ by $s|_{\Sigma}$.  We define the
\textit{stabiliser} of $\Sigma$ under $S$ to be
$$\stab_S(\Sigma)=\set{s\in S^1}{\Sigma\cdot s=\Sigma}.$$
The quotient of the stabiliser by the kernel of its action on $\Sigma$,
i.e.~the congruence 
$$\set{(s,t)}{s,t\in\stab_S(\Sigma),\ s|_\Sigma=t|_\Sigma},$$
is isomorphic to
$$S_{\Sigma}=\set{s|_{\Sigma}}{s\in \stab_S(\Sigma)}$$
which in the case that $\Sigma$ is finite, is a subgroup of the symmetric group
$\sym(\Sigma)$ on $\Sigma$. The stabiliser $S_{\Sigma}$ can also be seen as a
subgroup of $\sym(\Omega)$ by extending the action of its elements so that they
fix $\Omega\setminus \Sigma$ pointwise.  It is immediate that $s|_{\Sigma}\cdot
t|_{\Sigma}=(st)|_{\Sigma}$ for all $s,t\in \stab_S(\Sigma)$. 

If $S$ acts on $\Omega$, the \textit{strongly connected component} (usually
abbreviated to s.c.c.) of an element $\alpha\in\Omega$ is the set of all
$\beta\in\Omega$ such that $\beta=\alpha\cdot s$ and $\alpha=\beta\cdot t$ for
some $s,t\in S^1$. We write $\alpha\sim \beta$ if $\alpha$ and $\beta$ belong to
the same s.c.c.~and the action is clear from the context. 

If $\alpha\in\Omega$, then 
$$\alpha \cdot S ^ 1 = \set{\alpha\cdot s}{s\in S ^ 1},$$
is a disjoint union of strongly connected components of the action of $S$. Note
that $\alpha\cdot S ^ 1$ might consist of more than one strongly connected
component.  If $S$ is a group, then $\alpha \cdot S ^ 1$ has only one
strongly connected component, which is usually called the \textit{orbit} of
$\alpha$ under $S$.  


\begin{prop}\label{prop-technical}
  Let $S=\genset{X}$ be a semigroup that acts on a finite set $\Omega$ on the
  right and let $\Sigma_1, \ldots, \Sigma_n\subseteq \Omega$ be the elements of
  a strongly connected component of the action of $S$ on $\P(\Omega)$. Then the
  following hold:
  \begin{enumerate}[\rm (a)]

    \item if $\Sigma_1\cdot u_i= \Sigma_i$ for some $u_i\in S^1$, then
      there exists $\ov{u_i}\in S^1$ such that $\Sigma_i\cdot \ov{u_i}= \Sigma_1$,
      $(u_i\ov{u_i})|_{\Sigma_1}=\id_{\Sigma_1}$, and
      $(\ov{u_i}u_i)|_{\Sigma_i}=\id_{\Sigma_i}$;

    \item $S_{\Sigma_i}$ and $S_{\Sigma_j}$ are conjugate subgroups of
      $\sym(\Omega)$ for all $i,j\in \{1,\ldots, n\}$;

    \item if $u_1=\ov{u_1}=1_S$ and $u_i, \ov{u_i}\in S$ are as in part (a)
      for $i>1$, then $S_{\Sigma_1}$ is generated by
      $$\{(u_ix\ov{u_j})|_{\Sigma_1}:1\leq i,j\leq n,\ x\in X,\ \Sigma_i\cdot
      x=\Sigma_j\}.$$
  \end{enumerate}
\end{prop}
\begin{proof} 
  Let $\theta: S\to \Omega^\Omega$ be the homomorphism induced by the action of
  $S$ on $\Omega$.  Then $(S)\theta$ is a transformation semigroup and the actions
  of $(S)\theta$ and $S$ on $\Omega$ are equal. Hence (a), (b), and (c) are just
  Lemma 2.2 and Theorems 2.1 and 2.3 in~\cite{Linton1998aa}, respectively.
\end{proof}\smallskip 


We will refer to the generators of $S_{\Sigma_1}$ in
Proposition~\ref{prop-technical}(c) as \textit{Schreier generators} of
$S_{\Sigma_1}$, due to the similarity of this proposition and Schreier's
Theorem \cite[Theorem 2.57]{B.-Eick2004aa}.  If $S$ is a semigroup acting on a
set $\Omega$, if $\Sigma, \Gamma\subseteq \Omega$ are such that $\Sigma\sim
\Gamma$ and if $u\in S$ is any element such that $\Sigma \cdot u = \Gamma$,
then we write $\ov{u}$ to denote an element of $S$ with the properties in 
Proposition~\ref{prop-technical}(a).

The analogous definitions can be made, and an analogous version of
Proposition~\ref{prop-technical} holds for left actions.  In the next section
there are several propositions involving left and right actions in the same
statement, and so we define the following notation for left actions.  If $S$ is
a semigroup acting on the left on a set $\Omega$ and $\Sigma\subseteq \Omega$,
then the induced left action of $S$ on $\P(\Omega)$ is defined analogously to
\eqref{equation-induced-action}; the function $\alpha\mapsto s\cdot \alpha$ is
denoted by ${}_{\Sigma}|s$; and we define
$${}_{\Sigma}S=\set{{}_{\Sigma}|s}{s\cdot \Sigma=\Sigma}$$
and in the case that $\Omega$ is finite, ${}_{\Sigma}S\leq \sym(\Omega)$.


\section{From transformation semigroups to arbitrary regular semigroups}
\label{section-paradigm}

In this section, we will generalise the results of \cite{Linton1998aa} from
transformation semigroups to subsemigroups of an arbitrary finite regular
semigroup. 

Suppose that $U$ is a regular semigroup whose Green's structure is known in
the sense that it is possible to efficiently verify if two elements are
$\L^U$-, $\R^U$-, $\H^U$-, or $\D^U$-related.  For instance, $U$ might be the
full transformation monoid or the symmetric inverse monoid on a finite set; the
Green's relations of which are described in
Propositions~\ref{prop-greens-full-trans}
and~\ref{prop-greens-symmetric-inverse}, respectively.  In this paper, the
central notion is to use our knowledge of the Green's structure of $U$ to
produce algorithms to compute any subsemigroup $S$ of $U$ specified by a
generating set.  Roughly speaking, the subsemigroup $S$ can be decomposed into
its $\R$-classes, and an $\R$-class can be decomposed into the stabiliser $S_L$
(under right multiplication on $\P(U)$) of an $\L^U$-class $L$ and the
s.c.c.~of $L$ under the action of $S$ on the $U/\L^U$.  Decomposing $S$ in this
way will permit us to efficiently compute many aspects of the structure of $S$
without enumerating its elements exhaustively. For instance, using this
decomposition, we can test membership in $S$, compute the size, Green's
structure, idempotents, elements, and maximal subgroups of $S$.

Throughout the remainder of this section we suppose that $U$ is an arbitrary
finite semigroup, and $S$ is a subsemigroup of $U$.  


\subsection{Actions on Green's classes}

We require the following actions of $S$: the action on $\P(U)$ induced
by multiplying elements of $U$ on the right by $s\in S^1$, i.e.:
\begin{equation}\label{equation-action-on-U}
  As=\set{as}{a\in A}  
\end{equation}
where $s\in S^1$ and $A\subseteq U$; and the action of $S$ on $U/\L^U$ defined by 
\begin{equation}\label{equation-action-on-L}
  L_x\cdot s=L_{xs}
\end{equation}
for all $x\in U$ and $s\in S^1$. The latter defines an action since $\L^U$ is a
right congruence.  

In general, the actions defined in \eqref{equation-action-on-U} and
\eqref{equation-action-on-L} are not equal when restricted to $U/\L^U$. For
example, it can be the case that $L_x s\subsetneq L_{xs}$ and, in
particular, $L_x s\not\in U/\L^U$.  However, the actions do coincide in one case,
as described in the next lemma, which is particularly important here. 


\begin{lem}\label{lem-equiv-actions}
  Let $x, y \in U$ be arbitrary. Then $L_x, L_y\in U/\L^U$ belong to the same
  s.c.c.\ of the action of $S$ defined by \eqref{equation-action-on-L} if and
  only if $L_x$ and $L_y$ belong to the same s.c.c.~of the action of $S$
  defined by \eqref{equation-action-on-U}.
\end{lem}
\begin{proof}
  The converse implication is trivial.
  If $L_x, L_y\in U/\L^U$ belong to the same strongly connected component under
  the action \eqref{equation-action-on-L}, then, there exists $s\in S^1$ such that
  $L_{xs}=L_y$. Hence, by Green's Lemma \cite[Lemma 2.2.1]{Howie1995aa}, the
  function from $L_x$ to $L_{xs}=L_y$ defined by $z\mapsto zs$ for all $z\in
  L_x$ is a bijection and so $L_x\cdot s=L_{xs}=\set{ys}{y\in L_x}$.
\end{proof}


Although $S$ does not, in general, act on the $\L$-classes of $U$ by right
multiplication, by Lemma~\ref{lem-equiv-actions}, it does act by right
multiplication on the set of $\L^U$-classes within a strongly connected component
of its action.  We will largely be concerned with the strongly connected
components of the restriction of the action on $\P(U)$ in
\eqref{equation-action-on-U} to $U/\L^U$.  In this context, by
Lemma~\ref{lem-equiv-actions}, we may, without ambiguity, use the
actions defined in \eqref{equation-action-on-U} and
\eqref{equation-action-on-L} interchangeably. 

If $x, y\in S$, then we write $L_x^U\sim L_y^U$ and $R_x^U\sim R_y^U$ to denote
that these $\L^U$- and $\R^U$-classes belong to the same strongly connected
component of an action defined in \eqref{equation-action-on-U} or
\eqref{equation-action-on-L}.
We will make repeated use of the following lemma later in this section. 

\begin{lem} \label{lem-greens-simple}
  Let $x\in S$ and $s,t\in S^1$ be arbitrary. Then:
  \begin{enumerate}[\rm (a)] 

    \item
      $L_x ^ U\sim L_{xs}^ U$ if and only if $x\R^S xs$;

    \item
      $R_x ^ U\sim R_{tx}^ U$ if and only if $x \L^S tx$;

    \item
       $L_x ^ U\sim L_{xs}^ U$ and $R_x ^ U\sim R_{tx}^ U$ together imply that
       $x \D^S txs$. 
  \end{enumerate}
\end{lem}
\begin{proof}
  We prove only parts (a) and (c), since part (b) is the dual of (a).
  
  \textbf{(a).}
  ($\Rightarrow$)
  By assumption, $L_{xs}^U$ and $L_x^U$ belong to the same s.c.c.~of the action
  of $S$ on $U/\L^U$ by right multiplication. Hence, by
  Proposition~\ref{prop-technical}(a), there exists $\ov{s}\in S^1$ such that
  $L_{xs}^U\ov{s}=L_x^U$ and $s\ov{s}$ acts on $L_x^U$ as the
  identity.  Hence, in
  particular, $xs\ov{s}=x$ and so $xs\R^S x$. 

  ($\Leftarrow$)
  Suppose $x\R^S xs$. Then there exists $t\in S^1$ such that
  $xst=x$. It follows that $L_{x}^U\cdot s= L_{xs}^U$
  and $L_{xs}^U\cdot t = L_{x}^U$. Hence
  $L_{x}^U\sim L_{xs}^U$.\smallskip

  \textbf{(c).} 
  Suppose that $L_{x}^U \sim L_{xs}^U$ and $R_{x}^U\sim R_{tx}^U$. Then,
  by parts (a) and (b), $x \R^S xs$ and $x\L^Stx$. The latter implies that
  $xs\L^Stxs$, and so $x\D^S txs$. 
\end{proof}


\subsection{Faithful representations of stabilisers}

Let $U$ be an arbitrary finite semigroup, and let $S$ be any subsemigroup of
$U$. 

The algorithms presented in this paper involve computing with the groups $S_L$
arising from the stabiliser $\stab_S(L)$ of an $\L^U$-class $L$.  It is
impractical to compute directly with $S_L$ or $U_L$ with their natural actions
on $L$.  For example, the $\L$-class $L$ of any transformation $x\in T_{10}$
with $5$ points in its image has $5,103,000$ elements but, in this case,
$U_{L}$ has a faithful permutation representation on only $5$ points.

With the preceding comments in mind, throughout this section we will make
statements about arbitrary faithful representations of $U_{L}$ or $S_L$, 
$L\in U/\L^U$, rather than about the natural action of $U_L$ (or $S_L$) on $L$.
More specifically, if $x\in U$ and $\zeta$ is any faithful representation of
$U_{L_x^U}$, then we define 
\begin{equation}\label{equation-demux}
  \l{x}:\stab_U(L_x^U)\to (U_{L_x^U})\zeta\qquad \text{by} \qquad
  (u)\l{x} = (u|_{L_x^U})\zeta\qquad\text{for all}\ u\in U.
\end{equation}
It is clear that $\l{x}$ is a homomorphism.  Since $S$ is a subsemigroup of
$U$, it follows that $\stab_S(L_x^U)$ is a subsemigroup of $\stab_U(L_x^U)$ and
$S_{L_x^U}$ is a subgroup of $U_{L_x^U}$. Hence $\l{x}:\stab_U(L_x^U)\to
(U_{L_x^U})\zeta$ restricted to $\stab_S(L_x^U)$ is a homomorphism from
$\stab_S(L_x^U)$ to $(S_{L_x^U})\zeta$.

It is possible that, since $\zeta$ depends on the $\L^U$-class of $x$ in $U$, we
should use $\zeta_{L_x^U}$ to denote the faithful representation given above.
However, to simplify our notation we will not do this.  Note that, with
this definition, $(s,t)\in \ker(\l{x})$ if and only if $s|_{L_{x}^U} =
t|_{L_x^U}$.  To further simplify our notation, we will write $S_{x}$ and $U_{x}$ to
denote $(\stab_{S}(L_x^U))\l{x}=(S_{L_x^U})\zeta$ and
$(\stab_{U}(L_x^U))\l{x}=(U_{L_x^U})\zeta$, respectively.  Analogously, for
every $x\in U$ we suppose that we have a homomorphism $\r{x}$ from
$\stab_U(R_x^U)$ into a group where the image of $\r{x}$ is isomorphic to
${}_{R_x^U}U$. We write ${}_{x}S$ and ${}_{x}U$ for
$(\stab_S(R_x^U))\r{x}$ and $(\stab_U(R_x^U))\r{x}$, respectively.

In the case that $S$ is a transformation, partial permutation, matrix, or
partition semigroup, or a subsemigroup of a Rees $0$-matrix semigroup,  we will
show in Section~\ref{section-special-cases} how to obtain faithful
representations of the stabilisers of $\L^U$- and $\R^U$-classes as permutation
or matrix groups. It is then possible to use algorithms, such as the
Schreier-Sims algorithm, from Computational Group Theory to compute with these
groups.

We will make repeated use of the following straightforward lemma. 

\begin{lem}\label{lem-free}
  Let $x\in U$ and $s,t\in \stab_{U}(L_x^U)$ be arbitrary. Then the following
  are equivalent:
  \begin{enumerate}[\rm (a)]
    \item $(s)\l{x}=(t)\l{x}$;
    \item $ys=yt$ for all $y\in L_x^U$;
    \item there exists $y\in L_x^U$ such that $ys=yt$.
  \end{enumerate}
\end{lem}
\begin{proof} 
  (a) $\Rightarrow$ (b) If $(s)\l{x}=(t)\l{x}$, then $(s|_{L_x^U})\zeta=
  (t|_{L_x^U})\zeta$ and so $s|_{L_x^U}= t|_{L_x^U}$. It follows that 
  $ys=yt$ for all $y\in L_x^U$.

  (b) $\Rightarrow$ (c) is trivial. 
  
  (c) $\Rightarrow$ (a)
  Suppose that $y\in L_x^U$ is such that $ys=yt$. If $z\in L_x^U$ is arbitrary,
  then there exists $u\in U^1$ such that $z=uy$. Hence $zs=uys=uyt=zt$ and so
  $s|_{L_x^U}=t|_{L_x^U}$. It follows, by the definition of $\mu_{x}$, that
  $(s,t)\in \ker(\l{x})$ and so $(s)\l{x} = (t)\l{x}$. 
\end{proof}

The analogue of Lemma~\ref{lem-free} also holds for $\stab_{U}(R_x^U)$ and
$\r{x}$; the details are omitted.


\subsection{A decomposition for Green's classes}

In this section, we show how to decompose an $\R$- or $\L$-class of our
subsemigroup $S$ as briefly discussed above. 

\begin{prop}[cf. Theorems 3.3 and 4.3 in \cite{Linton1998aa}]\label{prop-main-0}
 If $x\in S$ is arbitrary,  then:
  \begin{enumerate}[\rm (a)] 
    \item $\set{L_{y}^U}{y\R^S x}$ is a strongly connected component of
      the right action of $S$ on $\set{L_x^U}{x\in S}$;
    \item $\set{R_{y}^U}{y\L^S x}$ is a strongly connected component of the
      left action of $S$ on $\set{R_x^U}{x \in S}$.
  \end{enumerate}
\end{prop}
\begin{proof}
  We prove only the first statement, since the second is dual.
  Suppose $y\in S$ and $x\not =y$. Then $y\R^Sx$ if and only if there exists
  $s\in S$ such that $xs=y$ and $xs\R^S x$. By
  Lemma~\ref{lem-greens-simple}(a), $xs\R^S x$ if and only if $L_{x}^U\sim
  L_{xs}^U$.
\end{proof}

\begin{cor}\label{cor-greens-simple}
  Let $x, y, s, t\in S$. Then the following hold:
  \begin{enumerate}[\rm(a)]
    \item if $x \R^S y$ and $xs \L^U y$, then $xs \R^S y$;
    \item if $x \L^S y$ and $tx \R^U y$, then $tx \L^S y$.
  \end{enumerate}
\end{cor}
\begin{proof}
  Again, we only prove part (a).  Since $xs\L^U y$, it follows that $L_{xs}^U =
  L_{y}^U$ and, since  $x\R^S y$, by Proposition~\ref{prop-main-0}(a),
  $L_{x}^U\sim L_{y}^U = L_{xs}^U$.  Hence $xs \R^S x$ by
  Lemma~\ref{lem-greens-simple}(b), and, since $x \R^S y$ by assumption, the
  proof is complete. 
\end{proof}


\begin{prop}[cf. Theorems 3.7 and 4.6 in \cite{Linton1998aa}]\label{prop-main-1}
  Suppose that $x\in S$ and there exists $x'\in U$ where $xx'x=x$
  (i.e.~$x$ is regular in $U$).  Then the following hold:
  \begin{enumerate}[\rm (a)]
    \item $L_x^U\cap R_x^S$ is a group
      under the multiplication $*$ defined by $s\ast t=sx't$ for all $s,t \in
      L_x^U\cap R_x^S$ and its identity is $x$;
    \item $\phi:S_{x}\to L_x^U\cap R_x^S$ defined by 
      $((s)\l{x})\phi=xs$, for all $s\in \stab_S(L_x^U)$, is an isomorphism;
    \item $\phi^{-1}:L_x^U\cap R_x^S\to S_{x}$ is defined by 
      $(s)\phi^{-1}=(x's)\l{x}$ for all $s\in L_x^U\cap R_x^S$.
    \end{enumerate}
\end{prop}
\begin{proof}
  We begin by showing that $x$ is an identity under the multiplication $*$ of
  $L_x^U\cap R_x^S$.
  Since $x'x\in L_x^U$ and $xx'\in R_x^U$ are idempotents, it follows that $x'x$ is
  a right identity for $L_x^U$ and $xx'$ is a left identity for $R_x^S\subseteq
  R_x^U$.  So, if $s\in L_x^U\cap R_x^S$ is arbitrary, then 
  \begin{equation*}
    x*s=xx's=s=sx'x=s*x,
  \end{equation*}
  as required.

  We will prove that part (b) holds, which implies part (a). \smallskip
  
  \noindent\textbf{$\phi$ is well-defined.} If $s\in \stab_S(L_x^U)$, then
  $xs\L^U x$. Hence, by Corollary~\ref{cor-greens-simple}(a), $xs\R^S x$ and so
  $((s)\l{x})\phi=xs\in L_x^U\cap R_x^S$. If $t\in\stab_S(L_x^U)$ is such that
  $(t)\l{x}=(s)\l{x}$, then, by Lemma~\ref{lem-free}, $xt=xs$.\smallskip

  \noindent\textbf{$\phi$ is surjective.} 
  Let $s\in L_x^U\cap R_x^S$ be arbitrary. Then $xx's=x*s=s$ since $x$ is the
  identity of $L_x^U\cap R_x^S$. It follows that  
  $$L_x^U\cdot x's=L^U_s=L^U_x$$
  and so $x's\in \stab_U(L_x^U)$.  Since $x\R^{S}s$, there exists $u\in S^1$
  such that $xu=s=xx's$.  It follows that $u\in \stab_S(L_x^U)$ and, by
  Lemma~\ref{lem-free}, $(u)\l{x}=(x's)\l{x}$. Thus $((u)\l{x})\phi=xu=s$ and
  $\phi$ is surjective. 
  \smallskip

  \noindent\textbf{$\phi$ is a homomorphism.} Let $s, t\in \stab_S(L_x^U)$.
  Then, since $xs\in L_x^U$ and $x'x$ is a right identity for $L_x^U$, 
  $$((s)\l{x})\phi*((t)\l{x})\phi=xs*xt=xsx'xt=xst=
    ((st)\l{x})\phi=((s)\l{x}\cdot (t)\l{x})\phi,$$ 
  as required. 
  \smallskip

  \noindent\textbf{$\phi$ is injective.}
  Let $\theta: L_x^U\cap R_x^S\to S_{x}$ be defined by $(y)\theta=(x'y)\l{x}$
  for all $y\in L_x^U\cap R_x^S$. We will show that $\phi\theta$ is the
  identity mapping on $S_x$, which implies that $\phi$ is injective, that
  $(y)\theta\in S_x$ for all $y\in L_x^U\cap R_x^S$ (since $\phi$ is
  surjective), and also proves part (c) of the proposition. If $s\in
  \stab_S(L_x^U)$, then $((s)\l{x})\phi\theta=(xs)\theta=(x'xs)\l{x}$.  But
  $xx'xs=xs$ and so $(x'xs)\l{x}=(s)\l{x}$ by Lemma~\ref{lem-free}.  Therefore,
  $((s)\l{x})\phi\theta=(s)\l{x}$, as required. 
\end{proof}\smallskip


We state the analogue of Proposition~\ref{prop-main-1} for the action of $S$ on
$U/\R^U$ by left multiplication. 


\begin{prop}\label{prop-main-1.5}
  Suppose that $x\in S$ and there exists $x'\in U$ where $xx'x=x$.
  Then the following hold:
  \begin{enumerate}[\rm (a)]

    \item
      $L_x^S\cap R_x^U$ is a group under the multiplication $*$ defined by
      $s\ast t=sx't$ for all $s,t \in L_x^S\cap R_x^U$  and its identity is
      $x$;

    \item
      $\phi:{}_{x}S\to L_x^S\cap R_x^U$ defined by $((s)\r{x})\phi=sx$, for all
      $s\in \stab_S(R_x^U)$, is an isomorphism; 

    \item
      $\phi^{-1}:L_x^S\cap R_x^U\to {}_{x}S$ is defined by 
      $(s)\phi^{-1}=(sx')\r{x}$ for all $s\in L_x^S\cap R_x^U$.
  \end{enumerate}
\end{prop}


We can also characterise an $\H^S$-class $H$ in our subsemigroup $S$ in terms
of the stabilisers of the $\L^U$- and $\R^U$-class containing $H$. Note that in
the special case that $S = U$, it follows immediately from
Proposition~\ref{prop-main-1} that $U_x$ is isomorphic to $H_x^U = L_x^U\cap
R_x^U$ under the operation $*$ defined in the proposition.

\begin{prop}[cf. Theorem 5.1 in \cite{Linton1998aa}] \label{prop-main-2}
  Suppose that $x\in S$ and there exists $x'\in U$ where $xx'x=x$.
  Then the following hold:
  \begin{enumerate}[\rm (a)] 

    \item $\Psi: {}_{x}S\to U_{x}$ defined by
      $((s)\r{x})\Psi=(x'sx)\l{x}$, $s\in \stab_S(R_x^U)$, is an
      embedding;
    
    \item $H_x^S$ is a group under the multiplication $s*t=sx't$, with identity
      $x$, and it is isomorphic to $S_{x}\cap ({}_{x}S)\Psi$;
    
    \item $H_x^S$ under $*$ is isomorphic to the Sch\"utzenberger group of
      $H_x^S$.
  \end{enumerate}
\end{prop} 
\begin{proof} 
  \textbf{(a).}  Let $\phi:{}_{x}S\to L_x^S\cap R_x^U$ be the isomorphism defined in
  Proposition~\ref{prop-main-1.5}(b), and let $\theta: L_x^U\cap R_x^U\to
  U_{x}$ be the isomorphism defined in Proposition~\ref{prop-main-1}(c) (applied to
  $U$ as a subsemigroup of itself). Then, since $L_x^S\cap R_x^U\subseteq
  L_x^U\cap R_x^U$, $\phi\theta:{}_{x}S\to U_{x}$ is a embedding (being the
  composition of injective homomorphisms).  By definition,
  $((s)\r{x})\phi\theta=(x'sx)\l{x}=((s)\r{x})\Psi$, for all $s\in
  \stab_S(R_x^U)$.\smallskip 
  
  \textbf{(b).} Note that  
  $H_x^S=L_x^S\cap R_x^S=(L_x^U\cap R_x^S)\cap (R_x^U\cap L_x^S)$. 
  From Proposition~\ref{prop-main-1}(c), $\phi_1^{-1}:R_x^S\cap L_x^U\to
  S_{x}\leq U_{x}$ defined by 
  $$(s)\phi_1^{-1}=(x's)\l{x}$$
  is an isomorphism (where $R_x^S\cap L_x^U$ has multiplication $*$ defined in 
  Proposition~\ref{prop-main-1}(a) as $s*t=sx't$ for all $s,t\in R_x^S\cap
  L_x^U$).  Similarly, by Proposition~\ref{prop-main-1.5}(c), 
  $\phi_2^{-1}:R_x^U\cap L_x^S\to {}_{x}S$ defined by 
  $$(s)\phi_2^{-1}=(sx')\r{x}$$
  is an isomorphism. Hence if $s\in H_x^S$, then, by
  Proposition~\ref{prop-main-1}(a),
  $$(s)\phi_2^{-1}\Psi=(x'sx'x)\l{x}=(x's)\l{x}=(s)\phi_1^{-1}$$
  and so $\phi_1^{-1}$, restricted to $H_x ^ S$, is an injective homomorphism
  from $H_x^S$ under $*$ into $S_{x}\cap ({}_{x}S)\Psi$. 

  If $g\in S_{x}\cap ({}_{x}S)\Psi$, then there exists $a\in \stab_S(R_x^U)$
  such that $(x'ax)\l{x}=g$.  Since $ax\R^U x$ and $xx'$ is a
  left identity in $R_x^U$, it follows that  $xx'ax=ax$ and so 
  $ax=xx'ax=((x'ax)\l{x})\phi_1\in R_x^S\cap L_x^U\subseteq R_x^S$
  where $\phi_1$ is given in Proposition~\ref{prop-main-1}(b).
  Similarly,  $(a)\r{x}\in {}_{x}S$ implies that 
  $ax=((a)\r{x})\phi_2\in {R_x^U}\cap L_x^S\subseteq L_x^S$.
  Therefore $ax\in H_x^S$ and $(ax)\phi_1^{-1}=(x'ax)\l{x}=g$, and so
  $\phi_1^{-1}$ is surjective and thus an isomorphism from $H_x^S$ to
  $S_{x}\cap ({}_{x}S)\Psi$, as required. \smallskip

  \textbf{(c).} The \textit{Sch\"utzenberger group} of $H=H_x^S$ is defined to be
  the quotient of $\stab_S(H)$ by the kernel of its action on $H$ (by right
  multiplication on the elements of $H$). In our notation the Sch\"utzenberger
  group of $H$ is denoted $S_H$, however in the literature it is usually
  denoted by $\Gamma_R(H)$. It is well-known that $S_H$ acts transitively and
  freely on $H$, see for example \cite[Section A.3.1]{Rhodes2009aa}. It follows
  that $\phi:S_H\to H$ defined by
  $$(s|_{H})\phi=xs$$ 
  is a bijection. If $s|_H, t|_H\in S_H$, then $xs\in H$ and so
  $xsx'x=xs*x=xs$, since $x$ is the identity of $H$ by
  Proposition~\ref{prop-main-1}(a).  Thus 
  $$(s|_H\cdot t|_H)\phi=xst=xsx'xt=(s|_H)\phi*(t|_H)\phi$$
  and $\phi$ is an isomorphism, as required. 
\end{proof}\smallskip


The statement of Proposition~\ref{prop-main-2} can be simplified somewhat in the
case that the element $x\in S$ is regular in $S$ and not only in $U$.


\begin{cor}\label{cor-regular}
  Suppose that $x\in S$ and there exists $x'\in S$ where $xx'x=x$.
  Then the following three groups are isomorphic: $H_x^S$ under the
  multiplication $s*t=sx't$, $S_{x}$, and ${}_{x}S$.  Furthermore,
  $\Psi:{}_{x}S \to S_{x}$ defined by $((s)\r{x})\Psi=(x'sx)\l{x}$ is an
  isomorphism. 
\end{cor}
\begin{proof} 
  Since $x$ is regular in $S$, it follows that $H_x^S=L_x^S\cap R_x^S=L_x^U\cap
  R_x^S$ and, similarly, $H_x^S=R_x^U\cap L_x^S$.  So, the first part of the
  statement follows by Propositions~\ref{prop-main-1}(b)
  and~\ref{prop-main-1.5}(b).

  By Propositions~\ref{prop-main-1.5}(b) and~\ref{prop-main-1}(c), respectively,
  there exist isomorphisms $\phi:{}_{x}S\to L_x^S\cap R_x^U$ and $\theta:
  L_x^U\cap R_x^S\to S_{x}$. Therefore since $H_x^S=L_x^U\cap R_x^S=R_x^U\cap
  L_x^S$, it follows that $\Psi=\phi\theta: {}_{x}S \to S_{x}$ is an
  isomorphism.
\end{proof}\smallskip


We collect some corollaries of what we have proved so far. 

\begin{cor} \label{cor-collect}
  If $x,y\in S$ are regular elements of $U$, then the following hold: 
  \begin{enumerate}[\rm (a)]   
    \item
      If $x\R^{S}y$, then $S_{L_x^U}$ and $S_{L_y^U}$ are conjugate subgroups
      of $\sym(U)$. In particular, $S_x$ and $S_y$ are isomorphic;

    \item
      $|R_x^S|$ equals the size of the group $S_{x}$ multiplied by
      the size of the s.c.c.~of $L_{x}^U$ under the action of $S$ on
      $\set{L_x^U}{x\in S}$;

    \item
      If $L_{x}^U\sim L_{y}^U$, then $|R^S_x|=|R^S_y|$;

    \item
      If $x\R^{S}y$ and $u\in S ^ 1$ is such that $L_{x}^U\cdot
      u=L_{y}^U$, then the function from $L_x^U\cap R_x^S$ to $L_y^U\cap
      R_x^S$
      defined by $s \mapsto su$ is a bijection. 
  \end{enumerate}
\end{cor}
\begin{proof} \textbf{(a).} 
  Since $x\R^{S}y$, it follows by Proposition~\ref{prop-main-0}(a) that $L_x^U$ and
  $L_y^U$ are in the same s.c.c.~of the action of $S$ on the $\L$-classes of $U$.
  Thus, by Proposition~\ref{prop-technical}(b), it follows that $S_{L_x^U}$ and
  $S_{L_y^U}$ are conjugate subgroups of the symmetric group on $U$, and so, in
  particular, are isomorphic. 
  \smallskip

  \textbf{(b).} 
  The set $R_x^S$ is partitioned by the sets $R_x^S\cap L_y^U=R_{y}^S\cap
  L_{y}^U$ for all $y\in R_x^S$. By Proposition~\ref{prop-main-1}(b),
  $|R_{y}^S\cap L_{y}^U|=|S_{y}|$ and by part (a), $|S_{y}|= |S_{x}|$ for all
  $y\in R_x^S$. Thus $|R_x^S|$ equals $|S_{x}|$ multiplied by the number of
  distinct $\L^U$-classes of elements of $R_x^S$.
  Proposition~\ref{prop-main-0}(a) says that $\set{L_{y}^U}{y\in R_x^S}$ is a
  s.c.c.~of the right action of $S$ on $\set{L_x^U}{x\in S}$.\smallskip

  \textbf{(c).} 
  This follows immediately from parts (a) and (b).
  \smallskip

  \textbf{(d).} 
  Let $s\in L_x^U\cap R_x^S$ be arbitrary. Then $s\R^Sx\R^Sy$ and 
  $su\L^U xu\L^Uy$, and so, by Corollary~\ref{cor-greens-simple}(a), 
  $su\R^S y$, i.e. $su\in L_y^U\cap R_x^S$. 
  By Proposition~\ref{prop-technical}(a), there exists $\overline{u}\in S^1$
  such that $su\overline{u}=s$ for all $s\in L_x^U\cap R_x^S$. Therefore
  $s\mapsto su$ and $t\mapsto t\overline{u}$ are mutually inverse bijections
  from $L_x^U\cap R_x^S$ to $L_y^U\cap R_x^S$ and
  back.
\end{proof}\smallskip

For the sake of completeness, we state the analgoue of
Corollary~\ref{cor-collect} for $\L$-classes. 

\begin{cor} \label{cor-collect-analogue}
  If $x,y\in S$ are regular elements of $U$, then the following hold: 
  \begin{enumerate}[\rm (a)]   

    \item
      If $x\L^{S}y$, then ${}_{R_x^U}S$ and ${}_{R_y^U}S$ are conjugate
      subgroups of $\sym(U)$. In particular, ${}_xS$ and ${}_yS$ are
      isomorphic;

    \item
      $|L_x^S|$ equals the size of the group ${}_{x}S$ multiplied by
      the size of the s.c.c.~of $R_{x}^U$ under the action of $S$ on
$\set{R_x^U}{x\in S}$;

    \item
      If $R_{x}^U\sim R_{y}^U$, then $|L^S_x|=|L^S_y|$;

    \item
      If $x\L^{S}y$ and $u\in S ^ 1$ is such that $u\cdot R_{x}^U = R_{y}^U$,
      then the function from $R_x^U\cap L_x^S$ to $R_y^U\cap L_x^S$ defined by
      $s\mapsto us$ is a bijection.
  \end{enumerate}
\end{cor}


\subsection{Membership testing}

Let $U$ be a semigroup and let $S$ be a subsemigroup of $U$. The next
proposition shows that testing membership in an $\R$-class of $S$ is equivalent
to testing membership in a stabiliser of an $\L$-class of $U$.  Since the latter
is a group, this reduces the problem of membership testing in $\R^S$-classes to
that of membership testing in a group, so we can then take advantage of
efficient algorithms from computational group theory; such as the Schreier-Sims
algorithm \cite[Section 4.4]{B.-Eick2004aa}.

\begin{prop}
  \label{prop-R-membership}
  Suppose that $x\in S$ and there is $x'\in U$ with $xx'x=x$. If $y\in U$ is
  arbitrary, then $y \R^S x$ if and only if $y \R^U x$, $L_{y}^U\sim
  L_{x}^U$, and $(x'yv)\l{x}\in S_{x}$ where $v\in S^1$ is any element such
  that $L_{y}^U\cdot v=L_{x}^U$.
\end{prop}
\begin{proof}
  ($\Rightarrow$) 
  Since $R_x^S\subseteq R_x^U$, $y \R^U x$ and from
  Proposition~\ref{prop-main-0}(a), $L_{y}^U\sim L_{x}^U$.
  Suppose that $v\in S^1$ is such that $L_{y}^U\cdot v= L_{x}^U$. Then, by
  Corollary~\ref{cor-collect}(d),  $yv\in L_x^U\cap R_x^S$ and so, by
  Proposition~\ref{prop-main-1}(c), $(x'yv)\l{x}\in S_x$. 
  \smallskip

  ($\Leftarrow$) Since $y\in R_x^U$ and $xx'$ is a left identity in its
  $\R^U$-class, it follows that $xx'y=y$.  Suppose that $v\in S^1$ is any element
  such that $L_{y}^U\cdot v=L_{x}^U$ (such an element exists by
  assumption). Then, by assumption, $(x'yv)\mu_x\in S_x$ and so by 
  Proposition~\ref{prop-main-1}(b), $yv = x \cdot x'yv\in L_x^U\cap R_x^S$.
  But $L_{y}^U\sim L_{yv}^U$, and so, by
  Lemma~\ref{lem-greens-simple}(a), $y\R^S yv$, and so $x\R^Syv\R^Sy$, as
  required. 
\end{proof}\smallskip

The following corollary of Proposition~\ref{prop-R-membership} will be
important in Section~\ref{section-algorithms}.

\begin{cor}
  \label{cor-R-membership}
  Let $x\in S$ be such that there is $x'\in U$ with $xx'x=x$, and let $y\in U$
  be such that there exist $u\in S^1$ and $v\in U^1$ with $L_{x}^U \cdot u =
  L_{y}^U$ and $xuv = x$. Then $y \R^S x$ if and only if $y \R^U x$,
  $L_{y}^U\sim L_{x}^U$, and $(x'yv)\mu_x\in S_x$.
\end{cor}
\begin{proof}
  ($\Rightarrow$) That $y\R^U x$ and $L_{y}^U\sim L_{x}^U$ follows from
  Proposition~\ref{prop-R-membership}. Suppose $u\in S$ and $v\in U$ are such
  that $L_{x}^U \cdot u = L_{y}^U$ and $xuv = x$.  By
  Proposition~\ref{prop-technical}(a) there exists $\ov{u}\in S^1$ such that
  $xu\ov{u} = x = xuv$. It follows that $z\ov{u} = zv$ for all $z\in L_{xu}^S$,
  and, in particular, $y\ov{u} = yv$.  Hence, by
  Proposition~\ref{prop-R-membership}, $(x'yv)\mu_x = (x'y\ov{u})\mu_x  \in
  S_x$.

  ($\Leftarrow$) It suffices by Proposition~\ref{prop-R-membership} to show
  that there exists $w\in S^1$ such that $L_{y}^U\cdot w= L_{x}^U$ and
  $(x'yw)\l{x}\in S_{x}$. Since $L_{x}^U\sim
  L_{y}^U$ and $L_{x}^U \cdot u = L_{y}^U$, by
  Proposition~\ref{prop-technical}(a) and Lemma~\ref{lem-free}, there exists
  $\ov{u}\in S^1$ such that $xu\ov{u} = x = xuv$. Hence, as above, $y\ov{u} =
  yv$ and so $(x'y\ov{u})\mu_x = (x'yv)\mu_x \in S_x$, as required.
\end{proof}

We state an analogue of Proposition~\ref{prop-R-membership} for
$\L$-classes, with a slight difference. 

\begin{prop}
  \label{prop-L-membership}
  Suppose that $x\in S$ and there is $x'\in U$ with $xx'x=x$. If $y\in U$ is
  arbitrary, then $y\L^S x$ if and only if $y \L^U x$, $R_{y}^U\sim R_{x}^U$,
  and $(x'vy)\l{x}\in ({}_{x}S)\Psi$ where $v\in S^1$ is any element such that
  $v\cdot R_{y}^U=R_{x}^U$ and $\Psi:{}_{x}S\to U_{x}$ defined by
  $((s)\r{x})\Psi= (x'sx)\l{x}$ is the embedding from
  Proposition~\ref{prop-main-2}(a).
\end{prop}
\begin{proof}
  The direct analogue of Proposition~\ref{prop-R-membership} states that
  $y\L^Sx$ if and only if  $y \L^U x$, $R_{y}^U\sim R_{x}^U$, and
  $(vyx')\r{x}\in {}_{x}S$. The last part is equivalent to 
  $((vyx')\r{x})\Psi=(x'vyx'x)\l{x}=(x'vy)\l{x}\in ({}_{x}S)\Psi$, as required. 
\end{proof}

Propositions~\ref{prop-R-membership} and~\ref{prop-L-membership} allow us to
express the elements of an $\R^S$- and $\L^S$-class in a particular form, which
will be of use in the algorithms later in the paper. 

\begin{cor} \label{cor-elements}
  Suppose that $x\in S$ and there is $x'\in U$ with $xx'x=x$. Then 
  \begin{enumerate}[\rm(a)]
    \item 
      if $\mathcal{U}$ is any subset of $S^1$ such that 
      $\set{L_{xu}^U}{u\in \mathcal{U}} = \set{L_{y}^U}{L_{y}^U \sim
      L_{x}^U}$, then 
      $$R_x^S  = \set{xsu}{s\in \stab_S(L_x^U),\ u\in \mathcal{U}};$$
    \item 
      if $\mathcal{V}$ is any subset of $S^1$ such that 
      $\set{R_{vx}^U}{v\in \mathcal{V}} = \set{R_{y}^U}{R_{y}^U \sim
      R_{x}^U}$, then 
      $$L_x^S = \set{vtx}{t\in \stab_S(R_x^U),\ v\in \mathcal{V}}.$$
  \end{enumerate}
\end{cor}
\begin{proof}
  We prove only part (a), since part (b) is dual.
  
  Let $s\in \stab_S(L_x^U)$ and let $u\in \mathcal{U}$ be arbitrary.
  Since $L_{xsu}^U = L_{xu}^U \sim
  L_{x}^U$, it follows, by Lemma~\ref{lem-greens-simple}(a), that $xsu\R^Sx$.

  If $y\R^S x$, then, by Proposition~\ref{prop-main-0}, $L_{y}^U\sim
  L_{x}^U$, and so there exists $u\in \mathcal{U}$ such that
  $L_{x}^U\cdot u = L_{y}^U$.  Since $xx'$ is a left identity for
  $R_x^U$, $xx'y = y$.  By Proposition~\ref{prop-technical}(a), there is
  $\ov{u}\in S^1$ such that $L_{y}^U\cdot \ov{u}= L_{x}^U$ and
  $y\ov{u}u=y$.  Hence, by Proposition~\ref{prop-R-membership}, $x'y\ov{u}\in
  \stab_S(L_x^U)$ and $y = x\cdot x'y\ov{u}\cdot u$. 
\end{proof}

We will prove the analogue of Proposition~\ref{prop-R-membership} for $\D$-classes,
for which we require following proposition.

\begin{prop}[cf. Theorem 6.2 in \cite{Linton1997aa}]
  \label{prop-D-class-elements-1}
  If $x\in S$ is such that there is $x'\in U$ with $xx'x=x$, then 
  $D_x^S\cap H_x^U = \set{sxt}{s\in \stab_S(R_{x}^U),\ t\in \stab_S(L_x^U)}$.
\end{prop}
\begin{proof}
  Let $s\in \stab_S(R_x^U)$ and $t\in \stab_S(L_x^U)$ be arbitrary.
  It follows that $L_{xt}^U = L_{x}^U$ and $R_{sx}^U = R_{x}^U$, and so,
  by Lemma~\ref{lem-greens-simple}, $xt\R^S x$, $sx\L^Sx$, and $x\D^S sxt$.
  Also $xt\R^Sx$ implies that $sxt\R^Ssx\R^Ux$ and 
  $sx\L^Sx$ implies $sxt\L^Sxt\L^Ux$, and so $sxt\in H_x^U$, as required. 

  For the other inclusion, let $y\in D_x^S\cap H_x^U$ be arbitrary.
  Then $x\D^Sy$ and so there is $s\in S^1$ such that $sx\L^S x$ and $sx\R^Sy$.
  Hence
  $s\cdot R_x^U=R_{sx}^U=R_{y}^U=R_x^U$
  and so $s\in \stab_S(R_x^U)$. Since $sx\R^Sy$, there exists $t\in S^1$ such that 
  $sxt=y$ and so
  $L_x^U\cdot t=L_{sx}^U\cdot t=L_y^U=L_x^U,$
  which implies that $t\in \stab_S(L_x^U)$, as required. 
\end{proof}

\begin{prop}
  \label{prop-D-membership}
  Suppose that $x\in S$ and there is $x'\in U$ with $xx'x=x$. If $y\in U$ is
  arbitrary, then $y\D^S x$ if and only if $L_{y}^U\sim L_{x}^U$,
  $R_{y}^U\sim R_{x}^U$, and for any $u, v\in S^1$ such that $L_{y}^U\cdot
  u=L_{x}^U$ and $v\cdot R_{y}^U=R_{x}^U$ there exists $t\in \stab_S(L_x^U)$
  such that 
  $$(x'vyu)\l{x}\cdot ((t)\l{x}) ^ {-1}\in({}_{x}S)\Psi,$$ 
  where $\Psi:{}_{x}S\to U_{x}$, defined by $((s)\r{x})\Psi= (x'sx)\l{x}$,
  is the embedding from Proposition~\ref{prop-main-2}(a). 
\end{prop} 
\begin{proof} 
  ($\Rightarrow$) Let $y\in D_x^S$. Then there exists $w\in S$ such that $y\R^S
  w\L^S x$. By Proposition~\ref{prop-main-0}(a) and (b), respectively, it
  follows that $L_{y}^U\sim L_{w}^U=L_{x}^U$, and $R_{x}^U\sim
  R_{w}^U=R_{y}^U$.

  Suppose that $u, v\in S^1$ are any elements such that $L_{y}^U\cdot
  u=L_{x}^U$ and $v\cdot R_{y}^U=R_{x}^U$.  Then, by
  Lemma~\ref{lem-greens-simple}, $vy \L^S y$ and $yu\R^S y$, and so $vyu \L^S yu
  \L^U x$ and $vyu \R^S vy \R^U x$. Thus $vyu \H^U x$ and, since $y \R^S yu \L^S
  vyu$, it follows that $vyu \D^S x$. 

  By Proposition~\ref{prop-D-class-elements-1}, there exist $s\in
  \stab_S(R_x^U)$ and $t\in \stab_S(L_x^U)$ such that $vyu=sxt$. Since $s\in
  \stab_S(R_x^U)$, it follows that 
  $((s)\r{x})\Psi=(x'sx)\l{x}\in({}_{x}S)\Psi$
  (by Proposition~\ref{prop-main-2}(a)). In particular, $x'sx\in
  \stab_{U}(L_x^U)$ and so 
  $(x'vyu)\l{x}=(x'sxt)\l{x}=(x'sx)\l{x}\cdot (t)\l{x}$
  and so 
  $$(x'vyu)\l{x}\cdot ((t)\l{x})^{-1}=(x'sx)\l{x}\in({}_{x}S)\Psi,$$
  as required.
  \smallskip
  
  ($\Leftarrow$)   
  Suppose that $u, v\in S^1$ are any elements such that $L_{y}^U\cdot
  u=L_{x}^U$ and $v\cdot R_{y}^U=R_{x}^U$. 
  Since $L_{y}^U\sim L_{x}^U= L_{yu}^U$ and $R_{y}^U\sim
  R_{x}^U=R_{vy}^U$,  it follows from Lemma~\ref{lem-greens-simple}(c), that 
  $y\D^S vyu$. 
  By assumption, there exist $s\in\stab_S(R_{x}^U)$,
  $t\in \stab_S(L_x^U)$ such that 
  $$(x'vyu)\l{x}=((s)\r{x})\Psi\cdot (t)\l{x}=(x'sxt)\l{x}.$$ 
  In particular, by Lemma~\ref{lem-free}, $xx'vyu=xx'sxt$.  Since $xx'$ is a
  left identity for $R_x^U=R_{vy}^U$, we deduce that $xx'vy=vy$.  Also, by
  Proposition~\ref{prop-D-class-elements-1}, $sxt\in
  D_x^S\cap H^U_x$ implies that $R_{sxt}^U=R_x^U$ and so $xx'sxt=sxt$. Thus
  $y\D^S vyu=xx'vyu=xx'sxt=sxt\D^S x$.  
\end{proof}


\subsection{Classes within classes}

The next two propositions allow us to determine the $\R^S$-, $\L^S$- and
$\H^S$-classes within a given $\D^S$-class in terms of the groups $S_x$, ${}_xS$,
and the action of $S$ on $\set{L_x^U}{x\in S}$ and $\set{R_x^U}{x\in S}$.

If $G$ is a group and $H$ is a subgroup of $G$, then a \textit{left transversal}
of $H$ in $G$ is a set of left coset representatives of $H$ in $G$.
\textit{Right transversals} are defined analogously.

\begin{prop}
  \label{prop-D-class-R-class-reps}
  Suppose that $x\in S$ and there is $x'\in U$ with $xx'x=x$  and that:
  \begin{enumerate}[\rm (a)]
    \item $\mathcal{C}$ is a minimal 
      subset of $\stab_{U}(L_x^U)$ such that $\set{(c)\l{x}}{c\in \mathcal{C}}$
      is a left transversal of $S_{x}\cap({}_{x}S)\Psi$ in $({}_xS)\Psi$ where
      $\Psi:{}_{x}S\to U_{x}$, defined by $((s)\r{x})\Psi= (x'sx)\l{x}$, is the
      embedding from Proposition~\ref{prop-main-2}(a);
    \item $\{u_1, \ldots, u_m\}$ is a minimal 
      subset of $S^1$ such that 
      $\set{u_i\cdot R_{x}^U}{1\leq i\leq m}$ equals the s.c.c.\ of 
      $R_{x}^U$ under the left action of $S$ on $\set{R_{x}^U}{x\in S}$.
  \end{enumerate}
  Then $\set{u_ixc}{c\in \mathcal{C},\ 1\leq i\leq m}$ is a minimal 
  set of $\H^S$-class representatives of $L_x^S$, and hence a minimal
  set of $\R^S$-class representatives for $D_x^S$.
\end{prop} 
\begin{proof} 
  We start by proving that for all $c\in \mathcal{C}$, there exists $c^*\in
  \stab_S(R_x^U)$ such that $c^*x = xc$ and that $xc\L^S x$.
  Suppose that $c\in \mathcal{C}$ is arbitrary. Then $(c)\mu_x\in ({}_xS)\Psi$
  and so there exists $c^*\in \stab_S(R_x^U)$ such that $((c^*)\nu_{x})\Psi =
  (x'c^*x)\mu_x = (c)\mu_x$. Hence, by Lemma~\ref{lem-free}, $xx'c^*x = xc$.
  Since $c^*\in \stab_S(R_x^U)$, Proposition~\ref{prop-main-1.5}(b) implies that
  $c^*x\R^Ux$ and $c^*x\L^Sx$, and so $c^*x = xx'c^*x = xc$. It follows that
  $xc = c^*x\L^Sx$. 

  By the assumption in part (b) and by Lemma~\ref{lem-greens-simple}(b), $u_ix
  \L^S x$ for all $i$, and so $u_ixc \L^S xc \L^S x$ for all $i$. Hence it
  suffices to show that $\set{u_ixc}{c\in \mathcal{C},\ 1\leq i\leq m}$ is a
  minimal set of $\R^S$-class representatives for $D_x^S$. 
  In other words, if $y\D^S x$, then we must show that $y\R^S u_ixc$ for some
  $i\in \{1,\ldots, m\}$ and $c\in \mathcal{C}$, and that ($u_ixc,
  u_jxd)\not\in \R^S$ if $i\not = j$ or $c\not = d$. 
  
  We start by showing that for every $y\in D_x^S\cap H_x^U$ there is $c\in
  \mathcal{C}$ such that $y\R^Sxc$.  By
  Proposition~\ref{prop-D-class-elements-1}, there exist $s\in
  \stab_S(R_{x}^U),\ t\in \stab_S(L_x^U)$ such that $y=sxt$.  It follows that
  $sx\R^U x$ and $xt \L^U x$, and so, by Corollary~\ref{cor-greens-simple},
  $sx\L^Sx$ and $xt\R^S x$.  Thus $sx\in L_x^S\cap R_x^U$, and so, from
  Proposition~\ref{prop-main-1.5}(c), $(sxx')\r{x}\in {}_{x}S$ and
  $((sxx')\r{x})\Psi= (x'sxx'x)\l{x}=(x'sx)\l{x}\in ({}_{x}S)\Psi$.  If
  $c\in\mathcal{C}$ is such that $(c)\l{x}$ is the representative of the left
  coset of $S_{x}\cap({}_{x}S)\Psi$ containing $(x'sx)\l{x}$, then
  $(x'sxg)\l{x}=(x'sx)\l{x}\cdot (g)\l{x}=(c)\l{x}$ for some $g\in
  \stab_S(L_x^U)$ such that $(g)\l{x}\in S_{x}\cap ({}_xS)\Psi$.  Thus, by
  Lemma~\ref{lem-free}, $xx'sxg=xc$. But $sx\in R_x^U$ implies that $xx'sx= sx$
  and so $sxg=xc$.  Since $xg \L^U x$, it follows from
  Corollary~\ref{cor-greens-simple}(a) that $xg\R^S x$ and so $sxg \R^S sx$.
  But $xt\R^S x$ and so $y = sxt \R^S sx \R^S sxg=xc$. 

  If $y\D^S x$ is arbitrary, then $R_{y}^U \sim R_{x}^U$, by
  Proposition~\ref{prop-D-membership}, and so there exists $i\in \{1,\ldots,
  m\}$ such that $R_{y}^U = u_i\cdot R_{x}^U$. By
  Proposition~\ref{prop-technical}(a), there exists $\ov{u}_i\in S^1$ such that
  $u_i\ov{u_i}y=y$ and $\ov{u_i}\cdot R_{y}^U = R_{x}^U$.  Again by
  Proposition~\ref{prop-D-membership}, $L_{x}^U \sim L_{y}^U$ and so
  there exists $v\in S^1$ such that $L_{y}^U\cdot v = L_{x}^U$. It
  follows that $y \D^S \ov{u}_iy v$ by Lemma~\ref{lem-greens-simple}(c).  From
  Lemma~\ref{lem-greens-simple}(a), since $L_{y}^U \sim L_{x}^U =
  L_{yv}^U$, it follows that $y\R^S yv$ and so $R_{yv}^U = R_{y}^U$. Thus
  $R_{\ov{u}_iyv}^U = \ov{u}_i \cdot R_{yv}^U = \ov{u}_i\cdot R_{y}^U =
  R_{x}^U$ and so $\ov{u}_iyv\R^U x$. Dually, from
  Lemma~\ref{lem-greens-simple}(b), $\ov{u}_iyv\L^U x$ and so $\ov{u}_iyv\in
  D_x^S \cap H_x^U$.  Hence there exists $c\in \mathcal{C}$ such that
  $\ov{u}_iyv\R^S xc$ and so $y\R^S yv= u_i\ov{u}_iyv \R^S u_ixc$, as required.

  Suppose there exist $i,j\in \{1,2,\ldots, m\}$ and $c,d\in \mathcal{C}$ such
  that $u_ixc\R^S u_jxd$. Then, since $xc\R^U x\R^U xd$ (from the first
  paragraph), it follows that $R_{xc}^U = R_{x}^U = R_{xd}^U$. Thus
  $$u_i\cdot R_{x}^U = u_i\cdot R_{xc}^U = R_{u_ixc}^U = R_{u_jxd}^U = 
    u_j\cdot R_{xd}^U = u_j \cdot R_{x}^U$$
  and, by the minimality of $\{u_1, \ldots, u_m\}$, it follows that $u_i =
  u_j$ and $i = j$. By the analogue of
  Proposition~\ref{prop-technical}(a), there exists $\ov{u_i}$ such that
  $\ov{u_i}u_ixc = xc$ and $\ov{u_i}u_ixd = xd$. Hence since $u_ixc\R^S u_jxd$
  and $\R^S$ is a left congruence, it follows that $xc\R^S xd$.  If $xc = xd$,
  then, by Lemma~\ref{lem-free}, $(c)\mu_x=(d)\mu_x$, and by the minimality of
  $\mathcal{C}$, $c = d$. Suppose that $xc \not = xd$.  Then there exists $y\in
  S$ such that $xcy = xd$.  We showed above that $xc\L^S x \L^S xd$, and so
  $x\L^S xd = xcy \L^S xy$, and, in particular, $y\in \stab_S(L_x^U)$.  From
  Lemma~\ref{lem-free} applied to $xcy = xd$, we deduce that
  $(c)\mu_x(y)\mu_x=(cy)\mu_x = (d)\mu_x$. This implies that
  $((c)\mu_x)^{-1}(d)\mu_x = (y)\mu_x \in S_{x}$. Therefore $(c)\mu_x$ and
  $(d)\mu_x$ are representatives of the same left coset of $S_x\cap
  ({}_xS)\Psi$ in $({}_xS)\Psi$, and again by the minimality of $\mathcal{C}$,
  $c = d$. 
\end{proof}

Next, we give an analogue of Proposition~\ref{prop-D-class-R-class-reps} for
$\L^S$- and $\H^S$-class representatives. 

\begin{prop}
  \label{prop-D-class-L-class-reps}
  Suppose that $x\in S$ and there is $x'\in U$ with $xx'x=x$ and that:
  \begin{enumerate}[\rm (a)]
    \item $\mathcal{C}$ is a minimal
      subset of $\stab_{S}(L_x^U)$ such that $\set{(c)\l{x}}{c\in \mathcal{C}}$
      is a right transversal of $S_{x}\cap({}_{x}S)\Psi$ in $S_x$ where
      $\Psi:{}_{x}S\to U_{x}$, defined by $((s)\r{x})\Psi= (x'sx)\l{x}$, is the
      embedding from Proposition~\ref{prop-main-2}(a);
    \item $\{v_1, \ldots, v_m\}$ is a minimal 
      subset of $S^1$ such that 
      $\set{R_{x}^U\cdot v_i}{1\leq i\leq m}$ equals the s.c.c.\ of 
      $L_{x}^U$ under the right action of $S$ on $\set{L_{x}^U}{x\in S}$.
  \end{enumerate}
  Then $\set{xcv_i}{c\in \mathcal{C},\ 1\leq i\leq m}$ is a minimal
  set of $\H^S$-class representatives of $R_x^S$, and hence a minimal
  set of $\L^S$-class representatives for $D_x^S$.
\end{prop}
\begin{proof}
  It follows from part (b) and Lemma~\ref{lem-greens-simple}(a) that $xcv_i \R^S
  xc \R^S x$, and so it suffices to show that $\set{xcv_i}{c\in \mathcal{C},\
  1\leq i\leq m}$ is a set of $\L^S$-class representatives for $D_x^S$.
  The proof is somewhat similar to that of
  Proposition~\ref{prop-D-class-R-class-reps}, and so we will omit some details. 

  If $y\in D_x^S\cap H_x^U$, then we will show that there is $c\in \mathcal{C}$
  such that $y\L^S xc$. By Proposition~\ref{prop-D-class-elements-1}, there exist
  $s\in \stab_S(R_x^U)$ and $t\in \stab_S(L_x^U)$ such that $y = sxt$. As in the
  proof of Proposition~\ref{prop-D-class-R-class-reps}, it follows that $sx\L^Sx$
  and $xt\R^Sx$. Thus $xt\in L_x^U\cap R_x^S$ and so $(x'xt)\l{x}\in S_x$. If
  $c\in \mathcal{C}$ is such that $(c)\l{x}$ is the representative of the right
  coset containing $(x'xt)\l{x}$, then there exists $g\in \stab_S(R_x^U)$ such
  that $(x'gx)\l{x}\in S_x\cap ({}_xS)\Psi$ and $(x'gxt)\l{x} = (x'gxx'xt)\l{x}=
  (c)\l{x}$. Hence, by Lemma~\ref{lem-free}, $xx'gxt=xc$. But $xt\R^Sx$ and
  $g\in \stab_S(R_x^U)$ and so $R_{gxt}^U = g \cdot R_{xt}^U= g\cdot R_{x}^U =
  R_{x}^U$, which implies that $gxt\R^U x$. Since $xx'$ is a left identity for
  $R_x^U$, $xx'gxt=gxt$, and so $gxt=xc$. Since $gx\R^U x$, from
  Corollary~\ref{cor-greens-simple}(b), $gx\L^S x$ and so $xc=gxt\L^Sxt$.
  Therefore $y=sxt\L^Sxt\L^Sxc$, as required. 

  The proof that an arbitrary $y\in D_x^S$ is $\L^S$-related to $xcv_i$ for
  some $i$ and that $(xcv_i,xdv_j)\not\in\L^S$ if $i\not=j$ or $c\not =d$, is
  directly analogous to the final part of the proof of
  Proposition~\ref{prop-D-class-R-class-reps}, and so we omit it. 
\end{proof}


\section{Specific classes of semigroups}
\label{section-special-cases}

In this section, we show how the results in Section~\ref{section-paradigm} can
be efficiently applied to transformation, partial permutation, matrix, and
partition semigroups; and also to subsemigroups of finite regular Rees
$0$-matrix semigroups.  More precisely, suppose that  $U$ is any of the full
transformation monoid, the symmetric inverse monoid, the general linear monoid
over any finite field, the partition monoid, or a finite regular Rees $0$-matrix
semigroup (the definitions of these semigroups can be found below) and that $S$
is any subsemigroup of $U$. 

In general, computing with the action of $S$ on $U/\L^U$ defined in
\eqref{equation-action-on-L} directly has the disadvantage that the points
acted on are $\L^U$-classes, which are relatively complex to represent on a
computer, and for the purposes of checking equality, and hence for determining
the value of a point under the action.  
Instead, we represent each $\L^U$-class by a point in some convenient set
$\Omega$ and define an isomorphic action on $\Omega$. 

Let $\Omega$ be any set such that $|\Omega| = |U/\L^U|$ and let $\lambda: U\to
\Omega$ be a surjective function such that $\ker(\lambda)=
\L^U$. We define a right action of $S$ on $\Omega$ by 
$$(x)\lambda \cdot s = (xs)\lambda$$ 
for all $x\in U$ and $s\in S ^ 1$.
This action is well-defined since $\L^U$ is a right congruence
on $S$.  It follows immediately from the definition that $\lambda$ induces an
isomorphism of the action of $S$ on $U/\L^U$ (defined in
\eqref{equation-action-on-L}) and the action of $S$ on $\Omega$. 

Clearly, any statement about the action of $S$ defined in
\eqref{equation-action-on-L} can be replaced with an equivalent statement about
the action of $S$ on $\Omega$. Furthermore, any statement about the action of
$S$ defined in \eqref{equation-action-on-U} within a strongly connected
component can be replaced with an equivalent statement about the action of $S$
within a strongly connected component of $\Omega$. 

As an \textit{aide-m\'emoire}, we will write
$(U)\lambda$ or $(S)\lambda$ instead of $\Omega$.  We also write $\rho$ to
denote the analogue of $\lambda$ for left actions. More precisely, we suppose
that $S$ has a left action on a set $(U)\rho$, the kernel of this action is
$\R^U$, and there is a homomorphism $\rho: U \to (U)\rho$ of the actions of $S$
on $U$ by left multiplication and of $S$ on $(U)\rho$.  

In this section, we define the functions $\lambda: U\to (U)\lambda$ and $\rho:
U \to (U)\rho$ where $U$ is any of the monoids mentioned above. The functions
$\lambda$ and $\rho$ are defined so that if $S$ is any subsemigroup of $U$,
then we can compute with the action of $S$ on $(S)\lambda$ and $(S)\rho$
efficiently. For such subsemigroups $S$, we also show how to obtain faithful
representations of relatively small degrees of the stabilisers of $\L^U$- and
$\R^U$-classes under the action of $S$.


\subsection{Transformation semigroups}\label{subsection-trans}

Let $n\in\N$ and write $\bn=\{1,\ldots,n\}$.  As already stated, a
\textit{transformation} of $\bn$ is a function from $\bn$ to itself, and the
\textit{full transformation monoid of degree $n$}, denoted $T_n$, is the monoid
of all transformations on $\bn$ under composition.  We refer to subsemigroups
of $T_n$ as \textit{transformation semigroups of degree $n$}.  It is well-known
that the full transformation monoid is regular; see \cite[Exercise
2.6.15]{Howie1995aa}.  Hence, it is possible to apply the results from
Section~\ref{section-paradigm} to any transformation semigroup~$S$. 

Let $f\in T_n$ be arbitrary. Then the \textit{image} of $f$ is defined to be
$$\im(f)=\set{(i)f}{i\in\bn}\subseteq\bn$$
and the \textit{kernel} of $f$ is defined by
$$\ker(f)=\set{(i,j)}{(i)f=(j)f}\subseteq \bn\times \bn.$$
The kernel of a transformation is an equivalence relation, and every
equivalence relation on $\bn$ is the kernel of some transformation on $\bn$. We
will denote by $\K$ the set of all equivalence relations on $\bn$.  The
\textit{kernel classes} of a transformation $f\in T_n$, are just the
equivalence classes of the equivalence relation $\ker(f)$. 

The following well-known result characterises the Green's relations on the full
transformation monoid. 


\begin{prop}[Exercise 2.6.16 in \cite{Howie1995aa}.]
  \label{prop-greens-full-trans}
  Let $n\in \N$ and let $f,g\in T_n$.
  Then the following hold:
  \begin{enumerate}[\rm (a)]
    \item $f\L^{T_n}g$ if and only if $\im(f)=\im(g)$;
    \item $f\R^{T_n}g$ if and only if $\ker(f)=\ker(g)$;
    \item $f\D^{T_n}g$ if and only if $|\im(f)|=|\im(g)|$.
  \end{enumerate}
\end{prop}


\begin{prop}
  \label{prop-transformation}
  Let $S$ be an arbitrary transformation semigroup of degree $n\in\N$. Then:
  \begin{enumerate}[\rm (a)]
    \item $\lambda: T_n \to \P(\bn)$ defined by $(x)\lambda = \im(x)$
      is a homomorphism of the actions of $S$ on $T_n$ by right multiplication
      and the natural action on $\P(\bn)$, and $\ker(\lambda) = \L^{T_n}$;
    \item if $L$ is any $\L$-class of $T_n$, then $S_{L}$ acts faithfully on 
      $\im(x)$ for each $x\in L$;
    \item $\rho: T_n \to \K$ defined by $(x)\rho = \ker(x)$ 
      is a homomorphism of the actions of $S$ on $T_n$ by left multiplication
      and the left action of $S$ on $\K$ defined by 
      $$x\cdot K=\ker(xy) \qquad\text{where }y\in T_n,\ \ker(y)=K,$$
      and $\ker(\rho)=\R^U$;
    \item if $R$ is any $\R$-class of $T_n$, then ${}_RS$ acts faithfully on the
      set of kernel classes of $\ker(x)$ for each $x\in R$.
  \end{enumerate}
\end{prop}
\begin{proof} 
  We will prove only parts (a) and (b); the proofs of parts (c) and (d)
  are analogous. \smallskip
  
  \noindent\textbf{(a).}
  It follows from Proposition~\ref{prop-greens-full-trans}
  that $\ker(\lambda) = \L^{T_n}$. 
  If $x\in T_n$ and $s\in S$ are arbitrary, then 
  $(xs)\lambda = \im(xs) = \im(x)\cdot s = (x)\lambda\cdot s$
  and so $\lambda$ is a homomorphism of the actions in part (a).
  \smallskip

  \noindent\textbf{(b).}
  Let $x\in L$ and let $\zeta:S_L\to S_{\im(x)}$ be defined by
  $(s|_L)\zeta=s|_{\im(x)}$ where
  the action of $s|_{\im(x)}$ on $\im(x)$ (on the right) is defined by:
  $i\cdot (s|_{\im(x)})=(i)s,$
  for all $i\in \im(x)$.
  Let $s\in \stab_S(L)$ be arbitrary. Then $xs\in L$ and so
  $\im(xs)=\im(x)$, and, in particular, $s$ acts on $\im(x)$. Thus $\zeta$ is
  well-defined. It is routine to verify that $\zeta$ is a homomorphism.
  From the definition of $\zeta$, $s, t\in S$ have the same action on $\im(x)$ if
  and only if $xs=xt$. But, by Lemma~\ref{lem-free}, $xs = xt$ if and only if
  $s|_L=t|_L$ and the action of $S_{L}$ on $\im(x)$ is faithful.
\end{proof}\smallskip


\subsection{Partial permutation semigroups and inverse semigroups}
\label{subsection-pperm}

A \textit{partial permutation} on $\bn=\{1,\ldots, n\}$ is an injective function
from a subset of $\bn$ to another subset of equal cardinality.  The
\textit{symmetric inverse monoid of degree $n$}, denoted $I_n$, is the monoid of
all partial permutations on $\bn$ under composition (as binary relations).  We
refer to subsemigroups of $I_n$ as \textit{partial permutation semigroups of
degree $n$}.  A semigroup $U$ is called \textit{inverse} if for every $x\in U$
there exists a unique $y\in U$ such that $xyx=x$ and $yxy=y$.  Every inverse
semigroup is isomorphic to an inverse subsemigroup of a symmetric inverse monoid
by the Vagner-Preston Theorem; see \cite[Theorem 5.1.7]{Howie1995aa}.  Since
every inverse semigroup is regular, we may apply the results from
Section~\ref{section-paradigm} to arbitrary subsemigroups of the
symmetric inverse monoid. We will give an analogue of
Proposition~\ref{prop-transformation} for subsemigroups of any symmetric inverse
monoid over a finite set, for which we require a description of the Green's
relations in $I_n$.  

Let $f\in I_n$ be arbitrary. Then the \textit{domain} of $f$ is defined to be
$$\dom(f)=\set{i\in \bn}{(i)f\ \text{is defined}}\subseteq \bn$$
and the \textit{image} of $f$ is
$$\im(f)=\set{(i)f}{i\in \dom(f)}\subseteq \bn.$$
The \textit{inverse} of $f$ is the unique partial permutation $f^{-1}$ with the
property that $ff^{-1}f=f$ and $f^{-1}ff^{-1}=f^{-1}$; note that $f^{-1}$
coincides with the usual inverse mapping $\im(f)\to\dom(f)$.


\begin{prop}[Exercise 5.11.2 in \cite{Howie1995aa}.]
  \label{prop-greens-symmetric-inverse}
  Let $n\in \N$ and let $f,g\in I_n$ be arbitrary. Then the following hold:
  \begin{enumerate}[\rm (a)]
    \item $f\L^{I_n}g$ if and only if $\im(f)=\im(g)$;
    \item $f\R^{I_n}g$ if and only if $\dom(f)=\dom(g)$;
    \item $f\D^{I_n}g$ if and only if $|\im(f)|=|\im(g)|$.
  \end{enumerate}
\end{prop}


\begin{prop}
  \label{prop-partial-perms}
  Let $S$ be an arbitrary partial permutation semigroup of degree $n \in \N$. Then:
  \begin{enumerate}[\rm (a)]
    \item $\lambda: I_n \to \P(\bn)$ defined by $(x)\lambda = \im(x)$
      is a homomorphism of the actions of $S$ on $I_n$ by right multiplication,
      and the right action on $\P(\bn)$ defined by 
      $$A\cdot x=\set{(a)x}{a\in A\cap \dom(x)} \qquad\text{for
      }A\in\P(\bn)\text{ and }x\in I_n$$
      and $\ker(\lambda)=\L^{I_n}$;
    \item if $L$ is any $\L$-class of $I_n$, then $(I_n)_{L}$ acts faithfully on
      the right of $\im(x)$ for each $x\in L$;
    \item $\rho: I_n \to \P(\bn)$ defined by $(x)\rho = \dom(x)$
      is a homomorphism of the actions of $S$ on $I_n$ by left multiplication,
      and the left action on $\P(\bn)$ defined by 
      $$x\cdot A=\set{(a)x^{-1}}{a\in A\cap \im(x)} \qquad\text{for
      }A\in\P(\bn)\text{ and }x\in I_n;$$
      and $\ker(\rho)=\R^{I_n}$;
    \item if $R$ is any $\R$-class of $I_n$, then ${}_RS$ acts faithfully on
      the left of $\dom(x)$ for each $x\in R$.
  \end{enumerate}
\end{prop}
\begin{proof} 
  The proof of this proposition is very similar to that of
  Proposition~\ref{prop-transformation} and is omitted.
\end{proof}


\subsection{Matrix semigroups}
Let $R$ be a finite field, let $n\in\N$, and let $M_n(R)$ denote the
monoid of $n\times n$ matrices with entries in $R$ (under the usual matrix
multiplication). The monoid $M_n(R)$ is called a \textit{general linear monoid}.
In this paper, a \textit{matrix semigroup} is a subsemigroup of some general
linear monoid. It is well-known that $M_n(R)$ is a regular semigroup
\cite[Lemma 2.1]{Okninski1998aa}. 

If $\alpha\in M_n(R)$ is arbitrary, then denote by $r(\alpha)$ the \textit{row
space} of $\alpha$ (i.e.~the subspace of the $n$-dimensional vector space over
$R$ spanned by the rows of $\alpha$).  We denote the \textit{dimension} of
$r(\alpha)$ by $\dim(r(\alpha))$.  The notion of a \textit{column space} and its
dimension are defined dually.  We denote the column space of $\alpha\in M_n(R)$
by $c(\alpha)$. 



\begin{prop}[Lemma 2.1 in \cite{Okninski1998aa}]
  \label{prop-greens-general-linear}
  Let $R$ be a finite field, let $n\in \N$, and let
  $\alpha,\beta\in M_n(R)$ be arbitrary. Then the following hold:
  \begin{enumerate}[\rm (a)]
    \item $\alpha\L^{M_n(R)}\beta$ if and only if
      $r(\alpha)=r(\beta)$;
    \item $\alpha\R^{M_n(R)}\beta$ if and only if
      $c(\alpha)=c(\beta)$;
    \item $\alpha\D^{M_n(R)}\beta$ if and only if
      $\dim(r(\alpha))=\dim(r(\beta))$.
  \end{enumerate}
\end{prop}


\begin{prop}
  \label{prop-matrix-semigroups}
  Let $R$ be a finite field, let $n\in \N$, and let $S$ be
  an arbitrary subsemigroup of the general linear monoid $M_n(R)$. 
  Then the following hold:
  \begin{enumerate}[\rm (a)]
    \item if $\Omega$ denotes the collection of subspaces of $R^n$ as row vectors, 
      then $\lambda: M_n(R) \to \Omega$ defined by $(\alpha)\lambda = r(\alpha)$
      is a homomorphism of the actions of $S$ on $M_n(R)$ by right multiplication,
      and the action on $\Omega$ by right multiplication, and $\ker(\lambda) =
      \L^{M_n(R)}$;
    \item if $L$ is any $\L$-class of $M_n(R)$, then $S_{L}$ acts faithfully on 
      $r(\alpha)$ for each $\alpha\in L$;
    \item if $\Omega$ denotes the collection of subspaces of $R^n$ as column
      vectors, then $\rho: M_n(R) \to \Omega$ defined by $(\alpha)\rho =
      c(\alpha)$ is a homomorphism of the actions of $S$ on $M_n(R)$ by left
      multiplication, and the action on $\Omega$ by left multiplication, 
      and $\ker(\rho) = \R^{M_n(R)}$;
    \item if $R$ is any $\R$-class of $M_n(R)$, then ${}_RS$ acts faithfully on 
      $c(\alpha)$ for each $\alpha\in R$.
  \end{enumerate}
\end{prop}
\begin{proof} 
  We will prove only parts (a) and (b); the proofs of parts (c) and (d) follow by
  analogous arguments. We will write $L_{\alpha}$ to mean the $\L$-class of
  $\alpha\in M_n(R)$ in $M_n(R)$ throughout this proof. 
  \smallskip

  \textbf{(a).} 
  It follows from Proposition~\ref{prop-greens-general-linear}(a) that 
  $(\alpha)\lambda = (\beta)\lambda$ if and only if $\alpha\L^{M_n(R)}\beta$,
  and so $\ker(\lambda)= \L^{M_n(R)}$.  We also have 
  $$(\alpha\beta)\lambda = r(\alpha\beta)=r(\alpha)\cdot
    \beta=(\alpha)\lambda\cdot \beta$$
  for all $\alpha\in M_n(R)$ and $\beta\in S$, and so $\lambda$ is a
  homomorphism of actions. 
  \smallskip

  \textbf{(b).} Let $L$ be any $\L$-class in $M_n(R)$, let $\alpha\in L$ and let
  $\beta,\gamma\in \stab_{S}(L)$. Then $\alpha\beta\in L$ and so $\beta$ acts on
  $r(\alpha)$ by right multiplication. 
  By Lemma~\ref{lem-free}, it follows that $\beta$ and $\gamma$ have equal
  action on $r(\alpha)$ if and only if $\alpha\beta=\alpha\gamma$ if and only if
  $\alpha|_L=\beta|_L$.
\end{proof}\smallskip


\subsection{Subsemigroups of a Rees 0-matrix semigroup}

In this section, we describe how the results from Section~\ref{section-paradigm}
can be applied to subsemigroups of a Rees 0-matrix semigroup.  We start by
recalling the relevant definitions. 

Let $T$ be a semigroup, let $0$ be an element not in $T$, let $I$ and $J$ be
sets, and let $P=(p_{j,i})_{j\in J,i\in I}$ be a $|J|\times |I|$ matrix with
entries from $T\cup \{0\}$. Then the \textit{Rees $0$-matrix semigroup $\M^0[T;I,
J; P]$} is the set $(I\times T\times J)\cup \{0\}$ with multiplication defined
by 

\begin{equation*}
0x=x0=0 \text{\ for all $x\in\M^0[T;I, J; P]$} \qquad\text{and}\qquad
(i,g,j)(k,h,l)=
  \begin{cases}
    (i,g p_{j,k}h,l)&\text{if } p_{j,k}\neq 0\\
    0&\text{if } p_{j,k}= 0.
  \end{cases}
\end{equation*}

A semigroup $U$ with a zero element $0$ is $0$-\textit{simple} if $U$ and $\{0\}$
are its only ideals.


\begin{thm}[Theorem 3.2.3 in \cite{Howie1995aa} or Theorem A.4.15 in
  \cite{Rhodes2009aa}]
  A finite semigroup $U$ is $0$-simple if and only if it is isomorphic to a Rees
  0-matrix semigroup $\M^0[G;I,J;P]$, where $G$ is a group, and $P$ is regular, in
  the sense that every row and every column contains at least one non-zero entry.
\end{thm}


Green's relations of a regular Rees $0$-matrix semigroup are described in the
following proposition.


\begin{prop}\label{prop-rees-greens}
  Let $U=\M^0[G;I,J;P]$ be a finite Rees $0$-matrix semigroup where $G$ is a
  group and $P$ is regular. Then the following hold for all $x,y\in U$:
  \begin{enumerate}[\rm (a)]
    \item $x\L^U y$ if and only if $x,y\in I\times G\times \{j\}$ for some $j\in
      J$ or $x=y=0$.  
    \item $x\R^U y$ if and only if $x,y\in \{i\}\times G\times J$ for some $i\in
      I$ or $x=y=0$;
  \end{enumerate}
\end{prop}


Obviously, we do not require any theory beyond that given above to compute with
Rees $0$-matrix semigroups, since their size, elements, and in the case that
they are regular, their Green's structure and maximal subgroups too, are part
of their definition.  However, it might be that we would like to compute with a
proper subsemigroup of a Rees $0$-matrix semigroup. Several computational
problems for arbitrary finite semigroups can be reduced, in part, to problems
for associated Rees $0$-matrix semigroups (the principal factors of certain
$\D$-classes). For example, this is the case for finding the automorphism group
\cite{Araujo2010aa}, minimal (idempotent) generating sets \cite{East2015aa,
Gray2005}, or the maximal subsemigroups of a finite semigroup.  In the latter
example, we may wish to determine the structure of the maximal subsemigroups,
which are not necessarily Rees $0$-matrix semigroups themselves.  In the
absence of a method to find a convenient representation of a subsemigroup of a
Rees $0$-matrix semigroup, as, for example, a transformation semigroup, we
would have to compute directly with the subsemigroup. 


\begin{prop}
  \label{prop-rees-matrix}
  Let $S$ be an arbitrary subsemigroup of a finite regular Rees $0$-matrix
  semigroup $U=\M^0[G;I,J;P]$ over a permutation group $G$ acting faithfully on
  $\bn$ for some $n\in \N$. Then:
  \begin{enumerate}[\rm (a)]
    \item $\lambda:U\to J\cup \{0\}$ defined by $(i,g,j)\lambda=j$ and
      $(0)\lambda=0$ is a homomorphism of the actions of $S$ on $U$ by right
      multiplication, and the right action of $S$ on $J\cup\{0\}$ defined by 
      \begin{equation*}
        0\cdot (i,g,j)=0\cdot 0=0=k\cdot 0 \AND
        k\cdot (i,g,j)=
        \begin{cases}
          j&\text{if } p_{k,i}\not=0\\
          0&\text{if } p_{k,i}=0
        \end{cases}
      \end{equation*}
      for all $k\in J$, and $\ker(\lambda) = \L^U$;
    \item if $L$ is any non-zero $\L$-class of $U$, then the action of $S_{L}$
      on $\bn$ defined by 
      $$m\cdot (i,g,j)|_L=(m) p_{j,i}g\quad\text{ for all }\quad m\in
      \bn$$
      is faithful;
    \item $\rho: U\to I\cup \{0\}$ defined by $(i,g,j)\rho=i$ and
      $(0)\rho=0$ is a homomorphism of the actions of $S$ on $U$ by left
      multiplication, and the left action of $S$ on $I\cup\{0\}$ defined by 
      \begin{equation*}
        (i,g,j)\cdot 0=0\cdot 0=0=0\cdot k\AND        
        (i,g,j)\cdot k=
        \begin{cases}
          i&\text{if } p_{j,k}\not=0\\
          0&\text{if } p_{j,k}=0
        \end{cases}
      \end{equation*}
      for all $k\in I$, and $\ker(\rho) = \R^U$;
    \item if $R$ is any non-zero $\R$-class of $U$, then the action of ${}_RS$ on 
      $\bn$ defined by 
      $${}_R|(i,g,j)\cdot m=(m) g^{-1}p_{j,i}^{-1}\quad\text{ for all }\quad m\in
      \bn$$
      is faithful.
  \end{enumerate}
\end{prop}
\begin{proof} 
  We prove only parts (a) and (b); parts (c) and (d) follow by
  analogous arguments.
  \smallskip

  \noindent\textbf{(a).} 
  It follows by Proposition~\ref{prop-rees-greens}(a) that
  $(x)\lambda=(y)\lambda$ if and only if $x\L^U y$, for each $x,y\in U$, and so
  the kernel of $\lambda$ is $\L^U$.  We will show that $\lambda$ is a
  homomorphism of actions. 
  
  Let  $x\in U$ and $s\in S$ be arbitrary. We must show that
  $(xs)\lambda=(x)\lambda\cdot s$.  If $x=0$ or $s=0$, then
  $(xs)\lambda=(0)\lambda=0=(x)\lambda\cdot s$.  Suppose that $x=(i,g,j)\in
  U\setminus \{0\}$ and $s=(k,h,l)\in S\setminus \{0\}$. If $p_{j,k}=0$, then
  $xs=0$ and so 
  $$(xs)\lambda=(0)\lambda=0=j\cdot (k,h,l)=(x)\lambda\cdot s.$$
  If $p_{j,k}\not=0$, then 
  $$(xs)\lambda=l=j\cdot (k,h,l) = (x)\lambda\cdot s.$$
  \smallskip
  
  \noindent\textbf{(b).}
  Let $x=(i,g,j)\in U\setminus \{0\}$ and let $L=L_x^U=\set{(i', g', j)}{i' \in
  I,\ g' \in G}$. If $(k,h,l)\in
  \stab_S(L)$ is arbitrary, then, since $L\cdot (k,h,l)=L$, it follows that
  $p_{j, k}\not=0$ and $l=j$. It follows that we may define a mapping $\zeta:
  S_L\to G$ by $((k,h,l)|_L)\zeta=p_{j,k}h$. 

  If $(k_1, h_1, j), (k_2, h_2, j)\in \stab_S(L)$, then, by
  Lemma~\ref{lem-free}, it follows that 
  \begin{eqnarray*} 
    (k_1, h_1, j)|_L=(k_2, h_2, j)|_L
    &\iff&(i, g, j)(k_1, h_1, j)=(i, g, j)(k_2, h_2, j)\\
    &\iff& (i, g p_{j,k_1}h_1, j)= (i, g p_{j,k_2}h_2, j)\\
    &\iff&p_{j,k_1}h_1=p_{j, k_2}h_2. 
  \end{eqnarray*}
  Hence, $\zeta$ is well-defined and injective. 

  To show that $\zeta$ is a homomorphism, suppose that 
  $(k_1, h_1, j), (k_2, h_2, j)\in \stab_S(L)$. Then 
  \begin{eqnarray*}
    ((k_1, h_1, j)|_L(k_2, h_2, j)|_{L})\zeta&=&(((k_1, h_1, j)(k_2, h_2,
    j))|_L)\zeta\\
    &=&((k_1, h_1p_{j,k_2}h_2, j)|_{L})\zeta\\
    &=&p_{j, k_1}h_1p_{j,k_2}h_2\\
    &=&((k_1, h_1, j)|_L)\zeta\cdot ((k_2, h_2, j)|_L)\zeta,
  \end{eqnarray*}
  as required.  
\end{proof}


\subsection{Partition monoids}

Let $n\in\mathbb{N}$, let $\bn=\{1,\ldots,n\}$, and let
$-\bn=\{-1,\ldots,-n\}$.  A \textit{partition} of $\bn\cup-\bn$ is a
set of pairwise disjoint non-empty subsets of $\bn\cup-\bn$ (called
\textit{blocks}) whose union is $\bn\cup-\bn$. If
$i,j\in\bn\cup-\bn$ belong to the same block of a partition $x$,
then we write $(i,j)\in x$.

If $x$ and $y$ are partitions of $\bn\cup-\bn$, then we define the product $xy$
of $x$ and $y$ to be the partition where for $i,j\in\bn$
\begin{enumerate}[(i)] 
  \item  $(i,j)\in xy$ if and only if $(i,j)\in x$ or there exist $a_{1},
    \ldots, a_{2r}\in \bn$, for some $r \geq 1$, such that 
    \begin{equation*}
      (i,-a_{1})\in x, \quad (a_{1},a_{2})\in y,\quad (-a_{2},-a_{3})\in
      x,\quad \ldots, \quad (a_{2r-1}, a_{2r})\in y, \quad (-a_{2r}, j)\in x
    \end{equation*}
  \item $(i, -j)\in xy$ if and only if 
    there exist $a_{1}, \ldots, a_{2r - 1}\in \bn$, for some $r\geq 1$, such
    that 
    $$(i,-a_{1})\in x,\quad (a_{1},a_{2})\in y,\quad (-a_{2},-a_{3})\in x,\quad
    \ldots, \quad (-a_{2r-2},-a_{2r-1})\in x, \quad (a_{2r-1},-j)\in y$$
  \item $(-i, -j)\in xy$ if and only if $(-i, -j)\in y$ or there exist $a_{1},
    \ldots, a_{2r}\in \bn$, for some $r \geq 1$, such that
     $$(-i,a_{1})\in y,\quad (-a_{1},-a_{2})\in x,\quad (a_{2},a_{3})\in y,\quad
     \ldots,\quad (-a_{2r-1}, -a_{2r})\in x,\quad \quad(a_{2r}, -j)\in y$$
\end{enumerate}
for $i,j\in \bn$. 

This product can be shown to be associative, and so the collection of
partitions of $\bn\cup-\bn$ is a monoid; the identity element is the
partition $\left\{\{i,-i\}:i\in\bn\right\}$.  
This monoid is called the \textit{partition monoid} and is denoted $P_n$. 

It can be useful to represent a partition as a graph with vertices $\bn\cup-\bn$
and the minimum number of edges so that the connected components of the graph
correspond to the blocks of the partition. Of course, such a representation is
not unique in general.  An example is given in Figure~\ref{fig-1-partition} for the
partitions:
\begin{eqnarray*}
  x & = & \big\{ \{ 1, -1 \}, \{ 2 \}, \{ 3 \}, \{ 4, -3 \}, \{ 5, 6, -5, -6 \},
  \{ -2, -4 \} \big\}\\
  y & = & \big\{ \{1, 4, -1, -2, -6 \}, \{ 2, 3, 5, -4 \}, \{ 6, -3 \}, \{ -5
  \}\big\}
\end{eqnarray*}
and the product 
\begin{eqnarray*}
  xy & = & \big\{\{ 1, 4, 5, 6, -1, -2, -3, -4, -6 \}, \{ 2 \}, \{ 3 \}, \{ -5
  \} \big\}
\end{eqnarray*}
is shown in Figure~\ref{fig-2-partition-prod}. 

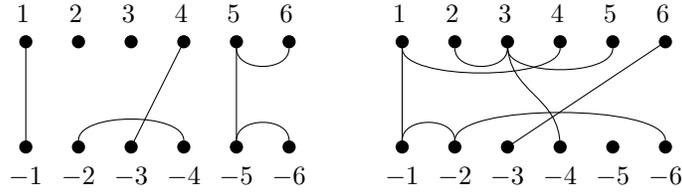
\begin{figure}[ht]
  \begin{center}
    \begin{tikzpicture}[scale = 0.7]

      \fill(1,2)circle(.125);
      \draw(0.95, 2.2) node [above] {{$1$}};
      \fill(1,0)circle(.125);
      \draw(1, -0.2) node [below] {{$-1$}};

      \draw(1,2)--(1,0);

      \fill(2,2)circle(.125);
      \draw(1.95, 2.2) node [above] {{$2$}};


      \fill(3,2)circle(.125);
      \draw(2.95, 2.2) node [above] {{$3$}};


      \fill(4,2)circle(.125);
      \draw(3.95, 2.2) node [above] {{$4$}};
      \fill(3,0)circle(.125);
      \draw(3, -0.2) node [below] {{$-3$}};

      \draw(4,2)--(3,0);

      \fill(5,2)circle(.125);
      \draw(4.95, 2.2) node [above] {{$5$}};
      \fill(6,2)circle(.125);
      \draw(5.95, 2.2) node [above] {{$6$}};
      \fill(5,0)circle(.125);
      \draw(5, -0.2) node [below] {{$-5$}};
      \fill(6,0)circle(.125);
      \draw(6, -0.2) node [below] {{$-6$}};

      \draw(5,1.875) .. controls (5,1.41667) and (6,1.41667) .. (6,1.875);
      \draw(5,0.125) .. controls (5,0.583333) and (6,0.583333) .. (6,0.125);
      \draw(5,2)--(5,0);

      \fill(2,0)circle(.125);
      \draw(2, -0.2) node [below] {{$-2$}};
      \fill(4,0)circle(.125);
      \draw(4, -0.2) node [below] {{$-4$}};

      \draw(2,0.125) .. controls (2,0.666667) and (4,0.666667) .. (4,0.125);
    \end{tikzpicture}
    \qquad
    \begin{tikzpicture}[scale = 0.7]

      \fill(1,2)circle(.125);
      \draw(0.95, 2.2) node [above] {{$1$}};
      \fill(4,2)circle(.125);
      \draw(3.95, 2.2) node [above] {{$4$}};
      \fill(1,0)circle(.125);
      \draw(1, -0.2) node [below] {{$-1$}};
      \fill(2,0)circle(.125);
      \draw(2, -0.2) node [below] {{$-2$}};
      \fill(6,0)circle(.125);
      \draw(6, -0.2) node [below] {{$-6$}};

      \draw(1,1.875) .. controls (1,1.25) and (4,1.25) .. (4,1.875);
      \draw(1,0.125) .. controls (1,0.583333) and (2,0.583333) .. (2,0.125);
      \draw(2,0.125) .. controls (2,0.833333) and (6,0.833333) .. (6,0.125);
      \draw(1,2)--(1,0);

      \fill(2,2)circle(.125);
      \draw(1.95, 2.2) node [above] {{$2$}};
      \fill(3,2)circle(.125);
      \draw(2.95, 2.2) node [above] {{$3$}};
      \fill(5,2)circle(.125);
      \draw(4.95, 2.2) node [above] {{$5$}};
      \fill(4,0)circle(.125);
      \draw(4, -0.2) node [below] {{$-4$}};

      \draw(2,1.875) .. controls (2,1.41667) and (3,1.41667) .. (3,1.875);
      \draw(3,1.875) .. controls (3,1.33333) and (5,1.33333) .. (5,1.875);
      \draw(3,2) .. controls (3,1) and (4,1) .. (4,0);

      \fill(6,2)circle(.125);
      \draw(5.95, 2.2) node [above] {{$6$}};
      \fill(3,0)circle(.125);
      \draw(3, -0.2) node [below] {{$-3$}};

      \draw(6,2)--(3,0);

      \fill(5,0)circle(.125);
      \draw(5, -0.2) node [below] {{$-5$}};

    \end{tikzpicture}
    \caption{Graphical representations of the partitions $x,y\in P_6$.}
    \label{fig-1-partition}

  \end{center}
\end{figure}

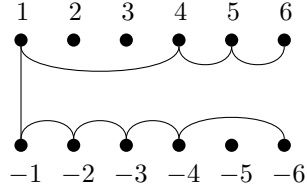
\begin{figure}[ht]
  \begin{center}
    \begin{tikzpicture}[scale = 0.7]

      \fill(1,2)circle(.125);
      \draw(0.95, 2.2) node [above] {{ $1$}};
      \fill(4,2)circle(.125);
      \draw(3.95, 2.2) node [above] {{ $4$}};
      \fill(5,2)circle(.125);
      \draw(4.95, 2.2) node [above] {{ $5$}};
      \fill(6,2)circle(.125);
      \draw(5.95, 2.2) node [above] {{ $6$}};
      \fill(1,0)circle(.125);
      \draw(1, -0.2) node [below] {{ $-1$}};
      \fill(2,0)circle(.125);
      \draw(2, -0.2) node [below] {{ $-2$}};
      \fill(3,0)circle(.125);
      \draw(3, -0.2) node [below] {{ $-3$}};
      \fill(4,0)circle(.125);
      \draw(4, -0.2) node [below] {{ $-4$}};
      \fill(6,0)circle(.125);
      \draw(6, -0.2) node [below] {{ $-6$}};

      \draw(1,1.875) .. controls (1,1.25) and (4,1.25) .. (4,1.875);
      \draw(4,1.875) .. controls (4,1.41667) and (5,1.41667) .. (5,1.875);
      \draw(5,1.875) .. controls (5,1.41667) and (6,1.41667) .. (6,1.875);
      \draw(1,0.125) .. controls (1,0.583333) and (2,0.583333) .. (2,0.125);
      \draw(2,0.125) .. controls (2,0.583333) and (3,0.583333) .. (3,0.125);
      \draw(3,0.125) .. controls (3,0.583333) and (4,0.583333) .. (4,0.125);
      \draw(4,0.125) .. controls (4,0.666667) and (6,0.666667) .. (6,0.125);
      \draw(1,2)--(1,0);

      \fill(2,2)circle(.125);
      \draw(1.95, 2.2) node [above] {{ $2$}};


      \fill(3,2)circle(.125);
      \draw(2.95, 2.2) node [above] {{ $3$}};


      \fill(5,0)circle(.125);
      \draw(5, -0.2) node [below] {{ $-5$}};

      \end{tikzpicture}
    \caption{A graphical representation of the product $xy\in P_6$.}
    \label{fig-2-partition-prod}
  \end{center}
\end{figure}

A block of a partition containing elements of both $\bn$ and -$\bn$ is called a
\textit{transverse block}. 
If $x\in P_n$, then we define $x^*$ to be the partition obtained from
$x$ by replacing $i$ by $-i$ and $-i$ by $i$ in every block of $x$ for all
$i\in\bn$.  It is routine to verify that if $x,y\in P_n$, then  
$$(x^*)^*=x,\quad xx^*x=x,\quad x^*xx^*=x^*,\quad (xy)^*=y^*x^*.$$
In this way, the partition monoid is a \emph{regular $*$-semigroup} in the sense
of \cite{Nordahl1978}.

If $x\in P_n$ is arbitrary, then $xx^*$ and $x^*x$ are idempotents; called the
\textit{projections} of $x$. We will write 
$$\proj(P_n) = \set{xx^*}{x\in P_n} = \set{x^*x}{x\in P_n}.$$
If $B$ is a transverse block of $xx^*$ (or $x^*x$), then $i\in B$ if and only
if $-i\in B$. If $B$ is a non-transverse block of $xx^*$, then
$-B=\set{-b}{b\in B}$ is also a block of $xx^*$.

\begin{prop}[cf. \cite{FitzGerald2011aa,Wilcox2007aa}]
  \label{prop-greens-partn}
  Let $n\in \N$ and let $x,y\in P_n$.
  Then the following hold:
  \begin{enumerate}[\rm (a)]
    \item $x\L^{P_n}y$ if and only if $x^* x = y^* y$;
    \item $x\R^{P_n}y$ if and only if $x x^* = y y ^*$;
    \item $x\D^{P_n}y$ if and only if $x$ and $y$ have the same number of
      transverse blocks.
  \end{enumerate}
\end{prop}

The characterisation in Proposition~\ref{prop-greens-partn} can be used to
define representations of the actions mentioned above. 

\begin{prop}
  \label{prop-partition}
  Let $S$ be an arbitrary subsemigroup of $P_n$. Then:
  \begin{enumerate}[\rm (a)]
    \item 
      $\lambda: P_n\to \proj(P_n)$ defined by $(x)\lambda = x^* x$
      is a homomorphism between the action of $S$ on $P_n$ by right multiplication
      and the right action of $S$ on $\proj(P_n)$ defined by 
      $$ x^* x\cdot y = (xy)^*xy = y^*  x^* x y$$
      and the kernel of $\lambda$ is $\L^{P_n}$;

    \item if $L$ is any $\L$-class of $P_n$, then $S_{L}$ acts faithfully on 
      the transverse blocks of $x^*x$ for each $x\in L$;

    \item $\rho: P_n\to \proj(P_n)$ defined by $(x)\rho = xx^*$
      is a homomorphism between the action of $S$ on $P_n$ by left multiplication
      and the left action of $S$ on $\proj(P_n)$ defined by 
      $$ y \cdot x x ^ * = yx(yx)^* = y  x x^* y^*$$
      and the kernel of $\rho$ is $\R^{P_n}$;

    \item if $R$ is any $\R$-class of $P_n$, then ${}_RS$ acts faithfully on 
      the transverse blocks of $xx^*$ for each $x\in R$.

  \end{enumerate}
\end{prop}
\begin{proof} We prove only parts (a) and (b), since the other parts are dual.
  \smallskip
  
  \noindent\textbf{(a).}
    Let $x\in P_n$ and $s\in S$ be arbitrary.  Then
    $$ (xs)\lambda = (xs)^*xs = s^*(x^*x)s = (x)\lambda\cdot s. $$
    Together with Proposition~\ref{prop-greens-partn}, this completes the proof
    of (a).
  \smallskip
  
  \noindent\textbf{(b).} Let $x\in P_n$ be arbitrary and suppose that $y\in
  \stab_S(L_x^{P_{n}})$. It follows that $x^*x = (x)\lambda = (xy)\lambda =
  (xy)^*xy = y^*x^*xy$.  We denote the intersection of the transverse blocks of
  $x^*x$ with $\bn$ by $B_1, \ldots, B_r$ and we define the binary relation
  $$p_y = \set{(i,j)\in \mathbf{r}\times\mathbf{r}}{\exists k\in B_i,\
  \exists l\in B_j,\ (k,-l)\in x^*xy}.$$ 
  We will show that $p_y$ is a permutation, and that $\zeta:S_{L_x^{P_n}}\to
  \sym(\mathbf{r})$ defined by $(y|_{L_x^{P_n}})\zeta = p_y$ is a monomorphism. 
  
  Seeking a contradiction, assume that there exist $i,j,j'\in \mathbf{r}$
  such that $j\not= j'$ and $(i, j), (i, j')\in p_y$. Then there exist
  $k,k'\in B_i$, $l\in B_{j}$, and $l'\in B_{j'}$ such that $(k,-l),
  (k',-l')\in x^*xy$.  Since $k,k'\in B_i$, it follows that $(k,k')\in x^*x$
  and so $(k,k')\in x^*xy$. Since $x^*xy$ is an equivalence relation, it
  follows that $(-l, -l') \in x^*xy$, which implies that $(-l, -l')\in
  y^*x^*xy = x^*x$, and so $(l,l')\in x^*x$, a contradiction. Hence
  $p_y$ is a function. 

  Since $\bn$ is finite, to show that $p_y$ is a permutation it suffices to
  show that it is surjective.  Suppose that $i\in \mathbf{r}$ and $l \in B_i$
  are arbitrary.  Then $(l, -l)\in x^*x = y^*x^*xy$, and so by part (ii) of the
  definition of the multiplication of the partitions $y^*$ and $x^*xy$, there
  exists $k \in \bn$ such that $(k,-l)\in x^*xy$.  In other words, $k$ belongs
  to a transverse block of $x^*xy$, and hence to a transverse block of
  $x^*x$. Thus there exists $j\in\mathbf{r}$ such that $k\in B_j$ and so
  $(j)p_y= i$, as required. 

  Note that, since $x^*xy$ is an equivalence relation and $p_y$ is a
  permutation, the transverse blocks of $x^*xy$ are of the form
  $B_{i}\times -B_{(i)p_y}$. 

  Next, we show that $\zeta:S_{L_x^{P_n}}\to \sym(\mathbf{r})$ defined by
  $(y|_{L_x^{P_n}})\zeta = p_y$ is a homomorphism.  Let
  $y,z\in\stab_S(L_x^{P_n})$. It suffices to prove that $p_{y}p_{z} = p_{yz}$.
  If $i\in \mathbf{r}$, then there exist $k\in B_i$, $l,l'\in B_{(i)p_y}$, and
  $m\in B_{((i)p_y)p_z}$ such that $(k, -l)\in x^*xy$ and $(l', -m) \in
  x^*xz$. Since $l,l'\in B_{(i)p_y}$ (a transverse block of $x^*x$),
  $(l,-l')\in x^*x$, and so $(k, -m)\in x^*xy\cdot x^*x\cdot x^*xz =  x^* x y x ^*
  xz$.  Since $xy\L^{P_n} x$ and $x^*x$ is a right identity in its $\L^{P_n}$-class,
  it follows that $xyx^*x= xy$ and so $x^* x y x ^* xz = x^*xyz$.  In
  particular, $(k, -m)\in x^*xyz$, and so $(i)p_{yz} = ((i)p_y)p_z$, as
  required.

  It remains to prove that $\zeta$ is injective. Suppose that $y, z\in
  \stab_S(L_x^{P_{n}})$ are such that $p_y = p_z$. It suffices, by
  Lemma~\ref{lem-free}, to show that $xy = xz$. We will prove that
  $x^*xy=x^*xz$ so that $xy = xx ^ *xy = xx ^ *xz = xz$. Since the transverse
  blocks of $x^*xy$ are $B_{i}\times -B_{(i)p_y} = B_{i}\times -B_{(i)p_z}$
  where $i\in \mathbf{r}$, it follows that the transverse blocks of $x^*xy$ and
  $x^*xz$ coincide. Suppose that $(k, l) \in x^*xy$ where neither $k$ nor $l$
  belongs to a transverse block of $x^*xy$. It follows from the form of the
  transverse blocks of $x^*xy$ that neither $k$ nor $l$ belongs to a transverse
  block of $x^*x$. There are two cases to consider: $k, l > 0$ and $k, l < 0$.
  In the first case, by part (i) of the definition of the mulitplication of
  $x^*x$ and $y$, either $(k,l) \in x^*x$ or $k$ and $l$ belong to transverse
  blocks of $x^*x$. Since the latter is not the case, $(k,l)\in x^*x$ and so
  $(k,l)\in x^*xz$.  In the second case, when $k, l < 0$, it follows from part
  (iii) of the definition of the multiplication of $x^*x$ and $y$, that $(k,
  l)\in y$ and so $(k,l) \in y^*x^*xy=z^*x^*xz$.  By part (iii) of the
  definition of the multiplication of $z^*$ and $x^*xz$, either $(k,l)\in
  x^*xz$ or both $k$ and $l$ belong to transverse blocks of $x^*xz$. Since the
  transverse blocks of $x^*xy$ and $x^*xz$ coincide and contain neither $k$ nor
  $l$, it is the case that $(k,l)\in x^*xz$. Thus $x^*xy\subseteq x^*xz$ and,
  by symmetry, $x^*xz \subseteq x^*xy$. Therefore $xy=xz$, as required.
\end{proof}


\section{Algorithms}
\label{section-algorithms}

In this section, we outline some algorithms for computing with semigroups
that utilise the results in Section~\ref{section-paradigm}.  The algorithms
described in this section are implemented in the \GAP \cite{GAP4} package
\Semigroups \cite{Mitchell2016aa} in their full generality, and can currently be
applied to semigroups of transformations, partial permutations, partitions, and
to subsemigroups of regular Rees $0$-matrix semigroups over groups. 

Throughout this section, we will suppose that $U$ is a finite regular
semigroup, and that $S$ is subsemigroup of $U$ generated by  $X=\{x_1, \ldots,
x_m\}\subseteq U$.  The purpose of the algorithms described in this section is
to answer various questions about the structure of the semigroup $S$.

Recall that $U ^ 1$ has an identity $1_U\not\in U$ adjoined regardless of
whether or not $U$ is already a monoid; the analogous statement holds for $S$
where the adjoined identity is denoted $1_S$. We note that it is not necessary
for $1_S$ and $1_U$ to be equal.

Apart from this introduction, this section has 6 subsections.  In
Subsection~\ref{subsection-assume}, we outline some basic operations that we
must be able to compute in order to apply the algorithms later in this section.
We then show how to perform these basic operations in the examples of
semigroups of transformations, partial permutations, matrices, partitions, and
in subsemigroups of a Rees $0$-matrix semigroup in
Subsection~\ref{subsection-comp-prereq}.  In
Subsection~\ref{subsection-components}, we describe how to calculate the
components of the action of a semigroup on a set, and how to use this to obtain
the Schreier generators from Proposition~\ref{prop-technical}(c).  In
Subsection~\ref{subsection-individual}, we describe a data structure for
individual Green's classes of $S$ and give algorithms showing how this data
structure can be used to compute various properties of these classes. In
Subsection~\ref{subsection-global}, we describe algorithms that can be used to
find global properties of $S$, such as its size, $\R$-classes, and so on.  In
Subsection~\ref{section-algorithms-regular}, we give details of how some of the
algorithms in Subsection~\ref{subsection-global} can be optimised when it is
known \textit{a priori} that $S$ is a regular or inverse semigroup. 

At several points in this section it is necessary to be able to determine the
strongly connected components (s.c.c.) of a directed graph. This can be
achieved using Tarjan's \cite{Tarjan1972aa} or Gabow's \cite{Gabow2000aa}
algorithms, for example; see also Sedgwick \cite{Sedgewick1983aa}.

In the algorithms in this section, ``$:=$'' indicates that we are assigning a
value (the right hand side of the expression) to a variable (the left hand
side), while ``$=$'' denotes a comparison of variables.  The symbol ``$\gets$''
is the replacement operator, used to indicated that the value of the variable
on the left hand side is replaced by the value on the right hand side. 


\subsection{Assumptions}\label{subsection-assume}

Suppose that $U$ is a finite regular semigroup and that $S$ is a subsemigroup
of $U$.  As discussed at the start of Section~\ref{section-special-cases}, we
suppose that we have a right action of $S$ on a set $(U)\lambda$ and a
homomorphism $\lambda: U \to (U)\lambda$ of this action and the action of $S$
on $U$ by right multiplication, where $\ker(\lambda) = \L ^ U$.  Furthermore,
for every $x\in U$ we assume that we have a faithful representation $\zeta$ of
the stabiliser $U_{L_x^U}$ and a function $\mu_x:\stab_U(L_x^U) \to
(U_{L_x^U})\zeta$ defined by
\begin{equation*}
  (u)\l{x} = (u|_{L_x^U})\zeta\qquad\text{for all}\qquad u\in U ^ 1;
\end{equation*} 
see \eqref{equation-demux}. We also assume that we have the left handed
analogues $\rho:U\to (U)\rho$ and $\nu_x:\stab_U(R_x^U) \to
({}_{R_x^U}U)\zeta'$ (where $\zeta'$ is any faithful representation of
${}_{R_x^U}U$) of $\lambda$ and $\mu_x$, respectively.  Recall that we write
$$S_x=(\stab_S(L_x^U))\mu_x\quad\text{and}\quad{}_xS = (\stab_S(R_x^U))\nu_x.$$

In order to apply the algorithms described in this section, it is necessary
that certain fundamental computations can be performed. 

\begin{ass} 
  We assume that we can compute the following:
  \begin{enumerate}[\rm (I)]
    \item\label{assumption-I} 
      the product $xy$;
    
    \item\label{assumption-II} 
      the value $(x)\lambda$;
      
    \item\label{lambda-inverse} 
      an element $\ov{s}\in U ^ 1$ such that $xs\ov{s} = x$
      whenever $(x)\lambda\sim (x)\lambda\cdot s = (xs)\lambda$;
    
    \item\label{lambda-perm}
       the value $(x's)\mu_x$, for some choice of $x'\in U$ such that $x x' x =
       x$, whenever $(x)\lambda = (s)\lambda$ and $(x)\rho = (s)\rho$;
  \end{enumerate}
  for all $x,y,s\in U$.  We also require the facility to perform the analogous
  computations involving the functions $\rho$ and the $\nu_x$ for all $x\in S$. 
\end{ass}

We will prove that $\ov{s}\in S ^ 1$ in Assumption \eqref{lambda-inverse} exists.
By the definition of $\lambda$, since $(x)\lambda \sim (xs)\lambda$, it follows
that $L_x^U\sim L_{xs}^U$ under the action of $S$ on $U/\L^U$ defined in
\eqref{equation-action-on-L}. Hence, by Proposition~\ref{prop-technical}(a),
there exists $\ov{s}\in S ^ 1$ such that 
$$L_{xs}^U\cdot \ov{s} = L_x^U\quad \text{and}\quad (s\ov{s})|_{L_x^U} =
\id_{L_x^U}$$
and so $xs\ov{s}=x$.  
Given the orbit graph of $(x)\lambda\cdot S ^ 1$, it is possible to compute
$\ov{s}$ by finding a path from $(xs)\lambda$ to $(x)\lambda$ and using
Algorithm~\ref{algorithm-trace-schreier}. However, this is often more expensive
than computing $\ov{s}\in U$ directly from $s$.  Some details of how to compute
Assumptions \eqref{assumption-I} to \eqref{lambda-perm} in the special cases given in
Section~\ref{section-special-cases} can be found in the next section. 

If $\ov{s}\in U ^ 1$ satisfies Assumption \eqref{lambda-inverse} for some $x\in S$,
then we note that $xs\ov{s} = x$ implies that $xs\ov{s}s = xs$ and so by
Lemma~\ref{lem-free}:
$$(\ov{s}s)|_{L_{xs}^{U}} =\id_{L_{xs}^U}.$$ 

Note that if $s\in \stab_S(R_x^U)$, then $(sx)\rho = (x)\rho$ and, by
Proposition~\ref{prop-main-1.5}(b), $sx\in L_x^S\cap R_x^U$. In particular,
$(sx)\lambda = (x)\lambda$, and so, by Assumption \eqref{lambda-perm}, we can
compute:
\begin{enumerate}[(I)]
  \addtocounter{enumi}{4}
     \item\label{lambda-conj}
       the value $((s)\nu_x)\Psi=(x'sx)\mu_x$ whenever $s\in \stab_S(R_x^U)$.
\end{enumerate}
We will refer to this as Assumption \eqref{lambda-conj}, even though it is not
an assumption.

It will follow from the comments in Subsection~\ref{subsection-comp-prereq}
that the algorithms described in this paper can be applied to any subsemigroup
of the full transformation monoid, the symmetric inverse monoid, the partition
monoid, the general linear monoid, or any subsemigroup of a regular Rees
$0$-matrix semigroup over a group.  However, we would like to stress that the
algorithms in this section apply to any subsemigroup of a finite regular
semigroup. In the worst case, the functions $\lambda: U\to U/\L^U$ and $\rho: U
\to U/\R^U$ defined by $(x)\lambda = L_x^U$ and $(x)\rho = R_x^U$, and the
natural mappings $\mu_x: \stab_S(L_x^U)\to S_{L_x^U}$ and $\nu_x:
\stab_S(R_x^U)\to {}_{R_x^U}S$ for every $x\in U$, fulfil the required
conditions,  although it might be that our algorithms are not very efficient in
this case.


\subsection{Computational prerequisites}\label{subsection-comp-prereq}

In this section, we describe how to perform the computations required in
Assumptions \eqref{assumption-I} to \eqref{lambda-perm} for semigroups of
transformations, partial permutations, partitions, and matrices, and for
subsemigroups of a Rees 0-matrix semigroup.


\subsubsection*{Transformations}

A transformation $x$ can be represented as a tuple $((1)x, \ldots, (n)x)$
where $n$ is the degree of $x$.

\begin{enumerate}[(I)]
  \item The composition $xy$ of transformations $x$ and $y$ is represented by
   $((1)xy, \ldots, (n)xy)$, which can be computed by simple
    substitution in linear time $O(n)$; 
  
  \item The value $(x)\lambda=\im(x)$ can be found by sorting and removing
    duplicates from  $((1)x, \ldots, (n)x)$ with complexity $O(n\log(n))$;
  
  \item If $x, s$ are such that $\im(x)$ and $\im(xs)$ have equal cardinality,
    then the transformation $\ov{s}$ defined by
    \begin{equation*}
      (i)\ov{s}=
      \begin{cases}
        (j)x & \text{if } i = (j)xs\in\im(xs) \\
        i    & \text{if } i\not\in \im(xs)
      \end{cases}
    \end{equation*}
    has the property that $x s\ov{s} = x$
    (finding $\ov{s}$ has complexity $O(n)$);

  \item If $x$ and $s$ are transformations such that $\ker(x) = \ker(s)$ and
    $\im(x) = \im(s)$ and $x'$ is any transformation such that $xx'x = x$,
    then, from Proposition~\ref{prop-transformation}(b), $(x's)\mu_x$ is just
    the restriction $(x's)|_{\im(x)}$ of $x's$ to $\im(x)$, which can be
    determined in $|\im(x)|$ steps from $x$ and $s$.

\end{enumerate}

The analogous calculations can be made in terms of the kernel of a transformation,
and the left actions of $S$; the details are omitted.


\subsubsection*{Partial permutations}

A partial permutation $x$ can be represented as a tuple $((1)x, \ldots, (m)x)$
where $m$ is the largest value where $x$ is defined, and $(i)x=0$ if $x$ is not
defined at $i$. The values required under Assumptions \eqref{assumption-I},
\eqref{assumption-II}, and \eqref{lambda-perm} can be computed for partial
permutations in the same way they were computed for transformations. 

Assumption \eqref{lambda-inverse} is described below:

\begin{enumerate}[(I)]
  \addtocounter{enumi}{2}
  \item if $x, s$ are such that $\im(x)$ and $\im(xs)$ have equal cardinality,
    then the partial permutation $\ov{s}=s^{-1}$ has the property 
    $xs\ov{s} = x$ (finding $s^{-1}$ has complexity $O(n)$).
\end{enumerate}


\subsubsection*{Matrices over finite fields}

The values required in Assumptions~\eqref{assumption-I} to \eqref{lambda-perm}
can be found for matrix semigroups in a similar way as they are found for
transformation and partial permutations, using elementary linear algebra; the
details will appear in \cite{Mitchell2015ab}. 


\subsubsection*{Rees $0$-matrix semigroups}

It should be clear from the definition of a Rees $0$-matrix semigroup, and by
Proposition~\ref{prop-rees-matrix}, how to compute the values in
Assumptions \eqref{assumption-I} and \eqref{assumption-II}.  Assumptions \eqref{lambda-inverse} and \eqref{lambda-perm} are described below:

\begin{enumerate}[(I)]
  \addtocounter{enumi}{2}

  \item 
    if $x = (a, b, i), s = (j, g, k) \in S$ are such that 
    $(x)\lambda = i \sim (s)\lambda = k$, then $p_{i,j}\not=0$.
    So, if $l\in I$ is such that $p_{k,l}\not=0$, then $\ov{s}=(l,
    p_{k,l}^{-1}g^{-1}p_{i,j}^{-1}, i)$ has the property that
    $xs\ov{s} = x$;
 
   \item 
     if $x=(i,g,j)\in S$, then there exist $k\in I$ and $l\in J$ such that
     $p_{j,k}, p_{l,i}\not=0$. One choice for $x'\in U$ is $(k,
     p_{j,k}^{-1}g^{-1}p_{l,i}^{-1}, l)$. 
     So, if $h\in G$ is arbitrary and $s= (i,h,j)$, then $(x)\lambda =
     (s)\lambda$ and $(x)\rho = (s)\rho$.  Hence  
     the action of $(x's)\mu_x$ on $m\in \bn$ (defined in
     Proposition~\ref{prop-rees-matrix}(b)) is given by 
     $$m\cdot (x's)\mu_x=(m) g^{-1}h.$$

\end{enumerate}


\subsubsection*{Partitions}

A partition $x\in P_n$ can be represented as a $2n$-tuple where the first $n$
entries correspond to the indices of the blocks containing $\{1,\ldots, n\}$ and
entries $n+1$ to $2n$ correspond to the indices of the blocks containing
$\{-1,\ldots, -n\}$. For example, the partition $x\in P_6$ shown in
Figure~\ref{fig-1-partition} is represented by $(1,2,3,4,5,5,1,6,4,6,5,5)$.

Given $x\in P_n$ represented as above, it is possible to compute $x^*\in P_n$
in $2n$ steps (linear complexity). This will be used in several of the
assumptions.

\begin{enumerate}[(I)]
  
  \item 
    The composition $xy$ of partitions $x$ and $y$ can be found using a
    variant of the classical Union-Find Algorithm; 
  
  \item 
    The value $x^*x$ can be found using (I);
  
  \item 
    If $x, s\in P_n$ are such that $(x)\lambda= x^*x\sim (xs)\lambda =
    s^*x^*xs$, then $x^*x$ and $s^*x^*xs$ have equal number of transverse
    blocks. A partition $\ov{s}$ with the property that $xs\ov{s} = x$ can then
    be found using a variant of the Union-Find Algorithm;

  \item 
    If $x, s\in P_n$ are such that $(x)\lambda = (s)\lambda$  and $(x)\rho =
    (s)\rho$, then, by Proposition~\ref{prop-partition}(b), $(x^*s)\mu_x$ is a
    permutation of the transverse blocks of $x$, which can be found in linear
    time (complexity $O(n)$). 
  
\end{enumerate}

In practice, it is possible to compute the values required by these assumptions
with somewhat better complexity than that given above. However, this is also
more complicated to describe and so we opted to describe the simpler methods.


\subsection{Components of the action}\label{subsection-components}

We start by describing a procedure for calculating $\alpha \cdot S ^
1=\set{\alpha\cdot s}{s\in S ^ 1}$ or $\Omega\cdot S ^ 1=\set{\alpha\cdot
s}{\alpha\in \Omega,\ s\in S ^ 1}$.  These are essentially the same as the
standard orbit algorithm for a group acting on a set (see, for example,
\cite[Section 4.1]{B.-Eick2004aa}), but without the assumption that $S$ is a
group.  An analogous algorithm can be used for left actions.  We will refer to
$\alpha \cdot S ^ 1$ as the \textit{component of the action} of $S$ containing
$\alpha$. 

In this subsection we present algorithms for computing: the components of an
action of a semigroup $S$; elements of $S$ that act on points in the component
in a specified way; generators for the stabiliser of a set.  Examples of how
these algorithms can applied can be found in Section~\ref{section-examples}.
The algorithms in this subsection and the next are somewhat similar to those
described in~\cite{Linton2002aa}.

Suppose that $X=\{x_1, \ldots, x_m\}$ is a generating set for a monoid $S$
acting on the right on a set $\Omega$. If $\alpha\in \Omega$ and $\alpha\cdot S
= \{\beta_1 = \alpha, \beta_2, \ldots, \beta_n\}$, then the \textit{orbit
graph} of $\alpha\cdot S$ is just the directed graph with vertices $\{1,\ldots,
n\}$ and an edge from $i$ to $g_{i,j}$ labelled with $j$ if $\beta_i\cdot
x_j=\beta_{g_{i,j}}$.  The orbit graph of $\Omega\cdot S$ is defined
analogously. A \textit{Schreier tree} for $\alpha\cdot S$ is just a spanning
tree for the orbit graph with root at $\beta_1$.  Schreier
trees for $\alpha \cdot S$ will be represented using a $2$-dimensional array  
\begin{equation*}
  \begin{array}{llll}
    v_{2}& \ldots &v_{n}\\
    w_{2}& \ldots &w_{n}
  \end{array}     
\end{equation*}
such that $\beta_{v_j}\cdot x_{w_j}=\beta_j$ and $v_j < j$ for all $j > 1$.

The orbit graph of $\Omega\cdot S ^ 1$ may not be connected and so it has a
forest of (not necessarily disjoint) Schreier trees rooted at some elements of
$\Omega$. To simplify things, we may suppose without loss of generality that
there is an $\alpha\in \Omega$ such that $\alpha\cdot S = \Omega\cdot S$. This
can be achieved by adding an artificial $\alpha$ to $\Omega$ (and perhaps some
further points) and defining the action of $S$ on these values so that $\alpha
\cdot S$ contains all of the roots of the Schreier forest for $\Omega \cdot S$.

Note that unlike the orbit graph of a component of a group acting on a set, the
orbit graph of a component of a semigroup acting on a set is, in general, not
strongly connected.  This makes several of the steps required below more
complicated than in the group case. 


\begin{algorithm}
  \caption{Compute a component of an action}
  \label{algorithm-component}
  \begin{algorithmic}[1]

    \item[\textbf{Input:}] $S:=\genset{X}$ where $X:=\{x_1,\ldots, x_m\}$, $S$
      acts on a set $\Omega$ on the right, and $\alpha\in \Omega$

    \item[\textbf{Output:}] $\alpha\cdot S ^ 1$, a Schreier tree for
      $\alpha\cdot S ^ 1$, and the orbit graph of $\alpha\cdot S ^ 1$

    \item $C:=\{\beta_1:=\alpha\}$, $n:=1$
      \Comment{$C$ will equal $\alpha\cdot S ^ 1$ by the end of the algorithm}

      \For{$\beta_i\in C$, $j\in \{1,\ldots, m\}$} 
      \Comment{loop over: values in $C$, the generators}

    \If{$\beta_i\cdot x_j\not\in C$} 

        \State $n:=n+1$, $\beta_n:=\beta_i\cdot x_j$, $C\gets
        C\cup \{\beta_n\}$; 
        \Comment{add $\beta_i\cdot x_j$ to $C$}
        \State $v_n:=i$, $w_n:=j$, $g_{i,j}:=n$ 
        \Comment{update the Schreier tree and orbit graph}

      \ElsIf{$\beta_i\cdot x_j=\beta_r\in C$}

      \State $g_{i,j}:=r$;\Comment{update the orbit graph}

      \EndIf

    \EndFor
    \State \Return $C$, $(v_2,\ldots, v_n, w_2, \ldots, w_n)$,
    $\set{g_{i,j}}{i\in \{1, \ldots, n\},\ j\in \{1,\ldots, m\}}$
  \end{algorithmic}
\end{algorithm}


The Schreier tree for $\alpha\cdot S ^ 1$ produced by
Algorithm~\ref{algorithm-component} can be used to obtain elements $u_i\in S ^
1$ such that $\beta_1\cdot u_i=\beta_i$ for all $i$ using
Algorithm~\ref{algorithm-trace-schreier}. The standard method for a group
acting on a set, such as the procedure {\sc U-Beta} in
\cite[p80]{B.-Eick2004aa}, cannot be used here, due to the non-existence of
inverses in semigroups.  


\begin{algorithm}
  \caption{Trace a Schreier tree}
  \label{algorithm-trace-schreier}
  \begin{algorithmic}[1]
  \item[\textbf{Input:}] a Schreier tree $(v_2,\ldots, v_n, w_2, \ldots, w_n)$ for
    $\alpha\cdot S ^ 1$, and $\beta_i\in \alpha\cdot S ^ 1$

    \item[\textbf{Output:}] $u\in S ^ 1$ such that $\alpha \cdot u=\beta_i$
 
    \item $u:=1_S$, $j:=i$
    
    \While{$j>1$}
      \State $u:=x_{w_j}u$ and $j:=v_j$
    \EndWhile
    \item \Return $u$
  \end{algorithmic} 
\end{algorithm}


The right action of $S$ on $\Omega$ induces an action of $S$ on $\P(\Omega)$.
In Algorithm~\ref{algorithm-schreier-generators-general},
Algorithms~\ref{algorithm-component} and~\ref{algorithm-trace-schreier} are
used to obtain the Schreier generators from Proposition~\ref{prop-technical}(c)
for the stabiliser $S_{\Sigma}$ of a subset $\Sigma$ of $\Omega$ under this
induced action.


\begin{algorithm}
  \caption{Compute Schreier generators for a stabiliser}
  \label{algorithm-schreier-generators-general}
  \begin{algorithmic}[1]
  \item[\textbf{Input:}] the component $\Sigma\cdot S ^ 1$ of $\Sigma\subseteq
    \Omega$ under the action of $S$ on $\P(\Omega)$, a Schreier tree and orbit
    graph for $\Sigma\cdot S ^ 1$

  \item[\textbf{Output:}] Schreier generators $Y$ for the stabiliser $S_{\Sigma}$
    
  \item $Y:=\{\id_{\Sigma}\}$

  \item find the s.c.c.~of $\Sigma$ in $\Sigma\cdot S ^ 1 := \{\Sigma_1:=\Sigma,
   \ldots, \Sigma_r\}$
    
    \For{$i\in \{1,\ldots, r\}$, $j\in \{1,\ldots, m\}$} 
    \Comment{loop over: $\Sigma\cdot S ^ 1$, the generators of $S ^ 1$}
      
    \State set $k:=g_{i,j}$
    \Comment{$g_{i,j}$ is from the orbit graph of $\Sigma\cdot S ^ 1$}
      
      \If{$\Sigma_k\sim \Sigma_1$}
      \Comment{$\Sigma_k$ is in the s.c.c.\ of $\Sigma_1$} 
        
        \State find $u_i, u_k\in S ^ 1$ such that $\Sigma_1\cdot u_i=\Sigma_i$ and
        $\Sigma_1\cdot u_k=\Sigma_k$
        \Comment{Algorithm~\ref{algorithm-trace-schreier}} 
        
        \State find $\ov{u_k}\in U$ such that $\Sigma_k\cdot \ov{u_k}=\Sigma_1$ and
        $(u_k\ov{u_k})|_{\Sigma_1}=\id_{\Sigma_1}$
        \Comment{Proposition~\ref{prop-technical}(a)}
        
        \State $Y\gets Y\cup \{(u_{i}x_{j}\ov{u_k})|_{\Sigma_1}\}$

      \EndIf 
    \EndFor
    \State \Return $Y$
  \end{algorithmic} 
\end{algorithm}

In Algorithm~\ref{algorithm-schreier-generators-special} we give
a more specialised version of
Algorithm~\ref{algorithm-schreier-generators-general}, which we require to find
$(\stab_S(L_x^U))\mu_x = S_x$ where $\mu_x$ is the function defined at the
start of this section. 


\begin{algorithm}
  \caption{Compute Schreier generators for $S_x$}
  \label{algorithm-schreier-generators-special}
  \begin{algorithmic}[1]
  \item[\textbf{Input:}] the component $(x)\lambda\cdot S ^ 1$ of $(x)\lambda$, a
    Schreier tree and orbit graph for $(x)\lambda\cdot S ^ 1$

  \item[\textbf{Output:}] Schreier generators $Y$ for the stabiliser $S_{x}$
  
  \item set $Y:=\varnothing$, $x_1 := x$
  
  \item find the s.c.c.~$\{(x_1)\lambda, \ldots,
    (x_r)\lambda\}$ of $(x)\lambda$ in $(x)\lambda\cdot S ^ 1$
    
    \For{$i\in \{1,\ldots, r\}$, $j\in \{1,\ldots, m\}$}
    \Comment{loop over: the s.c.c.~of $(x)\lambda$, the generators of $S$}
      
    \State set $k:=g_{i,j}$
    \Comment{$g_{i,j}$ is from the orbit graph of $(x)\lambda\cdot S ^ 1$}
      
      \If{$(x_k)\lambda\sim (x_1)\lambda$}
      \Comment{$(x_k)\lambda$ is in the s.c.c.\ of $(x_1)\lambda$} 
        \State find $u_i, u_k\in S ^ 1$ such that $(x_1)\lambda\cdot
        u_i=(x_i)\lambda$ and $(x_1)\lambda\cdot u_k=(x_k)\lambda$ 
        \Comment{Algorithm~\ref{algorithm-trace-schreier}} 
        
        \State find $\ov{u_k}\in U$ such that $xu_k\ov{u_k}=x$
        \Comment{Assumption \eqref{lambda-inverse}}
        
        \State $Y\gets Y\cup \{(x'u_{i}x_{j}\ov{u_k})\mu_x\}$
        \Comment{Assumption \eqref{lambda-perm}}

      \EndIf 
    \EndFor
    \State \Return $Y$
  \end{algorithmic} 
\end{algorithm}


Algorithms~\ref{algorithm-schreier-generators-general}
and~\ref{algorithm-schreier-generators-special} can also be used to find
generators for the stabiliser of any value in the component $\Sigma\cdot S ^ 1$ or
$(x)\lambda \cdot S ^ 1$, if we have a Schreier tree for the s.c.c.~rooted at
that value.  Algorithm~\ref{algorithm-component} returns a Schreier tree for
the entire component (possibly including several strongly connected components)
rooted at the first point in the component.  It is possible to find a Schreier
tree for any s.c.c.~rooted at any value in the component by finding a spanning
tree for the subgraph of the orbit graph the s.c.c.~induces.  Such a spanning
tree can be found in linear time using a depth first search algorithm, for
example.


\subsection{Individual Green's classes}
\label{subsection-individual}

\subsubsection*{Data structures}

By Corollary~\ref{cor-collect}, we can represent the $\R$-class $R^S_x$
of any element $x\in S$ as a quadruple consisting of:
\begin{itemize}
  
  \item the representative $x$;
  
  \item the s.c.c.~$\{(x)\lambda=\alpha_1, \ldots, \alpha_n\}$ of $(x)\lambda$
    under the action of $S$ (this can be found using
    Algorithms~\ref{algorithm-component} and any algorithm to find the strongly
    connected components of a digraph);
 
  \item a Schreier tree for~$\{\alpha_1, \ldots, \alpha_n\}$;
 
  \item the stabiliser group $S_{x}$ found using
    Algorithm~\ref{algorithm-schreier-generators-special}.

\end{itemize}
The $\L^S$-class of $x$ in $S$ can be represented using the analogous quadruple
using the s.c.c.~of $(x)\rho$, and the group ${}_{x}S$. The 
Green's $\H$- and $\D$-classes of $x$ in $S$ are represented using the
quadruple for $R_x^S$ and the quadruple for $L_x^S$.


\subsubsection*{Size of a Green's class}
Having the above data structures, it follows from Corollaries~\ref{cor-collect}(b)
and~\ref{cor-collect-analogue}(b) that $|R^S_x|=n \cdot |S_{x}|$ where $n$ is
the length of the s.c.c.~of $(x)\lambda$ in $(x)\lambda\cdot S ^ 1$. Similarly,
$|L^S_x|$ is the length of the s.c.c.~of $(x)\rho$ multiplied by $|{}_{x}S|$.
Suppose that $x'\in U$ is such that $xx'x=x$. Then, by
Proposition~\ref{prop-main-2}(b), $|H_x^S|=|S_{x}\cap ({}_{x}S)\Psi|$, where
$\Psi:{}_{x}S\to U_{x}$ is defined by $((s)\nu_x)\Psi = (x'sx)\mu_x$ for all
$s\in {}_{x}S$.  The group ${}_{x}S$ can be found using the analogue of
Algorithm~\ref{algorithm-schreier-generators-special}, and we can compute the
values $(x'sx)\mu_x$ for all $s\in {}_{x}S$ by Assumption \eqref{lambda-conj}.
The size of the $\D$-class $D^S_x$ is just
$$|D^S_x|=\frac{|L^S_x|\cdot |R^S_x|}{|H^S_x|},$$ and so $|D^S_x|$ can be found
using the values of $|L^S_x|$, $|R^S_x|$, and $|H^S_x|$. 


\subsubsection*{Elements of a Green's class}

Corollary~\ref{cor-elements}(a) states that
$$R_x^S=\set{xsu}{s\in \stab_S(L_x^U),\ u\in S ^ 1,\ (x)\lambda\cdot
u \sim (x)\lambda}.$$ 
If $s,t\in \stab_S(L_x^U)$ are such that $(s)\mu_x = (t)\mu_x$, then, by
Lemma~\ref{lem-free}, $xs=xt$.  It follows that if $M$ is any subset of
$\stab_S(L_x^U)$ such that $(M)\mu_x = S_x$ and $(s)\mu_x\not= (t)\mu_x$ for
all $s, t\in M$ such that $s\not = t$, then 
$$R_x^S=\set{xsu}{s\in M,\ u\in S ^ 1,\ (x)\lambda\cdot
u \sim (x)\lambda}.$$ 
If $Y=\{(y_1)\mu_x, \ldots, (y_k)\mu_x\}$ is a set of generators for $S_x$
(from Algorithm~\ref{algorithm-schreier-generators-special}), then every
element of $S_x$ is of the form $(s)\mu_x$ where $s\in \genset{y_1, \ldots,
y_k}$. Thus the set $M$ can be found by computing the elements of $S_x$, and
expressing each element as a product of the generators $Y$. An algorithm for
finding the elements of $R_x^S$ is given in
Algorithm~\ref{algorithm-R-class-elements}.


\begin{algorithm}
  \caption{Elements of an $\R$-class}
  \label{algorithm-R-class-elements}
  \begin{algorithmic}[1]
  \item[\textbf{Input:}] $x\in S$
    \item[\textbf{Output:}] the elements $Y$ of the $\R$-class $R_x^S$
    \item $Y:=\varnothing$
    \item find the s.c.c.~$\{(x)\lambda=\alpha_1,\ldots, \alpha_n\}$
      of $(x)\lambda$, and a Schreier tree for this s.c.c.
    \Comment{Algorithm~\ref{algorithm-component}}
    \item find the group $S_x$ 
      \Comment{Algorithm~\ref{algorithm-schreier-generators-special}}
      \For{$i\in \{1,\ldots, n\}$, $(s)\mu_x\in S_x$}
      \Comment{Loop over: the s.c.c., elements of the
      group}
      \State find $u_{i}\in S ^ 1$ such that $\alpha_1\cdot u_i=\alpha_i$
      \Comment{Algorithm~\ref{algorithm-trace-schreier}} 
      \State $Y\gets Y\cup \{xsu_i\}$
    \EndFor 
    \State \Return $Y$
  \end{algorithmic}
\end{algorithm}


The elements of an $\L$-class can be found using an analogous algorithm. By the
proof of Proposition~\ref{prop-main-2}(b), $\phi_1:S_{x}\cap
({}_{x}S)\Psi\to H_x^S$ defined by $((s)\mu_x)\phi_1= xs$ is a
bijection. It follows that if we can compute the intersection of groups
$S_{x}\cap ({}_{x}S)\Psi$, then we can obtain the elements of
$H_x^S$.  As mentioned above, $({}_{x}S)\Psi$ can be determined using
Algorithm~\ref{algorithm-schreier-generators-special} and by Assumption
\eqref{lambda-conj}.

The elements of a $\D$-class are slightly more complicated to compute.
In Algorithm~\ref{algorithm-R-classes-D-class} we show how to find the
$\R$-classes in a given $\D$-class of $S$ and this combined with
Algorithm~\ref{algorithm-R-class-elements} gives a method for finding the
elements of a $\D$-class. 


\subsubsection*{Classes within classes}

By Proposition~\ref{prop-D-class-R-class-reps}, if
$x\in S$, $x'\in U$ is such that $xx'x=x$, and:
\begin{itemize}
  \item $\mathcal{C}$ is a subset of $\stab_{U}(L_x^U)$ such that 
    $\set{(c)\l{x}}{c\in \mathcal{C}}$
    is a left transversal of $S_{x}\cap({}_{x}S)\Psi$ in $({}_xS)\Psi$ where
    $\Psi:{}_{x}S\to U_{x}$, defined by $((s)\r{x})\Psi= (x'sx)\l{x}$, is the
    embedding from Proposition~\ref{prop-main-2}(a);
  \item $\set{u_i\cdot(x)\rho}{1\leq i\leq m}$ is the s.c.c.\ of 
    $(x)\rho$ under the left action of $S$, where $u_i\in S ^ 1$ for all $i$, 
\end{itemize}
then $\set{u_ixc}{c\in \mathcal{C},\ 1\leq i\leq m}$ is a set of $\H^S$-class
representatives for $L_x^S$, and hence a set of $\R^S$-class representatives
for $D_x^S$. Using this result in Algorithm~\ref{algorithm-R-classes-D-class}
we show how to find the $\R^S$-classes of a $\D^S$-class.  Since
$(u_ixc)\lambda = (x)\lambda$, it follows that $S_{u_ixc}= S_x$  for all $i$
and all $c\in \mathcal{C}$.  Therefore, the data structures for $R_x^S$ and
$R_{u_ixc}^S$ are identical except for the representatives.

\begin{algorithm}
  \caption{$\R$-classes in a $\D$-class}
  \label{algorithm-R-classes-D-class}
  \begin{algorithmic}[1]
    \item[\textbf{Input:}] $x\in S$    

    \item[\textbf{Output:}] $\R$-class representatives $\mathfrak{R}$ of the
      $\D$-class $D_x^S$

    \item $\mathfrak{R}:=\varnothing$

    \item find the s.c.c.s~of $(x)\lambda$ and $(x)\rho$ and their Schreier
      trees
      \Comment{Algorithm~\ref{algorithm-component}}

    \State find $S_x\cap ({}_xS)\Psi$
      \Comment{Algorithm~\ref{algorithm-schreier-generators-special} and
      Assumption~\eqref{lambda-conj}} 

    \item find $\mathcal{C}\subseteq \stab_{U}(L_x^U)$ such that
      $\set{(c)\mu_x}{c\in \mathcal{C}}$ is a left transversal of $S_x\cap
      ({}_xS)\Psi$ in $({}_xS)\Psi$

    \item[] \Comment{Proposition~\ref{prop-D-class-R-class-reps}(a)}
    \For{$(y)\rho$ in the s.c.c.~of $(x)\rho$}
      \State find $u_{i}\in S$ such that $u_i\cdot (x)\rho=(y)\rho$
      \Comment{Algorithm~\ref{algorithm-trace-schreier}}
      \For{$c\in \mathcal{C}$}
        \State $\mathfrak{R}\gets \mathfrak{R}\cup \{u_ixc\}$
      \EndFor
    \EndFor
    \State \Return $\mathfrak{R}$
  \end{algorithmic}
\end{algorithm}

From Proposition~\ref{prop-D-class-L-class-reps}, algorithms analogous to
Algorithm~\ref{algorithm-R-classes-D-class}, can be used to find the
$\L^S$-classes in a $\D^S$-class, the $\H^S$-classes in an $\R^S$-class, or the
$\H^S$-classes in an $\L^S$-class.


\subsubsection*{Testing membership}

Using Corollary~\ref{cor-R-membership} and~\ref{prop-D-membership},
in Algorithms~\ref{algorithm-R-class-membership}
and~\ref{algorithm-D-class-membership} we show how the data structures described
at the start of the section can be used to test membership in an $\R$- or
$\D$-class. 

Using Proposition~\ref{prop-L-membership}, an algorithm analogous to
Algorithm~\ref{algorithm-R-class-membership} (for $\R$-classes) can be used to
test membership in an $\L$-class. Testing membership in an $\H$-class can then
be accomplished by testing membership in the corresponding $\L$- and
$\R$-classes. 


\begin{algorithm}
  \caption{Test membership in an $\R$-class}
  \label{algorithm-R-class-membership}
  \begin{algorithmic}[1]
    \item[\textbf{Input:}] $y\in U$ and the data structure of an $\R$-class $R^S_x$
    \item[\textbf{Output:}] true or false
      \If{$(x)\rho=(y)\rho$ and $(y)\lambda\sim (x)\lambda$} 
      \State find $u\in S ^ 1$ such that $(x)\lambda\cdot u=(y)\lambda$
      \Comment{Algorithm~\ref{algorithm-trace-schreier}}
      \State find $\ov{u}\in U ^ 1$ such that
      $(y)\lambda\cdot \ov{u}=(x)\lambda$ and $xu\ov{u}=x$ \Comment{Assumption 
      \eqref{lambda-inverse}}
      \State \Return $(x'y\ov{u})\mu_x\in S_x$ \Comment{Assumption
      \eqref{lambda-perm}}
    \Else
      \State \Return false
    \EndIf
  \end{algorithmic}
\end{algorithm}


\begin{algorithm}
  \caption{Test membership in a $\D$-class}
  \label{algorithm-D-class-membership}
  \begin{algorithmic}[1]
    \item[\textbf{Input:}] $y\in U$ and the data structure for the $\D$-class of
      $x\in S$
    \item[\textbf{Output:}] true or false
      \If{$(x)\lambda\sim (y)\lambda$ and $(x)\rho\sim (y)\rho$} 
      \State find $u_1, u_2\in S ^ 1$ such that $(x)\lambda\cdot{u_1}=(y)\lambda$
        and ${u_2}\cdot (x)\rho=(y)\rho$
      \Comment{Algorithm~\ref{algorithm-trace-schreier}}
      \State find $\ov{u_1}\in U ^ 1$ such that $(y)\lambda\cdot\ov{u_1}=(x)\lambda$ and 
        $xu_1\ov{u_1}=x$ 
        \Comment{Assumption \eqref{lambda-inverse}}
        \State find $\ov{u_2}\in U ^ 1$ such that $\ov{u_2}\cdot (y)\rho=(x)\rho$, and
        $\ov{u_2}u_2x=x$ 
        \Comment{the analogue of Assumption \eqref{lambda-inverse}}
        \State find $S_x\cap ({}_xS)\Psi$
        \Comment{Algorithm~\ref{algorithm-schreier-generators-special} and
        Assumption~\eqref{lambda-conj}} 
        \State find $\mathcal{C}\subseteq \stab_U(L_x^U)$ such that
        $\set{(c)\mu_x}{c\in \mathcal{C}}$ is a left transversal of $S_x\cap
        ({}_xS)\Psi$ in $({}_xS)\Psi$\\
          \Comment{Proposition~\ref{prop-D-class-R-class-reps}(a)}
        \For{$c\in \mathcal{C}$}
      \If{$(x'\ov{u_2}y\ov{u_1}c)\mu_x$ in $({}_xS)\Psi$}
          \State \Return true
        \EndIf
      \EndFor
    \EndIf
    \State \Return false
  \end{algorithmic}
\end{algorithm}


\subsubsection*{Regularity and idempotents}

An $\R$-class $R$ of $S$ is regular if and only if there is $x\in R$ such that
$H_x^U$ is a group. Since $y\in H_x^U$ if and only if $(x)\lambda=(y)\lambda$
and $(x)\rho=(y)\rho$, it is possible to verify that an $\R$-class is regular
only by considering the value $(x)\lambda$ and the s.c.c.~of $(x)\rho$. 
A similar approach can be used to to compute the
idempotents in an $\R$-class. 

How we test if the $\H$-class in $U$ corresponding to $(x)\lambda$ and $(y)\rho$
is a group, depends on the context. For example:
\begin{description}
  \item[Transformation semigroups:] in the full transformation monoid,
    $(x)\lambda$ and $(y)\rho$ are an image set and kernel of a transformation,
    respectively. In this case, the $\H$-class corresponding to $(x)\lambda$
    and $(y)\rho$ is a group if and only if $(x)\lambda$ contains precisely one
    element in every kernel class of $(x)\rho$. If $(x)\lambda$ and $(y)\rho$
    satisfy this property, then it is relatively straightforward to compute an
    idempotent with kernel $(y)\rho$ and image $(x)\lambda$;

  \item[Partial permutations:] in the symmetric inverse monoid, $(x)\lambda$ and
    $(y)\rho$ are the image set and domain of a partial permutation,
    respectively. In this case, the $\H$-class corresponding to $(x)\lambda$ and
    $(y)\rho$ is a group if and only if $(x)\lambda=(y)\rho$. Given such
    $(x)\lambda=(y)\rho$, the partial identity function with domain
    $(x)\lambda$ is the idempotent in the $\H$-class;

  \item[Matrix semigroups:] In this case, there is no simple criteria for the
    $\H$-class of $(x)\lambda$ and $(y)\rho$ to be a group. 

  \item[Rees matrix semigroups:] in a Rees $0$-matrix semigroup $\M^0[G; I, J;
    P]$, $(x)\lambda=j\in J$ and $(y)\rho=k\in I$. In this case, the $\H$-class
    corresponding to $(x)\lambda$ and $(y)\rho$ is a group if and only if
    $p_{j,k}\not=0$. The idempotent in the $\H$-class is then $(k,p_{j,k}^{-1},
    j)$;

  \item[Partitions:] in the partition monoid, $(x)\lambda=x^*x$ and
    $(y)\rho=yy^*$. In this case, the $\H$-class $L_{x^*x}\cap R_{yy^*}$
    contains an idempotent if and only if the number of transverse blocks of
    $yy^*x^*x$ equals the number of transverse blocks of $x^*x$. If the
    previous condition is satisfied, then the idempotent contained in
    $L_{x^*x}\cap R_{yy^*}$ is $yy^*x^*x$. 

\end{description}


\begin{algorithm}
  \caption{Regularity of an $\R$-class}
  \label{algorithm-R-class-regular}
  \begin{algorithmic}[1]
    \item[\textbf{Input:}] a representative $x\in S$ of an $\R$-class of $S$
    \item[\textbf{Output:}] true or false
    \item find the s.c.c.~of $(x)\rho$ in $S ^ 1\cdot (x)\rho$
     \Comment{analogue of Algorithm~\ref{algorithm-component}}
     \For{$(y)\rho$ in the s.c.c.~of $(x)\rho$}
      \If{the $\H$-class in $U$ corresponding to $(x)\lambda$ and
        $(y)\rho$ is a group}
          \State \Return true
      \EndIf
    \EndFor
    \State \Return false
  \end{algorithmic}
\end{algorithm}

Algorithms analogous to Algorithms~\ref{algorithm-R-class-regular}
and~\ref{algorithm-R-class-idempotents} can be used to test regularity and find
the idempotents in an $\L$-class.  A $\D$-class $D$ in $S$ is regular if and
only if any (equivalently, every) $\R$-class in $D$ is regular. Hence
Algorithms~\ref{algorithm-R-classes-D-class}
and~\ref{algorithm-R-class-regular} and can be used to verify if a $\D$-class is
regular or not. It is possible to calculate the idempotents in a $\D$-class $D$
by creating the $\R$-class representatives using
Algorithm~\ref{algorithm-R-classes-D-class} and finding the idempotents in
every $\R$-class of $D$ using Algorithm~\ref{algorithm-R-class-idempotents}. 


\begin{algorithm}
  \caption{Idempotents in an $\R$-class}
  \label{algorithm-R-class-idempotents}
  \begin{algorithmic}[1]
    \item[\textbf{Input:}] $x\in S$
    \item[\textbf{Output:}] the idempotents $E$ of $R^S_x$
    \item $E:=\varnothing$
    \item find the s.c.c.~of $(x)\rho$ in $S ^ 1 \cdot (x)\rho$
      \Comment{analogue of Algorithm~\ref{algorithm-component}}
     \For{$(y)\rho$ in the s.c.c.~of $(x)\rho$}
      \If{the $\H$-class $H$ in $U$ corresponding to $(x)\lambda$ and
        $(y)\rho$ is a group}
        \State find the identity $e$ of $H$
        \State $E\gets E\cup\{e\}$
      \EndIf
    \EndFor
    \State \Return $E$
  \end{algorithmic}
\end{algorithm}


\subsection{The global structure of a semigroup}\label{subsection-global}
In this section, we provide algorithms for determining the structure of an
entire semigroup, rather than just its individual Green's classes as in the previous
sections.  Unlike the previous two subsections, the algorithms described in this
section differ significantly from the analogous procedures described in
\cite{Linton2002aa}. 

In order to streamline the presentation of some algorithms in this section, we
will perform them on $U ^ 1$ and $S ^ 1$, instead of $U$ and $S$.  Of course,
if we want to perform a calculation on $U$ or $S$ instead of $U ^ 1$ or $S ^
1$, then we can simply perform whatever calculation we require in $U^1$ or $S ^
1$, then return the answer for $U$ or $S$.


\subsubsection*{The main algorithm}

Algorithm~\ref{algorithm-enumerate} is the main algorithm for computing the
size, the Green's structure, testing membership, and so on in $S$. This
algorithm could be replaced by an analogous algorithm which enumerates
$\L$-classes, rather than $\R$-classes. In the case of transformation
semigroups, in many test cases, the algorithm for $\R$-classes has better
performance than the analogous algorithm for $\L$-classes.  This is one reason
for presenting this algorithm and not the other.  

Since Green's $\R$-relation is a left congruence, it follows that
representatives of the $\R$-classes of $S$ can be obtained from the identity by
left multiplying by the generators. The principal purpose of
Algorithm~\ref{algorithm-enumerate} is to determine the action of $S$ on its
$\R$-class representatives by left multiplication. In particular, we enumerate
$\R$-class representatives of $S$, which we will denote by $\mathfrak{R}$. Since
we are calculating an action of $S$, we may also discuss the Schreier tree and
orbit graph of this action, as we did in Algorithm~\ref{algorithm-component}.
At the same time as finding the action of $S$ on its $\R$-class representatives,
we also calculate $(S)\rho$. Since $(x)\rho=(y)\rho$ if and only if $x\R^U y$,
it follows that 
$$(S)\rho = (\mathfrak{R})\rho = \set{(x)\rho}{x\in \mathfrak{R}}.$$ 
In Algorithm~\ref{algorithm-enumerate}, we find an addition parameter, which is
denoted by $K_i$. This parameter is used later in
Algorithm~\ref{algorithm-factorization}, which allows us to factorise elements
of a semigroup over its generators.

If the subsemigroup $S$ we are trying to compute is $\R$-trivial, then
Algorithm~\ref{algorithm-enumerate} simply exhaustively enumerates the elements
of $S$.  In such a case, unfortunately, Algorithm~\ref{algorithm-enumerate} has
poorer performance than a well-implemented exhaustive algorithm, since it
contains some superfluous steps.  For example, the calculations of $(S)\lambda$
and the groups $S_x$ are unnecessary steps if $S$ is $\R$-trivial. On a more
positive note, in some cases, it is possible to detect if a semigroup is
$\R$-trivial with relatively little effort. For example, if $S$ is a
transformation semigroup of degree $n$, then it is shown in \cite{Simon1975aa}
that $S$ is $\R$-trivial if and only if its action on the points in
$\{1,\ldots, n\}$ is acyclic. In other words, it is possible to check whether a
transformation semigroup is $\R$-trivial in polynomial time $O(n^3)$.  At least
in this case, it would then be possible to use the $\L$-class version of
Algorithm~\ref{algorithm-enumerate} instead of the $\R$-class version, or
indeed, an exhaustive algorithm such as that in \cite{Pin2009aa}.


\begin{algorithm}
  \caption{Enumerate the $\R$-classes of a semigroup}
  \label{algorithm-enumerate} 
  \begin{algorithmic}[1]
    \item[\textbf{Input:}] $S ^ 1:=\genset{X}$ where $X:=\{x_1,\ldots, x_m\}$

    \item[\textbf{Output:}] data structures for the $\R$-classes with
      representatives $\mathfrak{R}$ in $S ^ 1$, the set $(S ^ 1)\rho$, and Schreier
      trees and orbit graphs for $\mathfrak{R}$ and $(S ^ 1)\rho$

    \item set $\mathfrak{R}:=\{y_1:=1_S\}$, $M:=1$
      \Comment{initialise the list of $\R$-class reps}

    \item set $(S ^ 1)\rho:=\{\beta_1:=(1_S)\rho\}$, $N:=1$
      \Comment{initialise $(S)\rho$}

    \item find $(S ^ 1)\lambda = (1_S)\lambda\cdot S$
      \Comment{Algorithm~\ref{algorithm-component}}

    \item find representatives $(z_1)\lambda, \ldots, (z_r)\lambda$ of the
      s.c.c.s of $(S ^ 1)\lambda$ 
      \Comment{standard graph theory algorithm}
    
    \item find the groups $S_{z_i}$  for $i\in\{1,\ldots, r\}$
      \Comment{Algorithm~\ref{algorithm-schreier-generators-special}}

    \For{$y_i\in{\mathfrak{R}}$, $j\in \{1,\ldots, m\}$}
    \Comment{loop over: existing $\R$-representatives, generators of $S ^ 1$}

      \State find $n\in \{1,\ldots, r\}$ such that $(z_n)\lambda\sim
      (x_jy_i)\lambda$ 

      \State find $u\in S ^ 1$ such that $(z_n)\lambda\cdot u=(x_jy_i)\lambda$
        \Comment{Algorithm~\ref{algorithm-trace-schreier}}

      \State find $\ov{u}\in U ^ 1$ such that $(x_jy_i)\lambda\cdot
      \ov{u}=(z_n)\lambda$ and $z_nu\ov{u} = z_n$
        \Comment{Assumption \eqref{lambda-inverse}}

      \If{$(x_jy_i)\rho\not\in (S ^ 1)\rho$}
        \Comment{$x_jy_i\ov{u}$ is a new $\R$-representative}

        \State $N:=N+1$, $\beta_N:=(x_jy_i\ov{u})\rho=x_j\cdot (y_i)\rho$,
          $(S ^ 1)\rho\gets (S ^ 1)\rho\cup \{\beta_N\}$
          \Comment{add $(x_jy_i)\rho$ to $(S ^ 1)\rho$}

        \State $v_N:=i$, $w_N:=j$, $g_{i,j}:=N$
          \Comment{update the Schreier tree and orbit graph of $(S ^ 1)\rho$}

        \ElsIf{$(x_jy_i)\rho=\beta_k\in (S ^ 1)\rho$}

          \State $g_{l,j}:=k$ where $l$ is such that $(y_i)\rho=\beta_l$;
            \Comment{update the orbit graph of $(S ^ 1)\rho$}

          \For{$y_l\in \mathfrak{R}$ with $(y_l)\rho=(x_jy_i)\rho$ and
            $(y_l)\lambda=(z_n)\lambda$}

            \If{$(y_l'x_jy_i\ov{u})\mu_{z_n}\in
            S_{z_n}=S_{y_l}$}\Comment{$x_jy_i\R^Sy_l$}
              \State $G_{i,j}:=l$
                \Comment{update the orbit graph of $\mathfrak{R}$}

              \State go to line 6

            \EndIf
          \EndFor
        \EndIf

        \State $M:=M+1$, $y_M:=x_jy_i\ov{u}$,
        $\mathfrak{R}\gets\mathfrak{R}\cup\{y_M\}$
          \Comment{add $x_jy_i\ov{u}$ to $\mathfrak{R}$}

        \State$V_{M}:=i$, $W_M:=j$, $K_M:=$ the index of $(x_jy_i)\lambda$ in
        $(S ^ 1)\lambda$ 
          \Comment{update the Schreier tree of $\mathfrak{R}$}

        \State $G_{i, j}:=M$
        \Comment{update the orbit graph of $\mathfrak{R}$}
    \EndFor
    \State \Return the following:
      \begin{itemize}
        \item the $\R$-representatives: $\mathfrak{R}$ 
        \item the Schreier tree for $\mathfrak{R}$: $(V_2,\ldots, V_{M}, W_2,
          \ldots, W_M)$
        \item the orbit graph of $\mathfrak{R}$: $\set{G_{i,j}}{i=1,\ldots, M,\
          j = 1, \ldots, m}$
        \item the set $(S ^ 1)\rho$
        \item the Schreier tree for $(S ^ 1)\rho$: $(v_2, \ldots, v_N, w_2, \ldots,
          w_N)$
        \item the orbit graph of $(S ^ 1)\rho$: $\set{g_{i,j}}{i=1,\ldots, N,\ j =
          1,\ldots, m}$
        \item the parameters: $(K_1, \ldots, K_M)$
      \end{itemize}
  \end{algorithmic}
\end{algorithm}


From Algorithm~\ref{algorithm-enumerate}, it is routine to
calculate the size of $S$ as:
\begin{equation}\label{eq-size}
  |S|=\sum_{i=1}^{r} |S_{z_i}| \cdot |\text{s.c.c.~of}\ (z_i)\lambda| \cdot
  |\set{y\in \mathfrak{R}}{(y)\lambda=(z_i)\lambda}| - 1.
\end{equation}
The elements of $S$ are just the union of the sets of elements of the
$\R$-classes of $S$, which can be found using
Algorithm~\ref{algorithm-R-class-elements}.

The idempotents of $S$ can be found by determining the idempotents in the
$\R$-classes of $S$ using Algorithm~\ref{algorithm-R-class-idempotents}.
Similarly, it is possible to test if $S$ is regular, by checking that every
$\R$-class or $\D$-class is regular using
Algorithm~\ref{algorithm-R-class-regular}.  Note that the existence of an
idempotent in an $\R^S$-class depends only on the values of $\lambda$ and
$\rho$ of the elements of that class.  Hence if there are at least two distinct
$\R$-class representatives $x$ and $y$ in $S$ such that $(x)\lambda=(y)\lambda$
and $(x)\rho=(y)\rho$, then $S$ is not regular, since the disjoint $\R$-classes
$R^S_x$ and  $R^S_y$ cannot contain the same idempotents.

The $\D$-classes of $S ^ 1$ are in 1-1 correspondence with the strongly connected
components of the orbit graph of $\mathfrak{R}$, which is obtained in
Algorithm~\ref{algorithm-enumerate}. Since we also find $(S ^ 1)\rho$ in
Algorithm~\ref{algorithm-enumerate}, it is possible to find the $\D$-classes of
$S ^ 1$, and their data structures, by finding the strongly connected
components of the orbit graph of $\mathfrak{R}$.  The $\L$- and $\H$-classes of
$S$ and $S ^ 1$ can then be found from the $\D$-classes using the analogues of
Algorithm~\ref{algorithm-R-classes-D-class} (for finding the $\R$-classes in a
$\D$-class).

In Algorithms~\ref{algorithm-membership},~\ref{algorithm-factorization},
~\ref{algorithm-D-class-order}, and~\ref{algorithm-closure}, we will assume
that Algorithm~\ref{algorithm-enumerate} has been performed and that the result
has been stored somehow. So, for example, if we wanted to check membership of
several elements in the semigroup $U$ in its subsemigroup $S$, we perform
Algorithm~\ref{algorithm-enumerate} only once. 


\subsubsection*{Testing membership in a semigroup}

In Algorithm~\ref{algorithm-membership}, we give a procedure for testing 
if an element of $U$ belongs to $S$. This can be easily modified to
return the $\R$-class representative in $S$ of an arbitrary element of
$U$, if it exists (simply return the element $x$ in line 12). We require such an
algorithm when it comes to factorising elements of $S$ over the generators in
Algorithm~\ref{algorithm-factorization}.


\begin{algorithm}
  \caption{Test membership in a semigroup}
  \label{algorithm-membership} 
  \begin{algorithmic}[1]
    \item[\textbf{Input:}] $S ^ 1:=\genset{X}$ where $X:=\{x_1,\ldots, x_m\}$
      and $y\in U ^ 1$

    \item[\textbf{Output:}] true or false
    
    \item let $(S ^ 1)\lambda = (1_S)\lambda \cdot S$ and $(S ^ 1)\rho = S\cdot
      (1_S)\rho$
      \Comment{Algorithm~\ref{algorithm-component}}

    \If{$(y)\lambda\not\in (S ^ 1)\lambda$ or $(y)\rho\not\in (S ^ 1)\rho$}
      \Comment{$y\not\in S ^ 1$}
      \State \Return false
    \EndIf
    
    \item let $(z_1)\lambda, \ldots, (z_r)\lambda$ be representatives of the
      s.c.c.s of $(S ^ 1)\lambda$
    
    \item let $\mathfrak{R}$ denote the $\R$-class representatives of $S ^ 1$
      \Comment{Algorithm~\ref{algorithm-enumerate}} 

    \item find $n\in \{1,\ldots, r\}$ such that $(z_n)\lambda\sim (y)\lambda$ 

    \item find $u\in S ^ 1$ such that $(z_n)\lambda\cdot u=(y)\lambda$
      \Comment{Algorithm~\ref{algorithm-trace-schreier}}

    \item find $\ov{u}\in U ^ 1$ such that $(y)\lambda\cdot \ov{u}=(z_n)\lambda$
      and $yu\ov{u} = y$
      \Comment{Assumption \eqref{lambda-inverse}}
    
    \For{$x\in \mathfrak{R}$ such that $(x)\lambda=(z_n)\lambda$ and 
      $(x)\rho=(y)\rho$}
      \If{$(x'y\ov{u})\mu_x\in S_{z_n} = S_x$}\Comment{$x\R^{S ^ 1}y$}
        \State \Return true
      \EndIf
    \EndFor
    \State \Return false
  \end{algorithmic}
\end{algorithm}


\subsubsection*{Factorising elements over the generators}

In this part of the paper, we describe how to factorise an element of $S$ as a
product of the generators of $S$ using the output of
Algorithm~\ref{algorithm-enumerate}. 

Corollary~\ref{cor-elements}(a) says that 
\begin{equation*}
  R_x^S=\set{xsu}{s\in \stab_S(L_x^U),\ u\in S ^ 1,\ (xu)\lambda\sim\lambda(x)}.
\end{equation*}
Suppose that $y\in S$ is arbitrary and that $\mathfrak{R}$ denotes a set of
$\R$-class representatives of $S$. 
If we can write $y = xsu$, where $x\in \mathfrak{R}$ such that $x\R^S y$, $s\in
\stab_S(L_x^U)$, and $u\in S ^ 1$ such that $(xu)\lambda = (y)\lambda$, then it
suffices to factorise each of $x$, $s$, and $u$ individually.

The element $x \in \mathfrak{R}$ can be found using the
alternate version of Algorithm~\ref{algorithm-membership} mentioned above.
Algorithm~\ref{algorithm-trace-schreier}, applied to $(S)\lambda$, can be used
to find $u$ such that $(xu)\lambda = (y)\lambda$.

Suppose that $\ov{u}\in U ^ 1$ is such that $(y)\lambda\cdot \ov{u} = (x)\lambda$ and
$xu\ov{u} = x$. Since $(xs)\lambda = (x)\lambda$, it follows from
Lemma~\ref{lem-free} that $xsu\ov{u} = xs$. 
If $x'\in U$ is any element such that $xx'x=x$, then 
from Proposition~\ref{prop-main-1}(a), 
$x$ is the identity of the group $L_x^U\cap R_x^S$ under multiplication $*$
defined by $a*b=ax'b$. In particular, $xx'y\ov{u} = y\ov{u}$ since $y\ov{u}\in
L_x^U\cap R_x^S$.
This implies that 
$$xx'y\ov{u}=y\ov{u}=xsu\ov{u}=xs$$
and so, by Lemma~\ref{lem-free}, $(x'y\ov{u})\mu_x = (s)\mu_x$.
Since $(y\ov{u})\lambda = (x)\lambda$ and $(y\ov{u})\rho = (x)\rho$, by
Assumption \eqref{lambda-perm}, 
we can compute $(x'y\ov{u})\mu_x$. Thus for any $y\in S$ we can determine
$x,s,u\in S$ (as above) such that $y=xsu$.

We still require a factorisation of $x,s,u\in S$ over the generators of $S$
(here we suppose that $u\not= 1_S$).
Tracing the Schreier tree of $\mathfrak{R}$ returned by
Algorithm~\ref{algorithm-enumerate} and using the parameter $K_i$, we can
factorise $x\in \mathfrak{R}$; more details are in
Algorithm~\ref{algorithm-factorization}.  We may factorise $u$ over the
generators of $S$ using Algorithm~\ref{algorithm-trace-schreier} applied to the
s.c.c.~of $(x)\lambda$ in $(S)\lambda$.

Any algorithm for factorising elements of a group can be used to
factorise $(s)\mu_x$ as a product of the generators of $S_{x}$ given by
Algorithm~\ref{algorithm-schreier-generators-special}. For example, in a group
with a faithful action on some set, such a factorisation can be obtained from a
stabiliser chain, produced using the Schreier-Sims algorithm.
The generators of $S_{x}$ are of the form $(u_ix_k\ov{u_j})\mu_x$, where $u_i$ is
obtained by tracing the Schreier tree of $(S)\lambda$, $x_k$ is one of the
generators of $S$, and $\ov{u_j}$ is obtained using Assumption
\eqref{lambda-inverse}. Therefore to factorise $s$ over the generators of $S$,
it suffices to factorise those $\ov{u_j}$ where $\ov{u_j} \not= 1_S$ over the
generators of $S$. 

Suppose that $\{\alpha_1, \ldots, \alpha_K\}$ is a s.c.c.~of $(S)\lambda$ for
some $K\in \N$. Then from the orbit graph of $(S)\lambda$ we can find 
$$a_2, \ldots, a_{K}, b_2, \ldots, b_K$$
such that 
$$\alpha_j\cdot x_{a_j}=\alpha_{b_j}$$
and $b_j < j$ for all $j > 1$.
In other words, $(a_2, \ldots, a_{K}, b_2, \ldots, b_K)$ describes a spanning
tree, rooted at $\alpha_1$, for the component of the orbit graph of
$(S)\lambda$, whose edges have the opposite orientation to those in the usual
Schreier tree.  We refer to  $(a_2, \ldots, a_{K}, b_2, \ldots, b_K)$ as a
\textit{reverse Schreier tree}.

It follows that for any $i\in \{1,\ldots, K\}$ we can use an analogue of
Algorithm~\ref{algorithm-trace-schreier} to obtain $v\in S ^ 1$ such that
$\alpha_j\cdot v=\alpha_1$.  However, if $u\in S ^ 1$ is obtained using
Algorithm~\ref{algorithm-trace-schreier} such that $\alpha_1\cdot u=\alpha_j$
and $z\in S$ is such that $(z)\lambda = \alpha_1$, then it is
possible that $(uv)\mu_z\not=(1_U)\mu_z$. 
So, if $w\in \stab_S(L_z^U)$ is such that
$(w)\mu_z=((uv)\mu_z) ^ {-1}$, then $\alpha_j \cdot vw = \alpha_1$ 
$(uvw)\mu_z = (1_U)\mu_z$. We have factorisations of $v$, since it was
obtained by tracing the reverse Schreier tree, and $w$, since it can be given
as a power of $uv$ ($u$ and $v$ are factorised), over the generators of $S$. 
Hence $vw$ is factorized over the generators of $S$, and it has the properties
required of $\ov{u_j}$ from above. 

We note that all the information required to factorise any element of $S$ is
returned by Algorithm~\ref{algorithm-enumerate} except the factorisation of
$(s)\mu_x$ in $S_{x}$. So, Algorithm~\ref{algorithm-factorization} is just
concerned with putting this information together. There is no guarantee that
the word produced by Algorithm~\ref{algorithm-factorization} is of minimal
length. 


\begin{algorithm}
  \caption{Factorise an element over the generators}
  \label{algorithm-factorization} 
  \begin{algorithmic}[1]
    \item[\textbf{Input:}] $S:=\genset{X}$ where $X:=\{x_1,\ldots, x_m\}$
      and $s\in S$

    \item[\textbf{Output:}] a word in the generators $X$ equal to $s$

    \item let $\theta:X^{+}\to S$ be the unique homomorphism extending the
      inclusion of $X$ in $S$
    
    \item suppose that $(S)\lambda:=\{(z_1)\lambda, \ldots, (z_K)\lambda\}$ for
      some $z_1, \ldots, z_K\in S$
      \Comment{Algorithm~\ref{algorithm-component}}
    
    \item let $\mathfrak{R}:= \{y_1, \ldots, y_r\}$ be the $\R$-representatives
      of $S$
      \Comment{Algorithm~\ref{algorithm-enumerate}}
    
    \item let $(V_2, \ldots, V_{r}, W_2, \ldots,
      W_{r})$ be the Schreier tree for $\mathfrak{R}$
      \Comment{Algorithm~\ref{algorithm-enumerate}}

    \item let $(K_1, \ldots, K_{r})$ denote the additional parameter
      return from Algorithm~\ref{algorithm-enumerate}

    \item find $y_i\in \mathfrak{R}$ such that $y_i\R^S s$
      \Comment{the modified version of Algorithm~\ref{algorithm-R-class-membership}}

    \item find a word $\omega_1\in X^+$ such that $(\omega_1)\theta=u\in S$
      where $(y_i)\lambda\cdot u=(s)\lambda$
      \Comment{factorise $u$ using Algorithm~\ref{algorithm-trace-schreier}}

    \item find $\ov{u}\in U ^ 1$ such that $(s)\lambda\cdot \ov{u}=(y_i)\lambda$, and
      $y_iu\ov{u}=y_i$
      \Comment{Assumption \eqref{lambda-inverse}}

    \item compute $(y_i's\ov{u})\mu_{y_i}\in S_{y_{i}}$
      \Comment{Assumption \eqref{lambda-perm}}
    
    \item find a word $\omega_2\in X^+$ such that $(\omega_2)\theta=y_i's\ov{u}$
      \Comment{factorise $s$}
    \item[] \Comment{factorise $(y_i's\ov{u})\mu_{y_i}$ over the generators of
      $S_{y_i}$ and then factorise these generators over $X$}
   
    \State $\omega_3:=\varepsilon$ (the empty word), $j=i$
      \Comment{trace the Schreier tree of $\mathfrak{R}$, factorise $y_i$}
    
    \While{$j>1$}
      \State find $\beta\in X^{+}$ such that
        $(z_{K_{j}})\lambda \cdot {(\beta)\theta}=(y_j)\lambda$

      \State set $\omega_3:=x_{W_j}\omega_3 \beta$ and $j:=V_j$

    \EndWhile\Comment{$(\omega_3)\theta=y_i$}

    \item \Return $\omega_3\omega_2\omega_1$
      \Comment{$(\omega_3\omega_2\omega_1)\theta = y_iy_i's\ov{u}u = s$}

  \end{algorithmic}
\end{algorithm}


\subsubsection*{The partial order of the $\J$-classes}

Recall that there is a partial order $\leq_{\J}$ on the $\J$-classes of a
semigroup $S$, which is induced by containment of principal two-sided ideals.
If $S$ is finite, then $\J = \D$, and so $\leq_{\J}$ is also a partial order on
the $\D$-classes of $S$.  More precisely, if $A$ and $B$ are $\J$-classes of
$S$, then we write $A\leq_{\J}B$ if $S^1aS^1\subseteq S^1bS^1$ for any (and
every) $a\in A$ and $b\in B$.

The penultimate algorithm in this paper allows us to calculate the partial
order of the $\D$-classes of $S$. This algorithm is based on \cite[Algorithm
Z]{Linton2002aa} and the following proposition which appears as Proposition 5.1 in
\cite{Lallement1990aa}. The principal differences between our algorithm and
Algorithm Z in \cite{Linton2002aa} are that our algorithm applies to classes of
semigroups other than transformation semigroups, and it takes advantage of
information already determined in Algorithm~\ref{algorithm-enumerate}.


\begin{prop}[cf. Proposition 5.1 in \cite{Lallement1990aa}]
  \label{prop-D-class-order}
  Let $S$ be a finite semigroup generated by a subset $X$, and let
  $D$ be a $\D$-class of $S$. If $R$ and $L$ are representatives of the $\R$-
  and $\L$-classes in $D$, then the set $XR\cup LX$ contains representatives for
  the $\D$-classes immediately below $D$ under $\leq_{\J}$.
\end{prop}


We described above that we find the $\D$-classes (or representatives for the
$\D$-classes) of $S$ by finding the s.c.c.s of the orbit graph of the $\R$-class
representatives $\mathfrak{R}$ of $S$. Thus, after performing
Algorithm~\ref{algorithm-enumerate} and finding the s.c.c.s of the orbit graph
of $\mathfrak{R}$, we know both the $\D$-classes of $S$ and the $\R$-classes
contained in the $\D$-classes. In Algorithm~\ref{algorithm-enumerate}, we obtain
the $\R$-class representatives of $S$ by left multiplying existing
representatives. Therefore we have already found the information required to
determine the set $XR$ in Proposition~\ref{prop-D-class-order}. In particular,
we do not have to multiply the $\R$-class representatives of a $\D$-class by
each of the generators in Algorithm~\ref{algorithm-D-class-order}, or determine
which $\D$-classes correspond to the elements of $XR$. 

In Algorithm~\ref{algorithm-D-class-order} we represent the partial order of
the $\D$-classes $D_1, \ldots, D_n$ of $S$ as $P_1, \ldots, P_n$ where $P_i$
contains the indices of the $\D$-classes immediately below $D_i$ (and maybe
some more). The elements of $P_i$ are obtained by applying
Proposition~\ref{prop-D-class-order} to $D_i$. 


\begin{algorithm}
  \caption{The partial order of the $\D$-classes of a semigroup}
  \label{algorithm-D-class-order} 
  \begin{algorithmic}[1]
    \item[\textbf{Input:}] $S:=\genset{X}$ where $X:=\{x_1,\ldots, x_m\}$

    \item[\textbf{Output:}] the partial order of the $\D$-classes of $S$

    \item let $\mathfrak{R}$ denote the $\R$-representatives of $S$
      \Comment{Algorithm~\ref{algorithm-enumerate}}

    \item let $\Gamma:=\set{G_{i,j}}{1\leq i\leq |\mathfrak{R}|,\ 1\leq j\leq
      m}$ be the orbit graph of $\mathfrak{R}$
      \Comment{Algorithm~\ref{algorithm-enumerate}}

    \item find the $\D$-classes $D_1, \ldots, D_n$ of $S$ and $D_i\cap
      \mathfrak{R}$ for all $i$
      \Comment{find the s.c.c.s of $\Gamma$}

    \item let $P_i:=\varnothing$ for all $i$
      \Comment{initialise the partial order}

    \For{$i\in \{1,\ldots, n\}$}
     \Comment{loop over the $\D$-classes}
      \For{$y_j\in \mathfrak{R}\cap D_i$, $x_k\in X$} 
        \Comment{loop over: $\R$-classes of the $\D$-class, generators of $S$}

        \State find $l\in \{1,\ldots, n\}$ such that $y_{G_{j,k}}\in D_l$
        \Comment{use the orbit graph of $\mathfrak{R}$}
        
        \State $P_i\gets P_i\cup\{l\}$
      \EndFor

      \State find the $\L$-class representatives $\mathfrak{L}$ in $D_i$
        \Comment{the analogue of Algorithm~\ref{algorithm-R-classes-D-class},
        Proposition~\ref{prop-D-class-L-class-reps}}

      \For{$z_j\in \mathfrak{L}$, $x_k\in X$}
        \Comment{loop over: $\L$-classes of the $\D$-class, generators of $S$}
        
        \State find $l$ such that $z_jx_k\in D_l$
        \Comment{use the analogue of
        Algorithm~\ref{algorithm-R-class-membership} to find the 
        $\R$-rep.\ of $z_jx_k$}
        
        \State $P_i\gets P_i\cup\{l\}$
      \EndFor
    \EndFor
    \State \Return $P_1, \ldots, P_n$.
  \end{algorithmic}
\end{algorithm}


\subsubsection*{The closure of a semigroup and some elements}

The final algorithm (Algorithm~\ref{algorithm-closure}) we present describes a
method for taking the closure of a subsemigroup $S$ of $U$ with a set $V$ of
elements of $U$. The purpose of this algorithm is to reuse whatever
information is known about $S$ to make subsequent computations involving
$\genset{S, V}$ more efficient. 

At the beginning of this procedure we form $(\genset{S, V})\lambda$ by adding
the new generators $V$ to $(S)\lambda$. This can be achieved in \GAP using the
{\tt AddGeneratorsToOrbit} function from the \Orb package \cite{Muller2013aa},
which has the advantage that the existing information in $(S)\lambda$ is not
recomputed.  The orbit $(S)\lambda$ is extended by a breadth-first enumeration
with the new generators without reapplying the old
generators to existing values in $(S)\lambda$. 


\begin{algorithm}
  \caption{The closure of a semigroup and some elements}
  \label{algorithm-closure} 
  \begin{algorithmic}[1]
      
    \item[\textbf{Input:}] $S:=\genset{X}$ where $X:=\{x_1,\ldots, x_m\}$
      and any existing data structures for $S$, and $V\subseteq U$

    \item[\textbf{Output:}] data structures for the $\R$-classes with
      representatives $\mathfrak{R}$ in $\genset{S,V}$, the set
      $(\genset{S,V})\rho$, and Schreier trees and orbit graphs for
      $\mathfrak{R}$ and $(\genset{S,V})\rho$

    \item set $\widehat{\mathfrak{R}}:=\{\widehat{y_1}, \ldots,
      \widehat{y_{k}}\}$ to be the $\R$-class representatives of $S$
      \Comment{Algorithm~\ref{algorithm-enumerate}}

    \item set $\mathfrak{R}:=\{y_1:=\widehat{y_1}\}$, $M:=1$     
      \Comment{initialise the list of new $\R$-reps}
    
    \item define $(1)\iota = 1$
      \Comment{keep track of indices of old and new $\R$-reps}

    \item set $(\genset{S, V})\rho:=(S)\rho$, $N:=|(S)\rho|$
      \Comment{initialise $(\genset{S,V})\rho$}

    \item set the Schreier tree of $(\genset{S, V})\rho$ to be that of $(S)\rho$
      \Comment{initialise the Schreier tree of $(\genset{S,V})\rho$}
      
    \item set the orbit graph of $(\genset{S, V})\rho$ to be that of  
      $(S)\rho$
      \Comment{initialise orbit graph of $(\genset{S,V})\rho$}

    \item extend $(S)\lambda$ to $(\genset{S, V})\lambda$
      \Comment{as described above}

    \item find representatives $(z_1)\lambda, \ldots, (z_r)\lambda$ of the
      s.c.c.s of $(\genset{S,V})\lambda$

    \item find the groups $S_{z_i}$ for $i\in \{1,\ldots,r\}$
      \Comment{Algorithm~\ref{algorithm-schreier-generators-special}}

      \For{$x_j\widehat{y_k}v =\widehat{y_i}\in\widehat{\mathfrak{R}}$}
      \Comment{$j, k$ are obtained from the Schreier tree for
      $\widehat{\mathfrak{R}}$, $k < i$}


      \State find $n\in \{1,\ldots, r\}$ such that $(z_n)\lambda\sim
        (\widehat{y_i})\lambda$ (in $(\genset{S, V})\lambda$)
      
      \State find $u\in S ^ 1$ such that $(z_n)\lambda\cdot u=(\widehat{y_i})\lambda$
        \Comment{Algorithm~\ref{algorithm-trace-schreier}}

      \State find $\ov{u}\in U ^ 1$ such that $(\widehat{y_i})\lambda\cdot
      \ov{u}=(z_n)\lambda$ and $z_nu\ov{u}=z_n$
        \Comment{Assumption \eqref{lambda-inverse}}

      \For{$y_l\in \mathfrak{R}$ with $(y_l)\rho=(\widehat{y_i})\rho$ and
        $(y_l)\lambda=(z_n)\lambda$}

        \If{$(y_l'\widehat{y_i}\ov{u})\mu_{z_n}\in S_{z_n}=S_{y_l}$}
          \Comment{$\widehat{y_i}\R^{\genset{S,V}}y_l$}

            \State $G_{(k)\iota,j}:=l$
              \Comment{update the orbit graph of $\mathfrak{R}$}
            \State define $(i)\iota := l$
              \Comment{the old index $i$ is the new index $l$}
            \State go to line 10
        \EndIf
      \EndFor

      \State $M:=M+1$, $y_M:=\widehat{y_i}\ov{u}$,
        $\mathfrak{R}:=\mathfrak{R}\cup\{y_M\}$
        \Comment{add $\widehat{y_i}\ov{u}$ to $\mathfrak{R}$}

      \State $V_{M}:=(k)\iota$, $W_M:=j$
      \Comment{update the Schreier tree of $\mathfrak{R}$}
      
      \State $K_M:=\widehat{K_i}$
      \Comment{the index of $(x_j\widehat{y_k})\lambda$ in $(\genset{S,
      V})\lambda$ from Algorithm~\ref{algorithm-enumerate} applied to $S$}

      \State $G_{(k)\iota, j}:=M$ 
          \Comment{update the orbit graph of $\mathfrak{R}$}

        \State $(i)\iota := M$
          \Comment{the old index $i$ is the new index $M$}
    \EndFor

    \State $\Return$ apply Algorithm~\ref{algorithm-enumerate} to the
      data structures for $\genset{S, V}$ determined so far, and return the
      output
  \end{algorithmic}
\end{algorithm}


\subsection{Optimizations for regular and inverse semigroups}
\label{section-algorithms-regular}

Several of the algorithms presented in this section become more
straightforward if it is known \textit{a priori} that the subsemigroup $S$ of $U$
is regular. For example, the $\R$-classes of $S$ are just the $\R$-classes of
$U$ intersected with $S$, and so are in 1-1 correspondence with $(S)\rho$. 
The algorithms become simpler still under the assumption that $U$ is an
inverse semigroup since in this case we may define $(x)\rho=(x^{-1})\lambda$
for all $x\in U$ and so it is unnecessary to calculate $(S)\rho$ and
$(S)\lambda$ separately. 

The first algorithm where an advantage can be seen is
Algorithm~\ref{algorithm-R-classes-D-class}. Suppose that $S$ is a regular
subsemigroup of $U$, and $x\in S$.  Then, by Corollary~\ref{cor-regular},
$S_x=({}_xS)\Psi$ and so it is no longer necessary to compute $({}_xS)\Psi$ or
the cosets of $S_x\cap
({}_xS)\Psi$ in $({}_xS)\Psi$ in Algorithm~\ref{algorithm-R-classes-D-class}.  If
$U$ is an inverse semigroup with unary operation $^{-1}:x\mapsto x^{-1}$, then
Algorithm~\ref{algorithm-R-classes-D-class} becomes simpler still. In this case,
it is unnecessary to find the s.c.c.~of $(x)\rho$. This follows from the
observation that we may take $(x)\rho=(x^{-1})\lambda$, which implies that
$(S)\rho=(S)\lambda$ and  
$$u^{-1}\cdot (x^{-1})\rho=(u^{-1}x^{-1})\rho=(xu)\lambda=(x)\lambda\cdot u.$$
Similar simplifications can be made in
Algorithm~\ref{algorithm-D-class-membership}. 

The unary operation in the definition of $U$ means that given an inverse
subsemigroup $S$ of $U$ and $x\in S$, that we can find $x ^ {-1}$ without
reference to $S$. Or put differently, the inverse of $x$ is the same in every
inverse subsemigroup of $U$. For example, the symmetric inverse monoid has this
property, but the full transformation monoid does not. More precisely, there
exist distinct inverse subsemigroups $S$ and $T$ of $T_n$ and $x\in S\cap T$
such that the inverse of $x$ in $S$ is distinct from the inverse of $x$ in $T$. 

As noted above, in a regular subsemigroup $S$ of $U$, the $\R$-classes are in
1-1 correspondence with the elements of $(S)\rho$ and the $\L$-classes are in
1-1 correspondence with $(S)\lambda$. It follows that the search for $\R$-class
representatives in Algorithm~\ref{algorithm-enumerate} is redundant in this
case. Hence the $\R$-classes, $\L$-classes, $\H$-classes, size, and elements,
of a regular subsemigroup can be determined from $(S)\lambda$ and $(S)\rho$
using Algorithms~\ref{algorithm-trace-schreier}
and~\ref{algorithm-schreier-generators-special} alone.  The $\D$-classes of a
regular subsemigroup $S$ are then in 1-1 correspondence with the s.c.c.s of
$(S)\lambda$ (or $(S)\rho$). In the case that $U$ is an inverse semigroup, it
suffices to calculate either $(S)\lambda$ or $(S)\rho$, making these
computations simpler still.  The remaining algorithms in
Subsection~\ref{subsection-global} can also be modified to take advantage of
these observations, but due to considerations of space, we do not go
into the details here. 

The simplified algorithms alluded to in this section have been fully
implemented in the \Semigroups package for \GAP; see \cite{Mitchell2016aa}.


\section{Examples}\label{section-examples}

In this section we present some examples to illustrate the algorithms from
the previous section. 

One of the examples is that of a semigroup of partial permutations.  Similar to
permutations, a partial permutation can be expressed as a union of the
components of its action. Any component of the action of a partial permutation
$f$ is either a permutation with a single cycle, or a chain $[i\ (i)f\
(i)f^{2}\ \ldots\ (i)f^r]$ where $i\in \dom(f)\setminus \im(f)$ and $(i)f^r\in
\im(f)\setminus \dom(f)$, for some $r>1$.  For the sake of brevity, we will
use \textit{disjoint component} notation when writing a specific partial
permutation $f$, i.e.~we write $f$ as a juxtaposition of disjoint cycles and
chains. For example,
\begin{equation*}
  \begin{pmatrix}
     1& 2& 3& 4& 5& 6\\ 
     5& 4& -& 2& 6& -
  \end{pmatrix}
  =[1\ 5\ 6](2\ 4).
\end{equation*}
We include fixed points in the disjoint component notation for a partial
permutation $f$ so that it is possible to deduce the domain and image of $f$
from the notation, and so that the notation for $f$ is unique (up to the order
of the components, and the order of elements in a cycle); see
\cite{Lipscomb1996aa} for further details. 

Throughout this section, we will denote by $S$ the submonoid of the
symmetric inverse monoid on $\{1,\ldots, 9\}$ generated by 
\begin{equation}\label{equation-pp-example}
  x_1 = (1\ 4)(2\ 6)(3\ 8)(5)(7)(9), \quad x_2 = (1\ 5\ 4\ 2\ 7\ 6)(3\ 9\ 8)\  \quad
  x_3 = (2\ 5\ 6),                   \quad x_4 = (1\ 3\ 2),
\end{equation}
by $T$ the submonoid of the full transformation monoid on $\{1,\ldots,
5\}$ generated by 
\begin{equation}\label{equation-trans-example}
  x_1 = 
  \begin{pmatrix}
    1 & 2 & 3 & 4 & 5\\
    1 & 3 & 2 & 4 & 5
  \end{pmatrix},\quad
  x_2 = 
  \begin{pmatrix}
    1 & 2 & 3 & 4 & 5\\
    2 & 3 & 1 & 5 & 4
  \end{pmatrix},\quad
  x_3 =
  \begin{pmatrix}
    1 & 2 & 3 & 4 & 5\\
    1 & 3 & 3 & 2 & 2
  \end{pmatrix},
\end{equation} 
and we will use the notation of Sections~\ref{subsection-trans}
and~\ref{subsection-pperm}.

Since both $S$ and $T$ are submonoids, adjoining an identity to either would be
redundant, and so in the interest of saving space, we will not do this.


\subsection*{Components of the action}

Applying Algorithm~\ref{algorithm-component} to $S$ and $\alpha_1=(x_1)\lambda =
\{ 1, \ldots, 9 \}$, we obtain:
\begin{equation*}
  \begin{array}{rllllr}
     (S)\lambda = \{ 
                   & \alpha_1    = \{ 1, \ldots, 9 \}, 
                   & \alpha_2    = \{ 2, 5, 6 \}, 
                   & \alpha_3    = \{ 1, 2, 3 \}, 
                   & \alpha_4    = \{ 1, 4, 7 \}, \\
                   & \alpha_5    = \{ 1 \}, 
                   & \alpha_6    = \{ 4, 6, 8 \}, 
                   & \alpha_7    = \{ 5, 7, 9 \}, 
                   & \alpha_8    = \{ 5 \},       \\
                   & \alpha_9    = \varnothing, 
                   & \alpha_{10} = \{ 3 \}, 
                   & \alpha_{11} = \{ 4 \}, 
                   & \alpha_{12} = \{ 2 \},       \\
                   & \alpha_{13} = \{ 6 \}, 
                   & \alpha_{14} = \{ 8 \}, 
                   & \alpha_{15} = \{ 9 \}, 
                   & \alpha_{16} = \{ 7 \}  &
                 \}.
  \end{array}
\end{equation*}
The Schreier tree is:
\begin{equation*}
  \begin{array}{c|ccccccccccccccc}
    i   & 2 & 3 & 4 & 5 & 6 & 7 & 8 & 9 & 10 & 11 & 12 & 13 & 14 & 15 & 16 \\\hline
    v_i & 1 & 1 & 2 & 2 & 3 & 3 & 3 & 4 & 4  & 5  & 6  & 7  & 10 & 10 & 12 \\
    w_i & 3 & 4 & 2 & 4 & 1 & 2 & 3 & 3 & 4  & 1  & 3  & 3  & 1  & 2  & 2
  \end{array}
\end{equation*}
and the orbit graph is:
\begin{equation*}
  \begin{array}{c|cccccccccccccccc}
    i        & 1 & 2 & 3 & 4  & 5  & 6  & 7  & 8  & 9 & 10 & 11 & 12 & 13 & 14 & 15 & 16 \\ \hline
    g_{i, 1} & 1 & 2 & 6 & 4  & 11 & 3  & 7  & 8  & 9 & 14 & 5  & 13 & 12 & 10 & 15 & 16 \\ 
    g_{i, 2} & 1 & 4 & 7 & 2  & 8  & 3  & 6  & 11 & 9 & 15 & 12 & 16 & 5  & 10 & 14 & 13 \\
    g_{i, 3} & 2 & 2 & 8 & 9  & 9  & 12 & 13 & 13 & 9 & 9  & 9  & 8  & 12 & 9  & 9  & 9  \\
    g_{i, 4} & 3 & 5 & 3 & 10 & 10 & 9  & 9  & 9  & 9 & 12 & 9  & 5  & 9  & 9  & 9  & 9
  \end{array}.
\end{equation*}
Diagrams of the Schreier tree and orbit graphs of $(S)\lambda$ can be found
in Figure~\ref{1-ex-pp-lambda-orb}.
\begin{figure}
  \centering
    \begin{tikzpicture}[
      vertex/.style={circle, draw, fill=white, minimum size=0.65cm}, 
      edge/.style={arrows={-angle 90},thick}]

      \vertex{1}{5}{6} 
      \vertex{2}{1}{5}
      \vertex{3}{9}{5}
      \vertex{4}{1}{4}
      \vertex{5}{2}{2}
      \vertex{6}{8}{4}
      \vertex{7}{10}{4}
      \vertex{8}{5}{2}
      \vertex{9}{1}{3}
      \vertex{10}{5}{4}
      \vertex{11}{3}{0}
      \vertex{12}{8}{2}
      \vertex{13}{7}{0}
      \vertex{14}{6}{5}
      \vertex{15}{4}{5}
      \vertex{16}{9}{0}

      \edgethr{1}{2}{line width=0.35mm}
      \edgefou{1}{3}{line width=0.35mm}

      \edgetwo{2}{4}{bend left, line width=0.35mm}
      \edgefou{2}{5}{out=-150, in=150, line width=0.35mm}

      \edgeone{3}{6}{bend left, line width=0.35mm}
      \edgetwo{3}{7}{line width=0.35mm}
      \edgethr{3}{8}{bend right, line width=0.35mm}

      \edgetwo{4}{2}{bend left, dashed}
      \edgethr{4}{9}{line width=0.35mm}
      \edgefou{4}{10}{line width=0.35mm}

      \edgeone{5}{11}{bend right, line width=0.35mm}
      \edgetwo{5}{8}{bend left, line width=0.35mm, dashed}
      \edgefou{5}{10}{dashed}

      \edgeone{6}{3}{bend left, dashed}
      \edgetwo{6}{3}{dashed}
      \edgethr{6}{12}{line width=0.35mm}

      \edgetwo{7}{6}{dashed}
      \edgethr{7}{13}{bend left}

      \edgetwo{8}{11}{line width=0.35mm, dashed}
      \edgethr{8}{13}{dashed}

      \edgeone{10}{14}{bend left, line width=0.35mm}
      \edgetwo{10}{15}{line width=0.35mm}
      \edgefou{10}{12}{dashed}

      \edgeone{11}{5}{dashed}
      \edgetwo{11}{12}{bend right, dashed}

      \edgeone{12}{13}{bend left, dashed}
      \edgetwo{12}{16}{}
      \edgethr{12}{8}{dashed}
      \edgefou{12}{5}{bend right, dashed}

    \edgeone{13}{12}{bend left, dashed}
    \edgetwo{13}{5}{bend left, dashed}
    \edgethr{13}{12}{dashed}

    \edgeone{14}{10}{bend left, dashed}
    \edgetwo{14}{10}{dashed}

    \edgetwo{15}{14}{dashed}

    \edgetwo{16}{13}{dashed}

  \end{tikzpicture}

  \caption{The orbit graph of $(S)\lambda$ with loops and (all but one of
  the) edges to $\alpha_9$ omitted, and the Schreier tree indicated by solid
  edges.}\label{1-ex-pp-lambda-orb}
\end{figure}
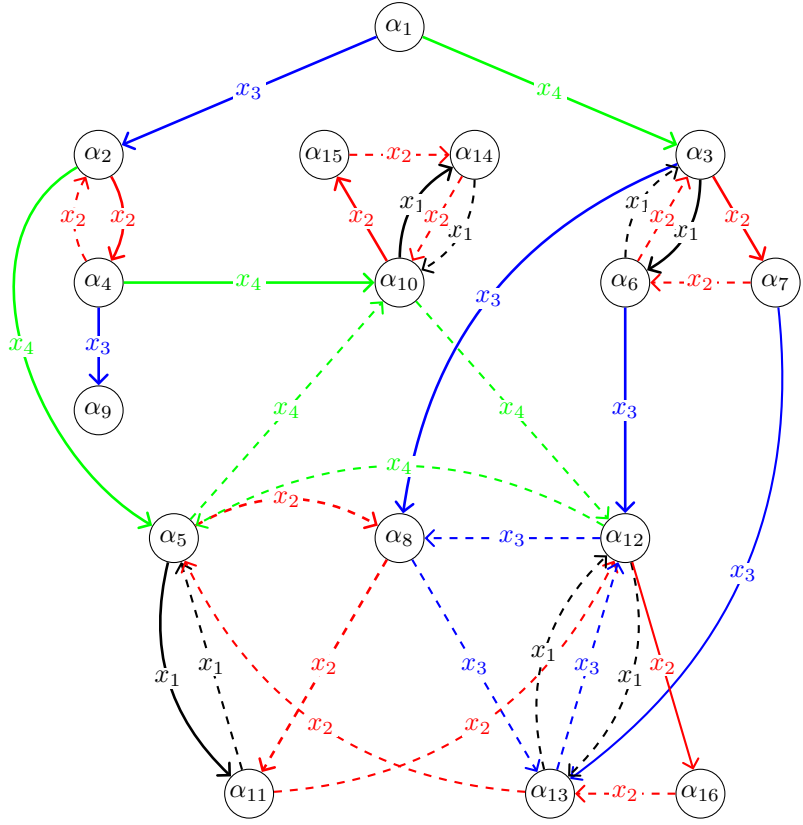

The strongly connected components of $(S)\lambda$ are:
$$\{\{ \alpha_{1} \}, \{ \alpha_{2}, \alpha_{4} \}, \{ \alpha_{3}, \alpha_{6},
\alpha_{7} \}, \{ \alpha_{5}, \alpha_{8}, \alpha_{10}, \alpha_{11}, \alpha_{12},
\alpha_{13}, \alpha_{14}, \alpha_{15}, \alpha_{16} \}, 
\{ \alpha_{9} \}\}.$$

From the Schreier tree, we deduce that 
$$\alpha_1 = (x_1)\lambda,\quad \alpha_2 =
(x_1x_3)\lambda,\quad \alpha_3 = (x_1x_4)\lambda,\quad \alpha_5 =
(x_1x_3x_4)\lambda,\quad\text{and}\quad\alpha_9 = (x_1x_3x_2x_3)\lambda.$$
In this case, we set $\overline{s}=s^{-1}$ in Assumption
\eqref{lambda-inverse}. Using Algorithms~\ref{algorithm-trace-schreier}
and~\ref{algorithm-schreier-generators-special}, after removing redundant
generators, we obtain Schreier generators for the stabilisers of 
$x_1, x_1x_3, x_1x_4, x_1x_3x_4,$ and $x_1x_3x_2x_3$:
\begin{equation*}
  \begin{array}{lcllcl}
    S_{x_1}          & = & \genset{x_1, x_2} \cong D_{12}, & 
    S_{x_1x_3}       & = & \genset{(2\ 6), (2\ 6\ 5)} \cong \sym(\{2,5,6\}),\\
    S_{x_1x_4}       & = & \genset{(1\ 3\ 2), (1\ 2)} \cong \sym(\{1,2,3\}), &
    S_{x_1x_3x_4}    & = & \sym(\{1\})\cong \mathbf{1}, \\
    S_{x_1x_3x_2x_3} & = & \sym(\varnothing) \cong \mathbf{1},
\end{array}
\end{equation*}
where $\mathbf{1}$ denotes the trivial group.


Applying Algorithm~\ref{algorithm-component} to $T$ (defined in
\eqref{equation-trans-example}) and $(x_1)\lambda=\{ 1, 2, 3,
4, 5 \}$, we obtain:
\begin{equation*}
  \begin{array}{rllllr}
      (T)\lambda = \{ 
                   & \alpha_1 = \{ 1, 2, 3, 4, 5 \}, 
                   & \alpha_2 = \{ 1, 2, 3 \}, 
                   & \alpha_3 = \{ 1, 3 \},           
                   & \alpha_4 = \{ 1, 2 \},          \\
                   & \alpha_5 = \{ 2, 3 \}, 
                   & \alpha_6 = \{ 3 \}, 
                   & \alpha_7 = \{ 2 \},             
                   & \alpha_8 = \{ 1 \} &
                 \},
  \end{array}
\end{equation*}
the Schreier tree is:
\begin{equation*}
  \begin{array}{c|ccccccc}
    i   & 2 & 3 & 4 & 5 & 6 & 7 & 8 \\\hline
    v_i & 1 & 2 & 3 & 4 & 5 & 6 & 6 \\
    w_i & 3 & 3 & 1 & 2 & 3 & 1 & 2 
  \end{array}
\end{equation*}
and the orbit graph is:
\begin{equation*}
  \begin{array}{c|cccccccc}
    i        & 1 & 2 & 3 & 4 & 5 & 6 & 7 & 8 \\\hline
    g_{i, 1} & 1 & 2 & 4 & 3 & 5 & 7 & 6 & 8 \\ 
    g_{i, 2} & 1 & 2 & 4 & 5 & 3 & 8 & 6 & 7 \\
    g_{i, 3} & 2 & 3 & 3 & 3 & 6 & 6 & 6 & 8 \\
  \end{array}.
\end{equation*}
Diagrams of the Schreier tree and orbit graphs of $(T)\lambda$ can be found
in Figure~\ref{2-ex-trans-lambda-orb}.



  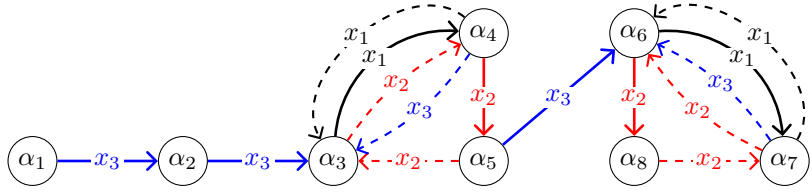
\begin{figure}
    \centering
    \begin{tikzpicture}[
      vertex/.style={circle, draw, fill=white, minimum size=0.65cm}, 
      edge/.style={arrows={-angle 90},thick}]

      \vertex{1}{0}{0} 
      \vertex{2}{2}{0} 
      \vertex{3}{4}{0} 
      \vertex{4}{6}{1} 
      \vertex{5}{6}{0} 
      \vertex{6}{8}{1} 
      \vertex{7}{10}{0} 
      \vertex{8}{8}{0} 
      
      \edgeone{3}{4}{bend left = 45, line width=0.35mm}
      \edgeone{4}{3}{bend right = 80, dashed}
      \edgeone{6}{7}{bend left = 45, line width=0.35mm}
      \edgeone{7}{6}{bend right = 80, dashed}
      
      \edgetwo{3}{4}{bend left = 15, dashed}
      \edgetwo{4}{5}{line width=0.35mm}
      \edgetwo{5}{3}{dashed}
      \edgetwo{6}{8}{line width=0.35mm}
      \edgetwo{7}{6}{bend left = 15, dashed}
      \edgetwo{8}{7}{dashed}

      \edgethr{1}{2}{line width=0.35mm}
      \edgethr{2}{3}{line width=0.35mm}
      \edgethr{4}{3}{bend left = 15, dashed}
      \edgethr{5}{6}{line width=0.35mm}
      \edgethr{7}{6}{bend right = 15, dashed}

    \end{tikzpicture}  
    \caption{The orbit graph of $(T)\lambda$ with loops omitted,
      and the Schreier tree indicated by solid edges.}
    \label{2-ex-trans-lambda-orb}

  \end{figure}

The strongly connected components of $(T)\lambda$ are:
$$\{\{ \alpha_{1} \}, \{ \alpha_{2} \}, \{ \alpha_{3}, \alpha_{4},
\alpha_{5} \}, \{ \alpha_{6}, \alpha_{7}, \alpha_{8} \}\}.$$

From the Schreier tree for $(T)\lambda$ and
Algorithms~\ref{algorithm-trace-schreier}, 
$$\alpha_1 = (x_1)\lambda,\quad \alpha_2 =
(x_1x_3)\lambda,\quad \alpha_3 =
(x_1x_3^2)\lambda,\quad\text{and}\quad\alpha_6 = (x_1x_3^2x_1x_2x_3)\lambda.$$
Using Algorithm~\ref{algorithm-schreier-generators-special}, after removing redundant
generators, we obtain Schreier generators for the stabilisers $T_y$ where 
$y = x_1, x_1x_3, x_1x_3^2, x_1x_3^2x_1x_2x_3$:
\begin{equation*}
  \begin{array}{lcllcl}
  T_{x_1}               & = & \genset{x_1, x_2} \cong D_{12}, & 
  T_{x_1x_3}            & = & \genset{(2\ 3), (1\ 2\ 3)} \cong \sym(\{1,2,3\}),\\
  T_{x_1x_3^2}          & = & \genset{(1\ 3)} \cong \sym(\{1,3\}), &
  T_{x_1x_3^2x_1x_2x_3} & = & \sym(\{3\})\cong \mathbf{1},
\end{array}
\end{equation*}
where $\mathbf{1}$ denotes the trivial group.


\subsection*{Individual Green's classes}

Let $S$ be partial permutation semigroup defined in
\eqref{equation-pp-example}.
The s.c.c.~of $(x_1x_3x_4)\lambda$ contains $9$ values:
$$\{ \alpha_{5}= \{1\}, \alpha_{8}=\{5\}, \alpha_{10}=\{3\},
     \alpha_{11}=\{4\}, \alpha_{12}=\{2\}, \alpha_{13}=\{6\}, 
     \alpha_{14}=\{8\}, \alpha_{15}=\{9\}, \alpha_{16}=\{7\} \},$$
and $|S_{x_1x_3x_4}| = 1$. It follows, by
Corollary~\ref{cor-collect}(b), that the size of the $\R$-class of $x_1x_3x_4$ in $S$
is $9 \cdot 1 = 9$. 

One choice for the Schreier tree (rooted at $\alpha_5 = (x_1x_3x_4)\lambda$)
of the s.c.c.~of $(x_1x_3x_4)\lambda$ is:
\begin{equation*}
  \begin{array}{c|ccccccccc}
    i   & 8  & 10 & 11 & 12 & 13 & 14 & 15 & 16 \\\hline
    v_i & 5  & 5  & 5  & 11 & 8  & 10 & 10 & 12 \\
    w_i & 2  & 4  & 1  & 2  & 3  & 1  & 2  & 2  \\
  \end{array}
\end{equation*}
A diagram of the Schreier tree can be found in
Figure~\ref{3-ex-pp-schreier-tree}.
%
%

  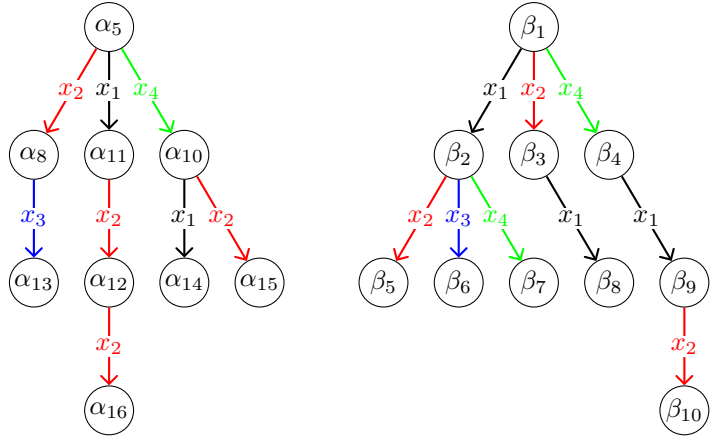
\begin{figure}
    \centering
    \begin{tikzpicture}[
      vertex/.style={circle, draw, fill=white, minimum size=0.65cm}, 
      edge/.style={arrows={-angle 90},thick}]

      \vertex{5}{1}{3} 
      \vertex{8}{0}{2} 
      \vertex{10}{2}{2} 
      \vertex{11}{1}{2} 
      \vertex{13}{0}{1} 
      \vertex{15}{3}{1} 
      \vertex{14}{2}{1} 
      \vertex{16}{1}{0} 
      \vertex{12}{1}{1} 
      
      \edgeone{5}{11}{}
      \edgeone{10}{14}{}
 
      \edgetwo{5}{8}{}
      \edgetwo{10}{15}{}
      \edgetwo{11}{12}{}
      \edgetwo{12}{16}{}

      \edgethr{8}{13}{}
      
      \edgefou{5}{10}{}

    \end{tikzpicture}  
    \qquad
    \renewcommand{\vertex}[3]{\node [vertex] (#1) at (#2, #3 * 1.7) {};
                              \node at (#2, #3 * 1.7) {$\beta_{#1}$};}
    \begin{tikzpicture}[
      vertex/.style={circle, draw, fill=white, minimum size=0.65cm}, 
      edge/.style={arrows={-angle 90},thick}]

      \vertex{1}{2}{3} 
      \vertex{2}{1}{2} 
      \vertex{3}{2}{2} 
      \vertex{4}{3}{2} 
      \vertex{5}{0}{1} 
      \vertex{6}{1}{1} 
      \vertex{7}{2}{1} 
      \vertex{8}{3}{1} 
      \vertex{9}{4}{1} 
      \vertex{10}{4}{0} 
      
      \edgeone{1}{2}{}
      \edgeone{3}{8}{}
      \edgeone{4}{9}{}

      \edgetwo{1}{3}{}
      \edgetwo{2}{5}{}
      \edgetwo{9}{10}{}
      
      \edgethr{2}{6}{}

      \edgefou{2}{7}{}
      \edgefou{1}{4}{}

    \end{tikzpicture}  
    \caption{Schreier trees for the strongly connected component of
    $(x_1x_3x_4)\lambda$ in $(S)\lambda$, and for $(x_1x_3x_4)\rho$ in
    $(S)\rho$.}
    \label{3-ex-pp-schreier-tree}
  \end{figure}

Since $x_1x_3x_4 = [2\ 1]$, using the Schreier tree and 
Algorithm~\ref{algorithm-R-class-elements}, the elements of $R_{x_1x_3x_4}^S$
are:
\begin{equation*}
  \begin{array}{rlcllcllcll}
    R_{x_1x_3x_4}^S = \{ 
                & x_1x_3x_4      & = & [2\ 1], 
                & x_1x_3x_4^3    & = & (2), 
                & x_1x_3x_4^2    & = & [2\ 3], \\ 
                & x_1x_3x_4x_1   & = & [2\ 4],  
                & x_1x_3x_4x_2   & = & [2\ 5], 
                & x_1x_3x_4^3x_1 & = & [2\ 6], \\ 
                & x_1x_3x_4^3x_2 & = & [2\ 7], 
                & x_1x_3x_4^2x_1 & = & [2\ 8], 
                & x_1x_3x_4^2x_2 & = & [2\ 9] 
                & \}, 
  \end{array}
\end{equation*}
Since $x_1x_3x_4^3$ is an idempotent, it also follows that $R_{x_1x_3x_3}^S$ is
a regular $\R$-class. 

We will calculate the $\R$-classes in $D_{x_1x_3x_4}^S$ using
Algorithm~\ref{algorithm-R-classes-D-class}. Since $(x_1x_3x_4)\rho = \{2\}$,
it follows immediately that ${}_{x_1x_3x_4}S=\{\id_{\{2\}}\}$ is trivial.  Set
$x = x_1x_3x_4$ and $x' = x^{-1} = [1\ 2]$. The embedding $\Psi: {}_xS\to
U_x$ (from Proposition~\ref{prop-main-2}(a)) defined by 
$$((s)\nu_x)\Psi=(x'sx)\mu_x$$
maps $\id_{\{2\}}$ to $\id_{\{1\}}$. Hence 
$S_x\cap ({}_xS)\Psi = S_x = \{\id_{\{1\}}\}$.  
It also follows from Proposition~\ref{prop-D-class-R-class-reps} that the
$\R^S$-class representatives of $D_x^S$ are in 1-1 correspondence with the
s.c.c.~of $(x)\rho = \{2\}$.  Using the left analogue of
Algorithm~\ref{algorithm-component}, we obtain 
\begin{equation*}
  \begin{array}{rllllll}
    S\cdot (x)\rho = \{ & \beta_1 = \{2\}, & \beta_2 = \{6\}, & \beta_3 = \{4\}, 
                        & \beta_4 = \{3\}, & \beta_5 = \{7\},\\
                        & \beta_6 = \{5\}, & \beta_7 = \varnothing, & \beta_8 = \{1\},
                        & \beta_9 = \{8\}, & \beta_{10} = \{9\} &\}
  \end{array}
\end{equation*}
with Schreier tree: 
\begin{equation*}
  \begin{array}{c|ccccccccc}
    i   & 2 & 3 & 4 & 5 & 6 & 7 & 8 & 9 & 10 \\\hline
    v_i & 1 & 1 & 1 & 2 & 2 & 2 & 3 & 4 & 9  \\
    w_i & 1 & 2 & 4 & 2 & 3 & 4 & 1 & 1 & 2  \\
  \end{array}
\end{equation*}
A diagram of the Schreier tree can be found in
Figure~\ref{3-ex-pp-schreier-tree}.
Hence the $\R^S$-class representatives of $D_x^S$ are:
\begin{equation*}
  \begin{array}{lcllcllcl}
               x_1x_3x_4 & = & [2\ 1], 
    &        x_1^2x_3x_4 & = & [6\ 1], 
    &       x_2x_1x_3x_4 & = & [4\ 1], \\
            x_4x_1x_3x_4 & = & [3\ 1], 
    &     x_2x_1^2x_3x_4 & = & [7\ 1], 
    &     x_3x_1^2x_3x_4 & = & [5\ 1], \\
         x_1x_2x_1x_3x_4 & = & (1), 
    &    x_1x_4x_1x_3x_4 & = & [8\ 1], 
    & x_2x_1x_4x_1x_3x_4 & = & [9\ 1].
  \end{array}
\end{equation*}
Since the number of $\R^S$-classes in $D_x^S$ is $9$ and each $\R^S$-class has
size $9$, it follows that $|D_x^S|= 81$. 

We will demonstrate how to use Algorithm~\ref{algorithm-R-class-membership} to
check if the partial permutation $y=[1\ 5][2\ 7][3\ 9]$ is $\R^S$-related to
either of the generators $x_3$ and $x_4$ of $S$. Since $(y)\rho = \{1,2,3\}$
and $(x_3)\rho = \{2,5,6\}$, it follows that $(y, x_3)\not\in \R^S$. 
However, $(y)\rho = (x_4)\rho$ and 
$$(y)\lambda = \{5,7,9\} = \alpha_7 \sim \alpha_3 = \{1,2,3\} =
  (x_4)\lambda.$$
Tracing the Schreier tree of $(S)\lambda$ from $\alpha_3$ to $\alpha_7$, we
obtain $u = x_2$ such that $(x_4)\lambda \cdot u = (y)\lambda$. It follows
that $\ov{u} = x_2 ^ {-1}$ has the property that $(y)\lambda \cdot \ov{u} =
(x)\lambda$. Also setting $x_4' = x_4 ^ {-1}$, it follows that $y \in R_x^S$
since
$$(x_4'y\ov{u})\mu_x = (x_4 ^ {-1} y x_2 ^ {-1})\mu_x = (1\ 2\ 3) \in S_{x_4} =
  S_{x_1x_4} = \sym(\{1,2,3\}).$$



\subsection*{The main algorithm}

We now determine the global structure of the transformation semigroup $T$
defined in \eqref{equation-trans-example}. We will do the same thing for the
partial permutation semigroup $S$ defined in \eqref{equation-pp-example} in the
next subsection.

If $x$ is a transformation of degree $n\in \N$, then the kernel $\ker(x)$ of $x$ is a
partition of $\{1,\ldots,n\}$. If the classes of $\ker(x)$ are $A_1, A_2,
\ldots, A_r$, for some $r$, then to avoid writing too many brackets we write 
$\ker(x)=\{A_1|\cdots|A_r\}$.

Applying Algorithm~\ref{algorithm-enumerate} to $T$ defined in
\eqref{equation-trans-example}, we find that
the $\R$-class representatives of $S$ are:
\begin{equation*}
  \begin{array}{rlcllcllcllcllcllclr}
    & y_1    & = & 1_T,               & 
      y_2    & = & x_3,               &     
      y_3    & = & x_2x_3,            &
      y_4    & = & x_3 ^ 2,           & \\
    & y_5    & = & x_1x_2x_3,         &  
      y_6    & = & x_3x_2x_3,         & 
      y_7    & = & x_2x_3 ^ 2,        &
      y_8    & = & x_3x_1x_2x_3,      & \\
    & y_9    & = & (x_2x_3) ^ 2,      &
      y_{10} & = & x_1x_2x_3 ^ 2,     &
      y_{11} & = & x_3 ^ 2 x_1x_2x_3, &
      y_{12} & = & x_1(x_2x_3) ^ 2,   &
  \end{array}
\end{equation*}
and that
 \begin{equation*}
   \begin{array}{rlcllcllcllclr}
     (T)\rho = \{ & (y_{1} )\rho & = &  \{ 1 | 2 | 3 | 4 | 5 \}, &
                    (y_{2} )\rho & = &  \{ 1 | 2, 3 | 4, 5 \} ,   &
                    (y_{3} )\rho & = &  \{ 1, 2 |3 | 4, 5 \} ,    & \\
                  & (y_{4} )\rho & = &  \{ 1 | 2, 3, 4, 5 \},    & 
                    (y_{5} )\rho & = &  \{ 1, 3| 2 |4, 5 \} ,     &
                    (y_{6} )\rho & = &  \{ 1, 4, 5 | 2, 3 \},    & \\
                  & (y_{7} )\rho & = &  \{ 1, 2, 4, 5 | 3 \},    & 
                    (y_{8} )\rho & = &  \{ 1, 2, 3 | 4, 5 \},    &
                    (y_{9} )\rho & = &  \{ 1, 2 | 3, 4, 5 \},    & \\
                  & (y_{10})\rho & = &  \{ 1, 3, 4, 5 | 2 \},    &
                    (y_{11})\rho & = &  \{ 1, 2, 3, 4, 5 \} ,     &
                    (y_{12})\rho & = &  \{ 1, 3 | 2, 4, 5 \}     &
              \}.
   \end{array}
\end{equation*}
The Schreier tree of the orbit graphs of $\mathfrak{R}$ and $(T)\rho$ are both
equal:
\begin{equation*}
  \begin{array}{c|ccccccccccc}
    i   & 2 & 3 & 4 & 5 & 6 & 7 & 8 & 9 & 10 & 11 & 12 \\\hline
    v_i & 1 & 2 & 2 & 3 & 3 & 4 & 5 & 6 & 7  & 8  & 9  \\
    w_i & 3 & 2 & 3 & 1 & 3 & 2 & 3 & 2 & 1  & 3  & 1  
  \end{array}
\end{equation*}
and the orbit graph is:
\begin{equation*}
  \begin{array}{c|cccccccccccc}
    i        & 1 & 2 & 3 & 4 & 5 & 6 & 7  & 8  & 9  & 10 & 11 & 12 \\\hline
    g_{i, 1} & 1 & 2 & 5 & 4 & 3 & 6 & 10 & 8  & 12 & 7  & 11 & 9  \\     
    g_{i, 2} & 1 & 3 & 5 & 7 & 2 & 9 & 10 & 8  & 12 & 4  & 11 & 6  \\
    g_{i, 3} & 2 & 4 & 6 & 4 & 8 & 4 & 6  & 11 & 6  & 8  & 11 & 8
  \end{array}.
\end{equation*}
A diagrams of the Schreier tree and orbit graph can be found
in Figure~\ref{3-ex-trans-rho-orb}. 

\renewcommand{\vertex}[3]{\node [vertex] (#1) at (#2, #3 * 1.7) {};
\node at (#2, #3 * 1.7) {$y_{#1}$};}


  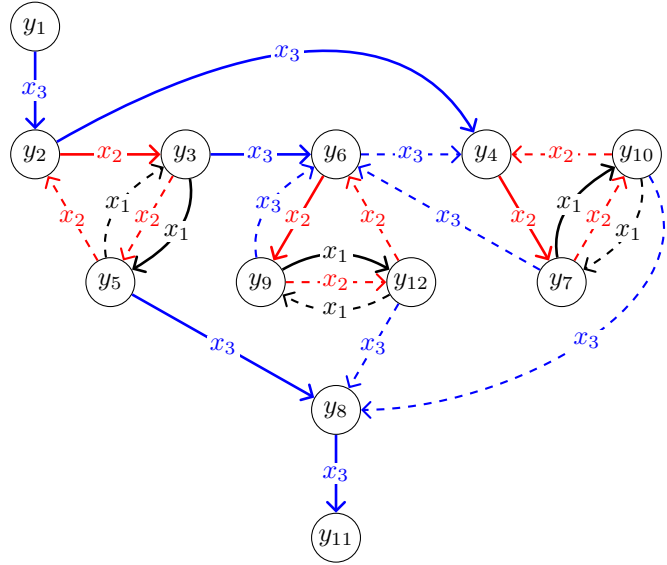
\begin{figure}
    \centering
    \begin{tikzpicture}[
      vertex/.style={circle, draw, fill=white, minimum size=0.65cm}, 
      edge/.style={arrows={-angle 90},thick}]

      \vertex{1}{0}{4} 
      \vertex{2}{0}{3} 
      \vertex{3}{2}{3} 
      \vertex{4}{6}{3} 
      \vertex{5}{1}{2} 
      \vertex{6}{4}{3} 
      \vertex{7}{7}{2} 
      \vertex{8}{4}{1} 
      \vertex{9}{3}{2} 
      \vertex{10}{8}{3} 
      \vertex{11}{4}{0} 
      \vertex{12}{5}{2} 
      
      \edgeone{3}{5}{out=-80, in=30, line width=0.35mm}
      \edgeone{5}{3}{out=100, in=210, dashed}
      \edgeone{7}{10}{out=100, in=210, line width=0.35mm}
      \edgeone{9}{12}{bend left, line width=0.35mm}
      \edgeone{10}{7}{out=-80, in=30, dashed}
      \edgeone{12}{9}{bend left, dashed}
  
      \edgetwo{2}{3}{line width=0.35mm}
      \edgetwo{3}{5}{dashed}
      \edgetwo{4}{7}{line width=0.35mm}
      \edgetwo{5}{2}{dashed}
      \edgetwo{6}{9}{line width=0.35mm}
      \edgetwo{7}{10}{dashed}
      \edgetwo{9}{12}{dashed}
      \edgetwo{10}{4}{dashed}
      \edgetwo{12}{6}{dashed}

      \edgethr{1}{2}{line width=0.35mm}
      \edgethr{2}{4}{out=30, in=120, line width=0.35mm}
      \edgethr{3}{6}{line width=0.35mm}
      \edgethr{5}{8}{line width=0.35mm}
      \edgethr{6}{4}{dashed}
      \edgethr{7}{6}{dashed}
      \edgethr{8}{11}{line width=0.35mm}
      \edgethr{9}{6}{out=100, in=210, dashed}
      \edgethr{10}{8}{out=-60, in=0, dashed}
      \edgethr{12}{8}{dashed}

    \end{tikzpicture}  
    \caption{The orbit graph of $(T)\rho$ and $\mathfrak{R}$ with loops omitted,
    and the Schreier tree indicated by solid edges.}
    \label{3-ex-trans-rho-orb}
  \end{figure}

It is coincidentally the case that the $\R$-class representatives of $T$ are
obtained by left multiplying previous $\R$-class representatives by a
generator. In other words, the $u$ and $\ov{u}$ in line 8 and 9 of
Algorithm~\ref{algorithm-enumerate} are just the identity of $T$ in this
example.  Hence the additional parameters returned by
Algorithm~\ref{algorithm-enumerate} in this case are $(1, 2, 2, 3, 2, 3, 3, 3,
3, 3, 6, 3)$.

Recall that the strongly connected components of $(T)\lambda$ have
representatives:
$\alpha_1 = (x_1)\lambda               = \{1,\ldots, 5\}$, 
$\alpha_2 = (x_1x_3)\lambda            = \{1,2,3\}$,
$\alpha_3 = (x_1x_3^2)\lambda          = \{1,3\}$, 
$\alpha_6 = (x_1x_3^2x_1x_2x_3)\lambda = \{3\}$, 
and sizes: $1$, $1$, $3$, and $3$, respectively. We saw above that 
$|T_{x_1}| = 12$, $|T_{x_1x_3}| = 6$, $|T_{x_1x_3^2}| = 2$, and
$|T_{x_1x_3^2x_1x_2x_3}| = 1$.
It follows from Corollary~\ref{cor-collect}(b) and (c) that
\begin{eqnarray*}
  |T| & = & |R_{x_1} ^ T| + 3 | R_{x_1x_3} ^ T| + 7 | R_{x_1x_3 ^ 2} ^ T | +
            |R_{x_1x_3 ^ 2 x_1x_2x_3} ^ T| \\
      & = & (1\cdot 1\cdot 12) + (3\cdot 1 \cdot 6) + 
            (7 \cdot 3 \cdot 2) + (1 \cdot 3 \cdot 1) = 75.
\end{eqnarray*}
The orbit graph of $\mathfrak{R}$ has 5 strongly connected components:
\begin{equation*}
  \begin{array}{l}
  \{y_1  =  x_1\},\qquad \{y_2  =  x_3, \quad y_3 = x_2x_3, y_5  =  x_1x_2x_3\},\\
  \{y_4  =  x_3^2, y_6  =  x_3x_2x_3, y_7  =  x_2x_3^2, y_9  =  (x_2x_3) ^ 2, 
      y_{10}  =  x_1x_2x_3 ^ 2, y_{12}  =  x_1(x_2x_3) ^ 2\},\\
    \{y_8  =  x_3x_1x_2x_3\},\qquad \{y_{11}  =  x^3x_1x_2x_3\} 
  \end{array}
\end{equation*}
and so there are five $\D$-classes in $T$.



\subsection*{Optimizations for inverse semigroups}

The semigroup $S$ defined in \eqref{equation-pp-example} is an inverse
semigroup, since the inverses of the generators can be obtained by taking
powers. From Section~\ref{section-algorithms-regular}, it follows that the
$\R$-class representatives of $S$ are in 1-1 correspondence with the values in
$(S)\lambda$ and that $(S)\rho = (S)\lambda$. Hence the number of $\R$-classes
in $S$ is $16$, and by tracing the Schreier tree of $(S)\lambda$, the $\R$-class
representatives are:
\begin{equation*}
  \begin{array}{rlcllcllcllcllcllclr}
    & y_1    & = & x_1,               & 
      y_2    & = & x_1x_3,            &     
      y_3    & = & x_1x_4,            &
      y_4    & = & x_1x_3x_2,         & \\
    & y_5    & = & x_1x_3x_4,         &  
      y_6    & = & x_1x_4x_1,         & 
      y_7    & = & x_1x_4x_2,         &
      y_8    & = & x_1x_4x_3,         & \\
    & y_9    & = & x_1x_3x_2x_3,      &
      y_{10} & = & x_1x_3x_2x_4,      &
      y_{11} & = & x_1x_3x_4x_1,      &
      y_{12} & = & x_1x_4x_1x_3,      & \\
    & y_{13} & = & x_1x_4x_2x_3,      &
      y_{14} & = & x_1x_3x_2x_4x_1,   &
      y_{15} & = & x_1x_3x_2x_4x_2,   &
      y_{16} & = & x_1x_4x_1x_3x_2.   
  \end{array}
\end{equation*}
The strongly connected components of $(S)\lambda$ are in 1-1 correspondence
with the $\D$-classes of $S$, and so $S$ has five $\D$-classes. Representatives
of $\L$-classes can be obtained by taking the inverses of the $\R$-class
representatives $\mathfrak{R}$. 

It follows from Corollary~\ref{cor-collect}(c) that
$$|S| = (1 ^ 2 \cdot 12) + (2 ^ 2 \cdot 6) + (3 ^ 2 \cdot 6) + (9 ^ 2 \cdot 1)
+ 1 = 172.$$


\subsection*{Testing membership}

In this subsection, we will use Algorithm~\ref{algorithm-membership} to test if
the following transformations belong to the semigroup $T$ defined in
\eqref{equation-trans-example}:
\begin{equation*}
  x = 
  \begin{pmatrix}
    1 & 2 & 3 & 4 & 5 \\
    1 & 2 & 3 & 3 & 1
  \end{pmatrix}\qquad \text{and}\qquad
  y = 
  \begin{pmatrix}
    1 & 2 & 3 & 4 & 5 \\
    2 & 3 & 3 & 2 & 2
  \end{pmatrix}.
  \quad
\end{equation*}
Although $(x)\lambda = \{1,2,3\} = \alpha_2 \in (T)\lambda$, 
$(x)\rho = \{1,5|2|3,4\} \not\in (T)\rho$, and so $x\not\in T$. 

Firstly, 
$$(y)\lambda = \{2,3\} = \alpha_5 \in (T)\lambda$$
and so the representative of the s.c.c.~of $(y)\lambda$, which we chose above,
is $\alpha_3$. Tracing the Schreier tree for $(T)\lambda$ from $\alpha_5$ back
to $\alpha_3$, using Algorithm~\ref{algorithm-trace-schreier}, we find 
$u = x_1x_2$ such that $\alpha_3 \cdot u = \alpha_5$. Since $x_1$, $x_2$ are
permutations, $\ov{u} = x_2 ^ {-1} x_1 ^ {-1}$ has the property that 
$\alpha_5 \cdot \ov{u} = \alpha_3$ and $yu\ov{u} = y$.
The $\R$-class representative $y_6 \in \mathfrak{R}$ is the only one such that
$(y_6)\lambda = \alpha_3 = (y)\lambda \cdot \ov{u}$ and 
$(y_6)\rho = \{1,4,5|2,3\} = (y)\rho$. Thus to check that $y\in T$, it suffices
to show that the permutation $(y_6'y\ov{u})\mu_x$ belongs to the group
$T_{y_6} = T_{x_1x_3^2} = \sym(\{1,3\})$. We know that $(y_6'y\ov{u})\mu_x$ is
a permutation on $\{1,3\}$, and so it must belong to $T_{y_6}$, and so $y \in
S$. 


\subsection*{Factorization}

In the previous subsection we showed that 
\begin{equation*}
  y = 
  \begin{pmatrix}
    1 & 2 & 3 & 4 & 5 \\
    2 & 3 & 3 & 2 & 2
  \end{pmatrix}
\end{equation*}
is an element of $T$ defined in \eqref{equation-trans-example}. In this
subsection, we will show how to use Algorithm~\ref{algorithm-factorization} to
factorize $y$ as a product of the generators of $T$. Recall that we will write 
$y = xsu$, where $x\in \mathfrak{R}$, $u\in S$ is such that $(x)\lambda\cdot u
= (y)\lambda$, and $s = x'y\ov{u}$, and that we factorise each of $x, s, u$
separately. From the previous subsection, the chosen $\R$-class representative
for $y$ is: 
\begin{equation*}
  x = y_6 = x_3x_2x_3 = 
  \begin{pmatrix}
    1 & 2 & 3 & 4 & 5 \\
    3 & 1 & 1 & 3 & 3
  \end{pmatrix}, 
\end{equation*}
and one choice for $x'\in T_5$ is:
\begin{equation*}
  x' =
  \begin{pmatrix}
    1 & 2 & 3 & 4 & 5 \\
    2 & 2 & 1 & 4 & 5
  \end{pmatrix}.
\end{equation*}
From the previous subsection, $u = x_1x_2$ and $\ov{u} = x_2 ^ {-1} x_1 ^
{-1}$. It follows that 
\begin{equation*}
  s = x'y\ov{u} = 
  \begin{pmatrix}
    1 & 2 & 3 & 4 & 5 \\
    2 & 2 & 1 & 4 & 5
  \end{pmatrix}
  \begin{pmatrix}
    1 & 2 & 3 & 4 & 5 \\
    2 & 3 & 3 & 2 & 2
  \end{pmatrix}
  \begin{pmatrix}
    1 & 2 & 3 & 4 & 5 \\
    2 & 1 & 3 & 4 & 5
  \end{pmatrix}
  = 
  \begin{pmatrix}
    1 & 2 & 3 & 4 & 5 \\
    3 & 3 & 1 & 1 & 1
  \end{pmatrix}
\end{equation*}
and so $(s)\mu_x = s|_{\im(x)} = s|_{\{1,3\}} = (1\ 3)$, which is the only
generator of $T_x$. From Algorithm~\ref{algorithm-schreier-generators-special},
one choice for $s$ such that $s|_{\im(x)} = (1\ 3)$ is $x_3^2 x_1 x_2^2$. Hence
$$y = xsu = x_3x_2x_3\cdot x_3^2 x_1 x_2^2 \cdot x_1x_2 = x_3 x_2 x_3^3 x_1 x_2
^2 x_1 x_2.$$

Note that $x_3x_2x_3x_2^2$ is a minimal length word in the generators that is
equal to $y$. 


\subsection*{The $\D$-class structure}

We showed above that the partial permutation semigroup $S$ defined in
\eqref{equation-pp-example} has five $\D$-classes $D_1$, $D_2$, $D_3$, $D_4$,
and $D_5$ with representatives $x_1$, $x_1x_3$, $x_1x_4$, $x_1x_3x_4$, and
$x_1x_3x_2x_3$, respectively. 
The $\D$-class $D_1$ has only one $\R$-class and one $\L$-class.  If we left
multiply the unique $\R$-class representative $x_1$ of $D_1$ by the generators
of $S$, then we obtain the $\R$-class representatives: 
$$x_1^2 \D^S x_1, \quad x_2x_1 \D^S x_1, \quad x_3x_1 = (2\ 5)(6) \D^S x_1x_3,
  \quad x_4x_1 = [1\ 8][2\ 4][3\ 6] \D^S x_1x_4$$
and so $D_2, D_3\leq_{\J} D_1$.  Note that, since $S$ is inverse, we have not
performed Algorithm~\ref{algorithm-enumerate}, and so we have not previously
left multiplied the $\R$-representatives of $S$ by its generators. 

Right multiplying the unique $\L$-class representative $x_1$ of $D_1$ by the
generators of $S$ we obtain:
$$
x_1^2 \D^S x_1, \quad x_1x_2 \D^S x_1, \quad x_1x_3 = (2)(5\ 6) \D^S x_1x_3,
  \quad x_1x_4 = [4\ 3][6\ 1][8\ 2] \D^S x_1x_4,
$$
which yields no additional information.

Continuing in this way, we obtain the partial order of the $\D$-classes of $S$.
A picture of the egg-box diagrams of the $\D$-classes of $S$ and the partial
order of $\D$-classes of $S$ can be seen in Figure~\ref{figure-D-order}. 
An analogous computation can be used to find the partial order of the
$\D$-classes of the transformation semigroup $T$ and this is
included in Figure~\ref{figure-D-order}.

\begin{figure}
  \begin{center}
    \includegraphics[width=150px]{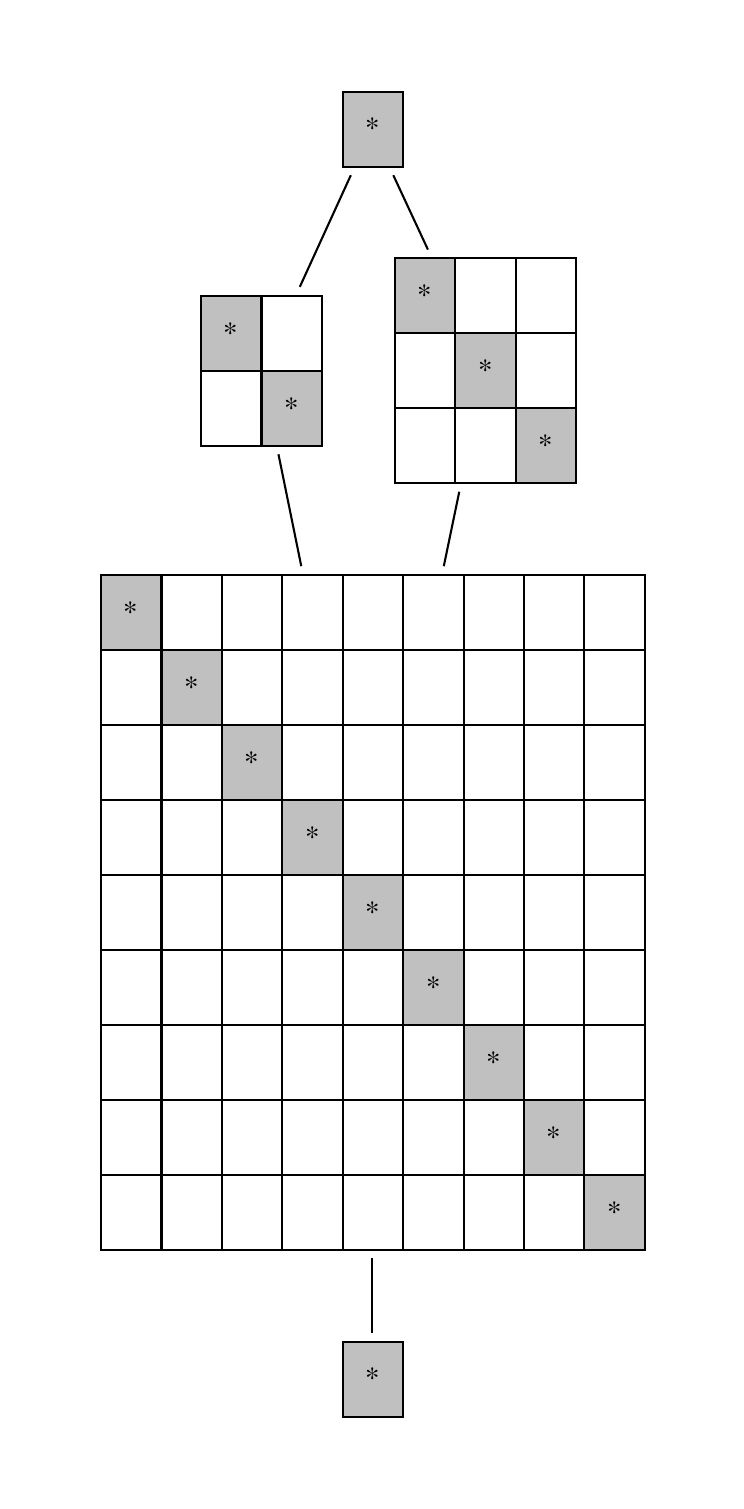}
    \includegraphics[width=95px]{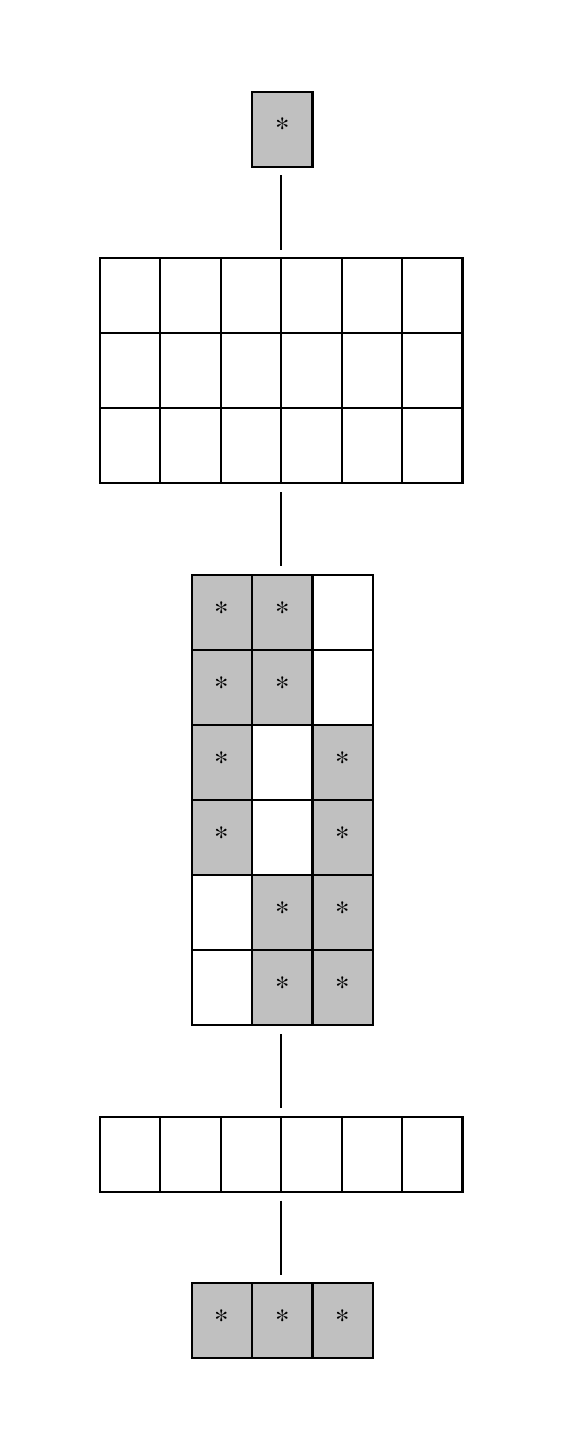}
  \end{center}
  \caption{The partial order of the $\D$-classes of $S$ (left) and $T$ (right),
  group $\H$-classes indicated by shaded boxes.}
  \label{figure-D-order}
\end{figure}


\section{Experimental results}\label{section-benchmarks}
In this section we compare
Algorithm~\ref{algorithm-enumerate}, the optimized variants of
Algorithm~\ref{algorithm-enumerate} for regular and inverse semigroups
(described in Section~\ref{section-algorithms-regular}) when they apply, and the
Froidure-Pin algorithm~\cite{Froidure1997aa} when computing the size of various
standard examples; see Figures~\ref{fig-bench-partition-monoid},
\ref{fig-bench-full-trans}, \ref{fig-bench-symmetric-inv},
\ref{fig-bench-order}, and~\ref{fig-bench-unitriangle}.  
The three variants of Algorithm~\ref{algorithm-enumerate}
used for this comparison are implemented in the \GAP~\cite{GAP4} package
\Semigroups~\cite{Mitchell2016aa}, and the Froidure-Pin algorithm is
implemented in \libsemigroups~\cite{Mitchell2016ab}. The latter is a highly
optimized C++ library, whereas the former is written largely in the high-level
\GAP programming language. If the three variants of
Algorithm~\ref{algorithm-enumerate} were also implemented in C/C++, then a
further improvement in performance could be expected.

For the partition monoid, full transformation monoid, and symmetric
inverse monoid, each of which contain large subgroups, the performance of
Algorithm~\ref{algorithm-enumerate} is considerably better than that of the
Froidure-Pin algorithm. The monoid of order-preserving transformations on
$\{1, \ldots, n\}$ is $\H$-trivial, but far from $\J$-trivial, and
Algorithm~\ref{algorithm-enumerate} is still faster than the Froidure-Pin
algorithm in this case, albeit by a lesser margin. The monoid of unitriangular
boolean matrices is $\J$-trivial, and in Figure~\ref{fig-bench-unitriangle}, we
see that, as expected, the Froidure-Pin algorithm outperforms our algorithms.

All of the computations in this section were performed using a 3.1 GHz Intel
Core i7 processor with 16GB of RAM. 

\begin{figure}
  \centering
  \begin{tabular}{l|r|r|r|r}
    $n$ & $|P_n|$ & Algorithm~\ref{algorithm-enumerate} &
    Algorithm~\ref{algorithm-enumerate} (regular) & Froidure-Pin \\\hline
  4  &         4\ 140 &      11&      7&      6 \\
  5  &       115\ 975 &      30&     21&    174 \\
  6  &      4\ 213\ 597 &     135&     97&   7\ 621 \\
  7  &    190\ 899\ 322 &     817&    576 &   -    \\
  8  &  10\ 480\ 142\ 147 &    5\ 985&   4\ 102 &   -     \\
    9  & 682\ 076\ 806\ 159 &   73\ 390&  29\ 852 & -
  \end{tabular}
    \caption{Time in milliseconds to find the size of the partition monoid $P_n$.}
    \label{fig-bench-partition-monoid}
\end{figure}

\begin{figure}
  \centering
  \begin{tabular}{l|r|r|r|r}
    $n$ & $|T_n|$ & Algorithm~\ref{algorithm-enumerate} &
    Algorithm~\ref{algorithm-enumerate} (regular) & Froidure-Pin \\\hline
    5  &          3\ 125   &          6&      5&      2 \\
    6  &         46\ 656   &         13&      9&     62 \\
    7  &        823\ 543   &         35&     20&    989 \\
    8  &      16\ 777\ 216   &        130&     79&  21\ 430 \\
    9  &     387\ 420\ 489   &        691&    374 &  -    \\
    10 &   10\ 000\ 000\ 000   &        4\ 135&   2\ 359 & -    \\
    11 & 285\ 311\ 670\ 611 & 31\ 699 & 15\ 227 & -
  \end{tabular}
    \caption{Time in milliseconds to find the size of the full transformation
    monoid $T_n$.}
    \label{fig-bench-full-trans}
\end{figure}

\begin{figure}
  \centering
  \begin{tabular}{l|r|r|r|r}
    $n$ & $|I_n|$ & Algorithm~\ref{algorithm-enumerate} &
    Algorithm~\ref{algorithm-enumerate} (inverse) & Froidure-Pin \\\hline
        6&            13\ 327 &   10 &    7 &     10  \\
        7&           130\ 922 &   14 &   10 &    151  \\
        8&          1\ 441\ 729 &   24 &   14 &   2\ 137  \\
        9&         17\ 572\ 114 &   33 &   20 &  28\ 151  \\
       10&        234\ 662\ 231 &   57 &   33  & -      \\
       11&       3\ 405\ 357\ 682 &  110 &   54  & -      \\
       12&      53\ 334\ 454\ 417 &  210 &   97  & -      \\
       13&     896\ 324\ 308\ 634 &  418 &  180  & -      \\
       14&   16\ 083\ 557\ 845\ 279 &  858 &  351  & -      \\
       15&  306\ 827\ 170\ 866\ 106 & 1\ 748 &  710  & - 
  \end{tabular}
    \caption{Time in milliseconds to find the size of the symmetric inverse
    monoid $I_n$.}
    \label{fig-bench-symmetric-inv}
\end{figure}

\begin{figure}
  \centering
  \begin{tabular}{l|r|r|r|r}
    $n$ & $|O_n|$ & Algorithm~\ref{algorithm-enumerate} &
    Algorithm~\ref{algorithm-enumerate} (regular) & Froidure-Pin \\\hline
    6&        462 &      8&      7&      1\\
  7&       1\ 716   &     14&     14&      2\\ 
  8&       6\ 435   &     29&     26&     10\\ 
  9&      24\ 310   &     61&     59&     28\\ 
 10&      92\ 378   &    141&    109&    145\\ 
 11&     352\ 716   &    291&    242&    611\\ 
 12&    1\ 352\ 078   &    631&    523&   2\ 639\\ 
 13&    5\ 200\ 300   &   1\ 380&   1\ 116&  11\ 074\\ 
 14&   20\ 058\ 300   &   3\ 120&   2\ 542&  45\ 532\\ 
    15&   77\ 558\ 760   &   7\ 246&   5\ 649& -    
  \end{tabular}
    \caption{Time in milliseconds to find the size of the monoid $O_n$ of
    order-preserving transformations on $\{1, \ldots, n\}$.}
    \label{fig-bench-order}
\end{figure}

\begin{figure}
  \centering
  \begin{tabular}{l|r|r|r|r}
    $n$ & $|U_n|$ & Algorithm~\ref{algorithm-enumerate} &
   Froidure-Pin \\\hline 
    3 &        8 &        32& 0  \\
 4 &       64    &        593& 1  \\
    5 &     1\ 024    &      -&  215  \\
 6 &    32\ 768    &        - & 7\ 820  
  \end{tabular}
    \caption{Time in milliseconds to find the size of a transformation
    semigroup $U_n$ isomorphic to the monoid of $n\times n$ unitriangular boolean
    matrices.}
    \label{fig-bench-unitriangle}
\end{figure}


\section*{Acknowledgements}
\addcontentsline{toc}{section}{Acknowledgements}
The authors would like to thank Wilf Wilson for his careful reading of
several early drafts of this paper, and for the many improvements he suggested.
We would also like to thank the two anonymous referees for their comments, in
particular, for pointing out the ``identity crisis'' in the submitted version of
the paper. 


\bibliography{citrus}{}
\bibliographystyle{plain}

\end{document}